\documentclass[UTF8]{article}
\usepackage{eurosym}
\usepackage{amsfonts}
\usepackage{amssymb}
\usepackage{amsthm}
\usepackage{amsmath}
\usepackage{graphicx}
\usepackage{empheq}
\usepackage{indentfirst}
\usepackage{cite}
\usepackage{mathrsfs}
\usepackage{cases}
\usepackage{graphics}
\usepackage{xcolor}
\usepackage{bm}
\usepackage{makeidx}
\usepackage[T1]{fontenc}
\usepackage{times}

\newtheorem{theorem}{\textbf{Theorem}}[section]
\newtheorem{lemma}{\textbf{Lemma}}[section]
\newtheorem{proposition}{\textbf{Proposition}}[section]
\newtheorem{corollary}{\textbf{Corollary}}[section]
\newtheorem{remark}{\textbf{Remark}}[section]
\newtheorem{definition}{\textbf{Definition}}[section]

\allowdisplaybreaks[4]

\def\be{\begin{equation}}
	\def\ee{\end{equation}}
\def\bea{\begin{eqnarray}}
	\def\eea{\end{eqnarray}}
\def\bt{\begin{theorem}}
	\def\et{\end{theorem}}
\def\bl{\begin{lemma}}
	\def\el{\end{lemma}}
\def\br{\begin{remark}}
	\def\er{\end{remark}}
\def\bp{\begin{proposition}}
	\def\ep{\end{proposition}}
\def\bc{\begin{corollary}}
	\def\ec{\end{corollary}}
\def\bd{\begin{definition}}
	\def\ed{\end{definition}}

\begin{document}
	
	\title{On the nonlocal Cahn--Hilliard equation with nonlocal dynamic boundary condition and singular potential: 
\\ well-posedness, regularity and asymptotic limits}
	\author{
		Maoyin Lv \thanks{%
			Corresponding author. School of Mathematical Sciences, Fudan University, Shanghai
			200433, P.R. China. Email: \texttt{mylv22@m.fudan.edu.cn} }, \ \
		Hao Wu \thanks{%
			School of Mathematical Sciences, Fudan University, Shanghai 200433, P.R. China. Email: \texttt{haowufd@fudan.edu.cn} } }
	\date{\today }
	\maketitle
	
	\begin{abstract}
		\noindent We consider a class of nonlocal Cahn--Hilliard equations in a bounded domain $\Omega\subset\mathbb{R}^{d}$ $(d\in\{2,3\})$,
subject to a nonlocal kinetic rate dependent dynamic boundary condition. This diffuse interface model describes phase separation processes with possible long-range interactions both within the bulk material and on its boundary.
		The kinetic rate $1/L$, with $L\in[0,+\infty]$, distinguishes different types of bulk-boundary interactions.
		For the initial boundary value problem endowed with general singular potentials, including the physically relevant logarithmic potential, we first establish the existence and uniqueness of global weak solutions when the bulk and boundary chemical potentials are coupled through a Robin-type boundary condition, i.e., $L\in(0,+\infty)$. The proof of existence is based on a Yosida approximation of singular potentials and a suitable Faedo--Galerkin scheme. Subsequently, we investigate asymptotic limits as the kinetic rate approaches zero or infinity, which yield weak well-posedness for the limiting cases $L=+\infty$ and $L=0$, respectively.
		Under appropriate assumptions on the interaction kernels, we derive explicit convergence rates for the Yosida approximation as $\varepsilon\to0$ and for the asymptotic limits as $L\to0$ or $L\to+\infty$.
		Finally, we demonstrate that every global weak solution exhibits a propagation of regularity over time and when $L\in(0,+\infty]$, we establish the instantaneous strict separation property by means of a suitable De Giorgi's iteration scheme.
		
		
	\end{abstract}
	
	
	\section{Introduction}
	The Cahn--Hilliard equation provides a continuous description of the phase segregation process in binary mixtures. It was first proposed in \cite{CH} to model spinodal decomposition of binary alloys and later, successfully extended to different contexts in various fields. To determine the evolution in a bounded domain, suitable boundary conditions must be taken into account. Classical choices are the homogeneous Neumann boundary conditions for the phase variable and chemical potential. The resulting initial boundary value problem has been extensively studied in the literature (see, e.g., \cite{AW,EZ,GGM,KNP,Mi19,MZ04,RH} and the references therein).
	In recent years, the study of boundary effects in the phase separation process of binary mixtures has attracted a lot of attention. Different types of dynamic boundary condition have been proposed to describe short-range interactions between the wall and components of the mixture (cf. \cite{KEMRSBD,FMD,LW,GMS,KLLM}). In terms
of mathematical analysis, the Cahn--Hilliard equation with dynamic boundary conditions has been investigated extensively, see, e.g.,    \cite{CF15,CFS,CFSJEE,GMS09,GK,LvWu,LvWu-2,LvWu-3,KS24-2,KS24,RZ03,CFW} and the references therein. For further information, we refer to the recent review paper \cite{W}.

	In the classical setting, nonlocal contributions are usually ignored in the mixture, that is, only short-range interactions between particles of the interacting materials are considered. To take into account long-range repulsive interactions between different species and short hard collisions between all particles, a nonlocal Cahn--Hilliard equation was rigorously derived in the seminal work \cite{GL96,GL97} through a stochastic argument. This equation serves as a macroscopic limit of microscopic phase segregation models with particle-conserving dynamics.
	Later on, new models consisting of a nonlocal Cahn--Hilliard equation in the bulk subject to a nonlocal dynamic boundary condition have been  introduced to describe both short and long-range interactions of materials in the bulk and on the boundary (see, e.g., \cite{KS,G}).
	In this paper, we focus on the following initial boundary value problem of the nonlocal Cahn--Hilliard equation subject to a class of nonlocal dynamic boundary conditions proposed in \cite{KS}:
	\begin{align}
	&\partial_{t}\varphi=\Delta\mu,&&\text{in }\Omega\times(0,T),\label{ch1}\\
	&\mu=a_{\Omega}\varphi-J\ast\varphi+F'(\varphi),&&\text{in }\Omega\times(0,T),\label{ch2}\\
	&\partial_{t}\psi=\Delta_{\Gamma}\theta-\partial_{\mathbf{n}}\mu,&&\text{on }\Gamma\times(0,T),\label{ch3}\\
	&\theta=a_{\Gamma}\psi-K\circledast\psi+G'(\psi),&&\text{on }\Gamma\times(0,T),\label{ch4}\\
	&\begin{cases}
		L\partial_{\mathbf{n}}\mu=\theta-\mu,&\text{if }L\in[0,+\infty),\\
		\partial_{\mathbf{n}}\mu=0,&\text{if }L=+\infty,
	\end{cases}&&\text{on }\Gamma\times(0,T),\label{ch5}\\
	&\varphi|_{t=0}=\varphi_{0},&&\text{in }\Omega,\label{ch6}\\
	&\psi|_{t=0}=\psi_{0},&&\text{on }\Gamma.\label{ch7}
	\end{align}
	Here, $T\in(0,+\infty)$ is an arbitrary but given final time and $\Omega\subset\mathbb{R}^{d}$ $(d\in\{2,3\})$ is a smooth bounded domain with boundary $\Gamma:=\partial\Omega$. The symbol $\Delta$ denotes the Laplace operator in $\Omega$ and $\Delta_{\Gamma}$ denotes the Laplace--Beltrami operator on $\Gamma$. The bold symbol $\mathbf{n}$ denotes the outward normal vector on the boundary and $\partial_{\mathbf{n}}$ means the outward normal derivative on $\Gamma$. The functions $\varphi:\Omega\times(0,T)\rightarrow\mathbb{R}$ and $\psi:\Gamma\times(0,T)\rightarrow\mathbb{R}$ stand for phase variables describing the difference of two local relative concentrations of materials in the bulk and on the boundary, respectively.  The total free energy functional associated with the system \eqref{ch1}--\eqref{ch5} is defined as
	\begin{align}
	E(\boldsymbol{\varphi}):= E_{\mathrm{bulk}}(\varphi)+E_{\mathrm{surf}}(\psi),\quad \boldsymbol{\varphi}:=(\varphi, \psi),
\label{totalenergy}
	\end{align}
	where the bulk free energy $E_{\text{bulk}}$ and surface free energy $E_{\text{surf}}$ are given by
	\begin{align}
	&E_{\mathrm{bulk}}(\varphi) :=\frac{1}{4}\int_{\Omega}\int_{\Omega}J(x-y)|\varphi(x)-\varphi(y)|^{2} \,\mathrm{d}y\,\mathrm{d}x+\int_{\Omega}F(\varphi(x))\,\mathrm{d}x,\notag\\
	&E_{\mathrm{surf}}(\psi) :=\frac{1}{4}\int_{\Gamma}\int_{\Gamma}K(x-y)|\psi(x)-\psi(y)|^{2} \,\mathrm{d}S_{y}\,\mathrm{d}S_{x}+\int_{\Gamma}G(\psi(x))\,\mathrm{d}S_{x}.\notag
	\end{align}
Then $\mu:\Omega\times(0,T)\rightarrow\mathbb{R}$ and $\theta:\Gamma\times(0,T)\rightarrow\mathbb{R}$ stand for the bulk and boundary chemical potentials, respectively, which can be expressed as Fr\'echet derivatives of the bulk and surface free energies (cf. \eqref{totalenergy}).
The mutual short and long-range interactions between particles are described through convolution integrals weighted by suitable interaction kernels  $J,K:\mathbb{R}^{d}\rightarrow\mathbb{R}$, which are assumed to be even functions, i.e., $J(x)=J(-x)$ and $K(x)=K(-x)$ for all $x\in\mathbb{R}^{d}$. Then the symbols $"\ast"$ in \eqref{ch2} and $"\circledast"$ in \eqref{ch4} denote the convolutions in the bulk and on the boundary, respectively, that is,
	\begin{align*}
	&(J\ast\varphi)(x,t) :=\int_{\Omega}J(x-y)\varphi(y,t)\,\mathrm{d}y,\quad\forall\,(x,t)\in \Omega\times(0,T),\\
	&(K\circledast\psi)(x,t) :=\int_{\Gamma}K(x-y)\psi(y,t)\,\mathrm{d}S_{y},\quad\forall\,(x,t)\in\Gamma\times(0,T).
	\end{align*}
	Moreover, the functions $a_{\Omega}$ and $a_{\Gamma}$ are defined by
	\begin{align*}
	a_{\Omega}(x):=(J\ast1)(x),\qquad a_{\Gamma}(y):=(K\circledast1)(y),
	\end{align*}
	for all $x\in\Omega$ and $y\in\Gamma$.
	The nonlinear potential functions $F$ and $G$ denote free energy densities in the bulk and on the boundary, respectively. In order to describe the phase separation process, $F$ and $G$ usually have a double-well structure and the physically relevant choices include the well-known logarithmic potential \cite{CH}:
	\begin{align}
	&\mathcal{W}_{\mathrm{log}}(s) :=\frac{\Theta}{2}[(1+s)\mathrm{ln}(1+s)+(1-s)\mathrm{ln}(1-s)]-\frac{\Theta_{0}}{2}s^{2},\quad s\in(-1,1).\label{log}
	\end{align}
	When the constants satisfy $\Theta_{0}>\Theta>0$, we find $\mathcal{W}_{\mathrm{log}}$ is nonconvex with two minima $\pm s_*\in (-1,1)$, where $s_*$ is the positive root of the equation $\mathcal{W}_{\mathrm{log}}'(s)=0$.
	 The nonlinearities $F'$ in \eqref{ch2} and $G'$ in \eqref{ch4} simply denote derivatives of the corresponding potentials $F$, $G$. Moreover, when non-smooth potentials are taken into account, $F'$ and $G'$ correspond to the subdifferential of the convex part (may be multi-valued graphs) plus the derivative of the smooth concave perturbations.
	
	In the current model, the bulk and boundary chemical potentials $\mu$, $\theta$ are coupled through the boundary condition \eqref{ch5}, which accounts for possible adsorption or desorption processes between the materials in the bulk and on the boundary, see \cite{KLLM}.
	For the case $L\in(0,+\infty)$, sufficiently regular solutions to problem \eqref{ch1}--\eqref{ch7} satisfy the properties of mass conservation and energy dissipation, that is,
	\begin{align}
	\int_{\Omega}\varphi(t)\,\mathrm{d}x+\int_{\Gamma}\psi(t)\,\mathrm{d}S=\int_{\Omega}\varphi_{0}\,\mathrm{d}x+\int_{\Gamma}\psi_{0}\,\mathrm{d}S,\quad\forall\,t\in[0,T]\label{totalmassconser}
	\end{align}
	as well as
	\begin{align}
	\frac{\mathrm{d}}{\mathrm{d}t}E(\boldsymbol{\varphi}(t)) +\int_{\Omega}|\nabla\mu(t)|^{2}\,\mathrm{d}x +\int_{\Gamma}|\nabla_{\Gamma}\theta(t)|^{2}\,\mathrm{d}S +\frac{1}{L}\int_{\Gamma}|\theta(t)-\mu(t)|^{2}\,\mathrm{d}S=0, \quad\forall\,t\in(0,T).\notag
	\end{align}
	The symbol $\nabla$ denotes the usual gradient operator and $\nabla_{\Gamma}$ denotes the tangential (surface) gradient operator.
	When $L=0$, the model describes the idealized scenario of instantaneous relaxation to the situation where the potentials $\mu$ and $\theta$ are in chemical equilibrium (cf. \cite{KLLM}). Sufficiently regular solutions satisfy the same mass conservation property as \eqref{totalmassconser} and the following energy dissipation law:
	\begin{align}
		\frac{\mathrm{d}}{\mathrm{d}t}E(\boldsymbol{\varphi}(t))+\int_{\Omega}|\nabla\mu(t)|^{2}\,\mathrm{d}x+\int_{\Gamma}|\nabla_{\Gamma}\theta(t)|^{2}\,\mathrm{d}S=0,\quad\forall\,t\in(0,T).\label{0dissipation}
	\end{align}
	When $L=+\infty$, we (formally) obtain the homogeneous Neumann boundary condition for the bulk chemical potential $\mu$. Moreover, the pair $(\varphi,\mu)$ satisfies the standard nonlocal Cahn--Hilliard equation with a homogeneous Neumann boundary condition
	\begin{align}
	&\partial_{t}\varphi=\Delta\mu,&&\text{in }\Omega\times(0,T),\label{infch1}\\
	&\mu=a_{\Omega}\varphi-J\ast\varphi+F'(\varphi),&&\text{in }\Omega\times(0,T),\label{infch2}\\
	&\partial_{\mathbf{n}}\mu=0,&&\text{on }\Gamma\times(0,T),\label{infch3}\\
	&\varphi|_{t=0}=\varphi_{0},&&\text{in }\Omega,\label{infch4}
	\end{align}
	while the pair $(\psi,\theta)$ satisfies the following nonlocal Cahn--Hilliard equation on surface
	\begin{align}
	&\partial_{t}\psi=\Delta_{\Gamma}\theta,&&\text{on }\Gamma\times(0,T),\label{surch1}\\
	&\theta=a_{\Gamma}\psi-K\circledast\psi+G'(\psi),&&\text{on }\Gamma\times(0,T),\label{surch2}\\
	&\psi|_{t=0}=\psi_{0},&&\text{on }\Gamma.\label{surch3}
	\end{align}
As a consequence, the system \eqref{infch1}--\eqref{infch4} and the system \eqref{surch1}--\eqref{surch3} are completely decoupled. Since the mass flux on the boundary $-\partial_{\mathbf{n}}\mu$ equals to zero, we obtain the separate mass conservation properties in the bulk and on the boundary, that is,
	\begin{align}
	\int_{\Omega}\varphi(t)\,\mathrm{d}x=\int_{\Omega}\varphi_{0}\,\mathrm{d}x,\quad\int_{\Gamma}\psi(t)\,\mathrm{d}S=\int_{\Gamma}\psi_{0}\,\mathrm{d}S,\quad\forall\,t\in[0,T].\notag
	\end{align}
	Furthermore, the energy dissipation law \eqref{0dissipation} holds for sufficiently regular solutions.

	The nonlocal Cahn--Hilliard equation \eqref{infch1}--\eqref{infch3} has been extensively studied from various viewpoints. We refer to \cite{ABG,BH05,DRST,GZ,G17,G18,GGG,GP} for results concerning well-posedness and regularity properties of solutions, to \cite{CFG22,CFG,DGG,FGGS,FG,GGGP} for the investigation on the nonlocal Cahn--Hilliard equation coupled to fluid equations, and also to \cite{ABG,FG12,GG,LP} for results on the long-time behavior.
	Convergence of the nonlocal Cahn--Hilliard equation to the local Cahn--Hilliard equation, under suitable assumptions on the convolution kernel $J$, has been investigated in \cite{DRST,DST,DST20,AH,HKP,AT}.
	To the best of our knowledge, there exist two contributions dealing with the nonlocal Cahn--Hilliard equation with dynamic boundary conditions in the literature \cite{G,KS}. In \cite{G}, the author considered the fractional Cahn--Hilliard equation subject to a fractional dynamic boundary condition:
	\begin{align}
		&\partial_{t}\varphi=\Delta\mu,&&\text{in }\Omega\times(0,T),\notag\\
		&\mu=(-\Delta)^{l_1}\varphi+F'(\varphi),&&\text{in }\Omega\times(0,T),\notag\\
		&\psi=\varphi,&&\text{on }\Gamma\times(0,T),\notag\\
		&\partial_{t}\psi=(-\Delta_{\Gamma})^{l_2}\psi+C_{l_1}N^{2-2l_1}\varphi+\varpi\psi+G'(\psi),&&\text{on }\Gamma\times(0,T),\notag\\
		&\varphi|_{t=0}=\varphi_{0},&&\text{in }\Omega,\notag\\
		&\psi|_{t=0}=\psi_0,&&\text{on }\Gamma,\notag
	\end{align}
	with $1/2<l_1<1$, $0<l_2<1$ and $\varpi>0$. Here, $(-\Delta)^{l_1}$ denotes the regional fractional Laplace operator, $(-\Delta_{\Gamma})^{l_2}$ denotes the fractional Laplace--Beltrami operator and $C_{l_1} N^{2-2l_1}$ stands for the fractional normal derivative.
    The author first proved the well-posedness and regularity of solutions, then he established the existence of finite-dimensional global attractors.
	In the second contribution \cite{KS}, the authors formulated the nonlocal problem \eqref{ch1}--\eqref{ch7} in a more general form with an additional boundary penalty term, which describes different regions of the boundary that attract the material associated with $\psi=\pm 1$ and repel the other material associated with $\psi= \mp 1$. They interpreted the bulk-surface Cahn--Hilliard system \eqref{ch1}--\eqref{ch5} as a gradient flow of type $H^{-1}$ (both in the bulk and on the boundary) with respect to suitable inner products.
	Then for regular potentials $F$, $G$ under suitable growth conditions, they established the well-posedness for the case $L\in(0,+\infty)$ using a gradient flow approach. Afterwards they investigated the asymptotic limits as $L\rightarrow0$ and $L\rightarrow+\infty$, respectively.
	

	In this paper, we aim to study the well-posedness, asymptotic limits, regularity and strict separation property of solutions to problem \eqref{ch1}--\eqref{ch7} with a wide class of non-smooth bulk/boundary potentials $F$, $G$ that have a double-well structure. For simplicity of presentation, here we neglect the penalty term in \cite{KS} and set constants that have no influence on the analysis to be one. Extensions of our results to the general case are straightforward.
	\begin{itemize}
	\item[(1)]
	\emph{Well-posedness.} We first prove well-posedness for the case $L\in(0,+\infty)$ (see Theorem \ref{exist}). For the existence of global weak solutions, we replace the singular potentials by their Yosida approximations and solve the resulting auxiliary problem by a suitable Faedo--Galerkin scheme.
	After deriving estimates for the approximating solutions that are uniform with respect to the parameter $\varepsilon$ of the Yosida regularization, we are able to pass to the limit as $\varepsilon\rightarrow0$ and construct a global weak solution by compactness argument.
	The continuous dependence estimate can be achieved by a standard energy method (see Theorem \ref{contidependence}).
	Then we study the well-posedness of the limiting cases with $L=0$ and $L=+\infty$ (see Theorem \ref{exist}). Applying the idea in \cite{LvWu}, we derive uniform estimates with respect to $L\in(0,1]$ and $L\gg 1$, respectively. Then, we establish the existence of weak solutions by performing the asymptotic limits as $L\to0$ and $L\to+\infty$. The corresponding continuous dependence estimates can be obtained similarly to the case $L\in(0,+\infty)$ by the energy method.
	\item[(2)] \emph{Convergence rates for limiting processes.} We establish convergence rates of the Yosida approximation as $\varepsilon\to 0$ and asymptotic limits as $L\to 0$, $L\to +\infty$, under the additional assumption $\mathbf{(A5)}$ for the interaction kernels,
   which requires that $J(x-\cdot)\in L^2(\Omega)$ and $K(y-\cdot)\in L^2(\Gamma)$ for all $x\in\Omega$ and $y\in\Gamma$. Then, the operator
   $$\mathbb{J}:\,\mathcal{L}^2 \to \mathcal{L}^{2},\quad(z,z_\Gamma)\mapsto(J\ast z,K\circledast z_\Gamma)\quad\forall\,(z,z_\Gamma)\in\mathcal{L}^2,$$
   becomes a Hilbert--Schmidt operator.
   Such operators are compact so that, up to a small error, their image is a finite-dimensional subspace of $\mathcal{L}^2$ on which the two norms $\mathcal{L}^2$ and $\mathcal{H}_{L,0}^{-1}$ are equivalent. This property enables us to establish an $O(\sqrt{\varepsilon})$-estimate for the convergence of approximating phase functions in $L^{\infty}(0,T;\mathcal{H}_{L,0}^{-1})\cap L^{2}(0,T;\mathcal{L}^{2})$ (see Theorem \ref{rate}).
   Moreover, we obtain an $O(\sqrt{L})$-estimate for the convergence of weak solutions $\boldsymbol{\varphi}^{L}$ to $\boldsymbol{\varphi}^0$ in $L^{\infty}(0,T;\mathcal{V}_{(0)}^{-1})\cap L^{2}(0,T;\mathcal{L}^{2})$ as $L\to0$ and an $O(1/L^{\frac{1}{4}})$-estimate for the convergence of $\varphi^{L}$ to $\varphi^{\infty}$ in $L^{\infty}(0,T;(H^1(\Omega))')\cap L^{2}(0,T;L^2(\Omega))$ and $\psi^{L}$ to $\psi^\infty$ in $L^{\infty}(0,T;(H^1(\Gamma))')\cap L^{2}(0,T;L^2(\Gamma))$ (see Theorem \ref{weak-convergence}).
	
	\item[(3)] \emph{Regularity propagation and strict separation property.} We first deal with the case with a finite kinetic rate $L\in(0,+\infty)$ and improve the regularity of time derivative $\partial_t \boldsymbol{\varphi}^L$ by taking difference quotient, which allows us to derive an energy equality.
	Then we establish the regularity of global weak solutions by a standard argument like in \cite{GGG}. For the case $L=0$, we obtain such a property via the asymptotic limit as $L\to0$, while for the case $L=+\infty$, we deal with the subsystems in the bulk and on the boundary separately, deriving similar estimates for suitable approximating systems and then passing to the limit as the approximating parameter tends to zero (see Theorem \ref{regularize}).
	Concerning the validity of the instantaneous strict separation property, we extend the methods in \cite{GP,Gior,Po} that worked well for the classical nonlocal Cahn--Hilliard equation with homogeneous Neumann boundary condition. We find that the strict separation property is valid for $L\in(0,+\infty]$ (see Theorem \ref{separation} and Remark \ref{L=infty}), while the case $L=0$ remains open, due to the nonlocal effects in the bulk and on the boundary as well as the bulk-surface interaction through the boundary condition \eqref{ch5} (see Remark \ref{except0}).
	\end{itemize}
	
Before ending the introduction, we recall some related results in the literature. There are several papers studying the convergence of approximating schemes for the nonlocal Cahn--Hilliard equation, see \cite{BMS,GLW,GST} and the references therein.
	In \cite{GST}, the authors studied the standard nonlocal Cahn--Hilliard equation \eqref{infch1}--\eqref{infch4} and established an explicit convergence rate of the Yosida approximation. In their proof, the $L^{2}$-regularity of the kernel $J$ is necessary so that the nonlocal operator $J\ast u$ is a Hilbert--Schmidt operator from $L^2 (\Omega)$ to $L^2 (\Omega)$. Since the usual method to construct weak solutions of the system \eqref{infch1}--\eqref{infch4} is based on the compactness argument and does not provide information on the rate of convergence, the authors filled the gap in this direction.
	When the nontrivial boundary effects are taken into account in the local Cahn--Hilliard equation with dynamic boundary conditions, we established an explicit convergence rate of the Yosida approximation in \cite{LvWu}.
On the other hand, for the local Cahn--Hilliard equation with dynamic boundary conditions, convergence rate of asymptotic limits as the kinetic rate tends to zero or infinity has been obtained in \cite{KLLM} for the case with regular potentials.
For our present problem \eqref{ch1}--\eqref{ch7}, the argument turns out to be more involved, since the nonlocal interactions and boundary effects have to be handled simultaneously.
	
	Concerning the instantaneous strict separation property of the nonlocal Cahn--Hilliard equation \eqref{infch1}--\eqref{infch4}, we refer to \cite{GGG,Po,Gior,GGG23,GP} for recently progress. In particular, such a property is valid for the physically relevant logarithmic potential \eqref{log} in three dimensions.
	As usual, we decompose the singular potential $F$ into a non-smooth convex part $\widehat{\beta}$ and a regular concave part $\widehat{\pi}$, that is, $F=\widehat{\beta}+\widehat{\pi}$.
	In the two-dimensional case, the instantaneous strict separation property was first obtained in \cite{GGG} by the Trudinger--Moser inequality combined with nonlinear estimates involving $e^{|\widehat{\beta}'(\varphi)|}$ of $\widehat{\beta}''(\varphi)$.
	The argument can be applied to more general singular potentials satisfying the growth condition (see \cite{GGG23})
$$\widehat{\beta}''(s)\leq C_{\ast}e^{C_{\ast}|\widehat{\beta}'(s)|^{\gamma}},\quad\text{for some }\gamma\in[1,2).$$
	In \cite{GP}, the authors extended the minimal assumptions on the singular potential, by requiring a mild growth condition of its first derivative near the singular points $\pm1$, without any pointwise additional assumption on its second derivative: there exists some $\kappa>1/2$, such that
	\begin{align*}
    \frac{1}{\widehat{\beta}'(1-2\delta)} =O\Big(\frac{1}{|\mathrm{ln}(\delta)|^{\kappa}}\Big), \quad\frac{1}{|\widehat{\beta}'(-1+2\delta)|} =O\Big(\frac{1}{|\mathrm{ln}(\delta)|^{\kappa}}\Big),\quad
    \text{as}\ \delta\rightarrow0.
	\end{align*}
	In the three-dimensional case, the instantaneous strict separation property has been established in recent works \cite{Po,Gior} by a suitable De Giorgi's iteration scheme.
	There, suitable growth conditions of $\widehat{\beta}'$ and $\widehat{\beta}''$ near the singular points $\pm1$ are required, that is, there exist $\delta_{\sharp}\in(0,1/2)$ and $C_{\sharp}>1$ such that
	\begin{align*}
		\frac{1}{\widehat{\beta}'(1-2\delta)}\leq\frac{C_{\sharp}}{|\text{ln}(\delta)|},\quad\frac{1}{|\widehat{\beta}'(-1+2\delta)|}\leq\frac{C_{\sharp}}{|\text{ln}(\delta)|},\quad\forall\,0<\delta\leq\delta_{\sharp}
	\end{align*}
	and
	\begin{align*}
		\frac{1}{\widehat{\beta}''(1-2\delta)}\leq C_{\sharp}\delta,\quad\frac{1}{\widehat{\beta}''(-1+2\delta)}\leq C_{\sharp}\delta,\quad\forall\,0<\delta\leq\delta_{\sharp}.
	\end{align*}
	Finally, we mention that the strict separation property for the fractional Cahn--Hilliard equation has been analyzed in \cite{GGG23,GP}.

	\textbf{Plan of the paper.} The remaining part of this paper is organized as follows. In Section 2, we first introduce the notation and basic assumptions, then we state our main results. In Section 3, we establish well-posedness of the system with $L\in(0,+\infty)$ and then derive a convergence rate of the Yosida approximation. Well-posedness of the system with $L=0$ or $L=+\infty$ are established in Section 4, together with the convergence rates of the asymptotic limits as $L\to0$ or $L\to+\infty$. In Section 5, we first show the regularity propagation of weak solutions and then validate the instantaneous strict separation property. In the Appendix, we report some technical lemmas that have been used in our analysis.

	\section{Main Results}
	\setcounter{equation}{0}
	\subsection{Preliminaries}
	For any real Banach space $X$, we denote its norm by $\|\cdot\|_X$, its dual space by $X'$
	and the duality pairing between $X'$ and $X$ by
	$\langle\cdot,\cdot\rangle_{X',X}$. If $X$ is a Hilbert space,
	its inner product will be denoted by $(\cdot,\cdot)_X$.
	The space $L^q(0,T;X)$ ($1\leq q\leq +\infty$)
	denotes the set of all strongly measurable $q$-integrable functions with
	values in $X$, or, if $q=+\infty$, essentially bounded functions. The space $L_{\text{uloc}}^p(0,+\infty;X)$ denotes the uniformly local variant of $L^p(0,+\infty;X)$ consisting of all strongly measurable $f:[0,+\infty)\to X$ such that
	$$\|f\|_{L_{\text{uloc}}^p(0,+\infty;X)}:=\sup_{t\geq0}\|f\|_{L^p(t,t+1;X)}<+\infty.$$
	If $T\in(0,+\infty)$, we find $L_{\text{uloc}}^p(0,T;X)=L^p(0,T;X)$.
	The space $C([0,T];X)$ denotes the Banach space of all bounded and
	continuous functions $u:[ 0,T] \rightarrow X$ equipped with the supremum
	norm, while $C_{w}([0,T];X)$ denotes the topological vector space of all
	bounded and weakly continuous functions.
	
	Let $\Omega$ be a bounded domain in $\mathbb{R}^d$ ($d\in \{2,3\}$) with sufficiently smooth boundary $\Gamma:=\partial \Omega$. The associated outward unit normal vector field on $\Gamma$ is denoted by $\mathbf{n}$. We use $|\Omega|$ and $|\Gamma|$
	to denote the Lebesgue measure of $\Omega$ and the Hausdorff measure of $\Gamma$, respectively.
	For any $1\leq q\leq +\infty$, $k\in \mathbb{N}$, the standard Lebesgue and Sobolev spaces on $\Omega$ are denoted by $L^{q}(\Omega )$
	and $W^{k,q}(\Omega)$. Here, we use $\mathbb{N}$ for the set of natural numbers including zero.
	For $s\geq 0$ and $q\in [1,+\infty )$, we denote by $H^{s,q}(\Omega )$ the Bessel-potential spaces and by $W^{s,q}(\Omega )$ the Slobodeckij spaces.
	If $q=2$, it holds $H^{s,2}(\Omega)=W^{s,2}(\Omega )$ for all $s$ and these spaces are Hilbert spaces.
	For simplicity, we use the notation $H^s(\Omega)=H^{s,2}(\Omega)=W^{s,2}(\Omega )$ and $H^0(\Omega)$ can be identified with $L^2(\Omega)$.
	The Lebesgue spaces, Sobolev spaces and Slobodeckij spaces on the boundary $\Gamma$ can be defined analogously,
	provided that $\Gamma$ is sufficiently regular.
	We then write $H^s(\Gamma)=H^{s,2}(\Gamma)=W^{s,2}(\Gamma)$ and identify $H^0(\Gamma)$ with $L^2(\Gamma)$.
	Hereafter, the following shortcuts will be applied:
	\begin{align*}
		&H:=L^{2}(\Omega),\quad H_{\Gamma}:=L^{2}(\Gamma),\quad V:=H^{1}(\Omega),\quad V_{\Gamma}:=H^{1}(\Gamma).
	\end{align*}
	For every $y\in (H^1(\Omega))'$, we denote by $%
	\langle y\rangle_\Omega=|\Omega|^{-1}\langle
	y,1\rangle_{(H^1(\Omega))',\,H^1(\Omega)}$ its generalized mean
	value over $\Omega$. If $y\in L^1(\Omega)$, then its spatial mean is simply
	given by $\langle y\rangle_\Omega=|\Omega|^{-1}\int_\Omega y \,\mathrm{d}x$.
	The spatial mean for a function $y_\Gamma$ on $\Gamma$, denoted by $\langle
	y_\Gamma\rangle_\Gamma$, can be defined in a similar manner.
	Then we introduce the spaces for functions with zero mean:
	\begin{align*}
		&V_{0}:=\big\{y\in V:\;\langle y\rangle_\Omega =0\big\},&& V_{0}^*:=\big\{y^*\in V':\;\langle y^*\rangle_\Omega =0\big\},
		\\[1mm]
		&V_{\Gamma,0}:=\big\{y_{\Gamma}\in V_{\Gamma}:\; \langle y_\Gamma\rangle_\Gamma =0\big\},&& V_{\Gamma,0}^*:=\big\{y_\Gamma^* \in V_{\Gamma}':\;\langle y_\Gamma^*\rangle_\Gamma =0\big\}.
	\end{align*}
	The following Poincar\'{e}--Wirtinger inequalities in $\Omega$ and on $\Gamma$ hold (see, e.g., \cite[Theorem 2.12]{DE13} for the case on $\Gamma$):
	\begin{align}
		&\|u-\langle u\rangle_\Omega\|_{H} \leq C_\Omega\|\nabla u\|_{H},\qquad\quad\quad \ \ \ \forall\,
		u\in V, \notag\\[1mm]
		&\|u_\Gamma-\langle u_\Gamma\rangle_\Gamma\|_{H_\Gamma} \leq C_\Gamma\|\nabla_\Gamma u_\Gamma\|_{H_\Gamma},\qquad \forall\,
		u_\Gamma\in V_\Gamma,
		\notag
	\end{align}
	where $C_\Omega$ (resp. $C_\Gamma$) is a positive constant depending only on $\Omega$ (resp. $\Gamma$).
	
	Consider the Neumann problem
	\begin{align}
		\begin{cases}
			-\Delta u=y, \,\quad \text{in} \ \Omega,\\
			\partial_{\mathbf{n}} u=0, \quad \ \ \, \text{on}\ \Gamma.
		\end{cases}
		\label{P-n}
	\end{align}
	Owing to the Lax--Milgram theorem, for every $y\in V_0^*$, problem \eqref{P-n} admits a unique weak solution $u\in V_{0}$ satisfying
	\begin{align}
		\int_{\Omega}\nabla u\cdot\nabla \zeta\,\mathrm{d}x
		= \langle y,\zeta \rangle_{V',V},\quad\forall\,\zeta\in V.
		\notag
	\end{align}
	Thus, we can define the solution operator $\mathcal{N}_{\Omega}: V_{0}^*\rightarrow V_{0}$ such that $u=\mathcal{N}_{\Omega}y$.
	Analogously, let us consider the surface Poisson equation
	\begin{align}
		-\Delta_\Gamma u_\Gamma = y_\Gamma,\quad \text{on}\ \ \Gamma,
		\label{P-s}
	\end{align}
	where $\Delta_\Gamma=\mathrm{div}_\Gamma\nabla_\Gamma$ denotes the Laplace--Beltrami operator on $\Gamma$.
	For every $y_\Gamma\in V_{\Gamma,0}^*$, problem \eqref{P-s} admits a unique weak solution $u_\Gamma\in V_{\Gamma,0}$ satisfying
	\begin{align}
		\int_{\Gamma}\nabla_{\Gamma} u_{\Gamma}\cdot\nabla_{\Gamma} \zeta_{\Gamma}\,\mathrm{d}S
		= \langle y_{\Gamma},\zeta_{\Gamma} \rangle_{V_{\Gamma}',V_{\Gamma}},\quad\forall\,\zeta_{\Gamma}\in V_{\Gamma}.
		\notag
	\end{align}
	Thus, we can define the solution operator $\mathcal{N}_{\Gamma}:V_{\Gamma,0}^*\rightarrow V_{\Gamma,0}$ such that $u_{\Gamma}=\mathcal{N}_{\Gamma}y_{\Gamma}$.
	By virtue of these definitions, we introduce the following equivalent norms
	\begin{align*}
		&\Vert y\Vert_{V_{0}^*}:=\Big(\int_{\Omega}|\nabla\mathcal{N}_{\Omega}y|^{2}\,\mathrm{d}x\Big)^{1/2},
		&& \forall\, y\in V_{0}^*,\\
		&\Vert y\Vert_{V'}:=\Big(\Vert y - \langle y\rangle_\Omega \Vert_{V_{0}^*}^2+ |\langle y\rangle_\Omega|^2\Big)^{1/2},
		&& \forall\, y\in V',\\
		&\Vert y_{\Gamma}\Vert_{V_{\Gamma,0}^*}:=\Big(\int_{\Gamma}|\nabla_{\Gamma}\mathcal{N}_{\Gamma}y_{\Gamma}|^{2}\,\mathrm{d}S\Big)^{1/2},
		&& \forall\,y_{\Gamma}\in V_{\Gamma,0}^*,\\
		&\Vert y_{\Gamma}\Vert_{V_{\Gamma}'}:=\Big(\Vert y_{\Gamma}- \langle y_{\Gamma}\rangle_\Gamma \Vert_{V_{\Gamma,0}^*}^2 + |\langle y_{\Gamma}\rangle_\Gamma|^2\Big)^{1/2},
		&& \forall\,y_{\Gamma}\in V_{\Gamma}'.
	\end{align*}

	Next, let us define the product spaces
	$$\mathcal{L}^{q}:=L^{q}(\Omega)\times L^{q}(\Gamma)\quad\mathrm{and}\quad\mathcal{H}^{k}:=H^{k}(\Omega)\times H^{k}(\Gamma),$$
	for $q\in [1,+\infty]$ and $k\in \mathbb{N}$.
	Like before, we can identify  $\mathcal{H}^{0}$ with $\mathcal{L}^{2}$.
	For any $k\in \mathbb{N}$, $\mathcal{H}^{k}$ is a Hilbert space endowed with the standard inner product
	$$
	((y,y_{\Gamma}),(z,z_{\Gamma}))_{\mathcal{H}^{k}}:=(y,z)_{H^{k}(\Omega)}+(y_{\Gamma},z_{\Gamma})_{H^{k}(\Gamma)},\quad\forall\, (y,y_{\Gamma}), (z,z_{\Gamma})\in\mathcal{H}^{k}
	$$
	and the induced norm $\Vert\cdot\Vert_{\mathcal{H}^{k}}:=(\cdot,\cdot)_{\mathcal{H}^{k}}^{1/2}$.
	We introduce the duality pairing
	\begin{align*}
		\langle (y,y_\Gamma),(\zeta, \zeta_\Gamma)\rangle_{(\mathcal{H}^1)',\mathcal{H}^1}
		= (y,\zeta)_{L^2(\Omega)}+ (y_\Gamma, \zeta_\Gamma)_{L^2(\Gamma)},
		\quad \forall\, (y,y_\Gamma)\in \mathcal{L}^2,\ (\zeta, \zeta_\Gamma)\in \mathcal{H}^1.
	\end{align*}
	By the Riesz representation theorem, this product
	can be extended to a duality pairing on $(\mathcal{H}^1)'\times \mathcal{H}^1$.
	
	For any $k\in\mathbb{Z}^+$, we consider the Hilbert space
	$$
	\mathcal{V}^{k}:=\big\{(y,y_{\Gamma})\in\mathcal{H}^{k}\;:\;y|_{\Gamma}=y_{\Gamma}\ \ \text{a.e. on }\Gamma\big\},
	$$
	endowed with the inner product $(\cdot,\cdot)_{\mathcal{V}^{k}}:=(\cdot,\cdot)_{\mathcal{H}^{k}}$ and the associated norm $\Vert\cdot\Vert_{\mathcal{V}^{k}}:=\Vert\cdot\Vert_{\mathcal{H}^{k}}$.
	Here, $y|_{\Gamma}$ stands for the trace of $y\in H^k(\Omega)$ on the boundary $\Gamma$, which makes sense for $k\in \mathbb{Z}^+$.
	The duality pairing on $(\mathcal{V}^1)'\times \mathcal{V}^1$ can be defined in a similar manner.
	
	For any given $m\in\mathbb{R}$, we set
	$$
	\mathcal{L}^{2}_{(m)}:=\big\{(y,y_{\Gamma})\in\mathcal{L}^{2}\;:\;\overline{m}(y,y_{\Gamma})=m\big\},
	$$
	where the generalized mean is defined by
	\begin{align}
		\overline{m}(y,y_{\Gamma}):=\frac{|\Omega|\langle y\rangle_{\Omega}+|\Gamma|\langle y_{\Gamma}\rangle_{\Gamma}}{|\Omega|+|\Gamma|}.\label{gmean}
	\end{align}
    Define the projection operator $\mathbf{P}:\mathcal{L}^{2}\rightarrow\mathcal{L}^{2}_{(0)}$ by
	\begin{align*}
		\mathbf{P}(y,y_{\Gamma})=(y-\overline{m}(y,y_{\Gamma}),y_{\Gamma}-\overline{m}(y,y_{\Gamma})),\quad\forall\,(y,y_\Gamma)\in \mathcal{L}^2.
	\end{align*}
	The closed linear subspaces
	$$
	\mathcal{H}_{(0)}^k=\mathcal{H}^{k}\cap\mathcal{L}_{(0)}^{2},
	\qquad \mathcal{V}_{(0)}^k=\mathcal{V}^{k}\cap\mathcal{L}_{(0)}^{2},
	\qquad k\in \mathbb{Z}^+,
	$$
	are Hilbert spaces endowed with the inner products $(\cdot,\cdot)_{\mathcal{H}^{k}}$
	and the associated norms $\Vert\cdot\Vert_{\mathcal{H}^{k}}$, respectively.
	For $L\in [0,+\infty)$ and $k\in \mathbb{Z}^+$,  we introduce the notation
	$$
	\mathcal{H}^{k}_{L}:=
	\begin{cases}
		\mathcal{H}^k,\quad \text{if}\ L\in (0,+\infty),\\[1mm]
		\mathcal{V}^{k},\quad \ \text{if}\ L=0,
	\end{cases}\qquad
	\mathcal{H}^{k}_{L,0}:=
	\begin{cases}
		\mathcal{H}_{(0)}^k,\quad \text{if}\ L\in (0,+\infty),\\[1mm]
		\mathcal{V}^{k}_{(0)},\quad \ \text{if}\ L=0.
	\end{cases}
	$$
	Consider the bilinear form
	\begin{align}
		a_{L}((y,y_{\Gamma}),(z,z_{\Gamma})) :=\int_{\Omega}\nabla y\cdot\nabla z \,\mathrm{d}x +\int_{\Gamma}\nabla_{\Gamma}y_{\Gamma}\cdot\nabla_{\Gamma}z_{\Gamma}\,\mathrm{d}S
		+\chi(L)\int_{\Gamma}(y-y_{\Gamma})(z-z_{\Gamma})\,\mathrm{d}S,
		\notag
	\end{align}
	for all $(y,y_{\Gamma}), (z,z_{\Gamma})\in \mathcal{H}^{1}$, where
	$$
	\chi(L)=\begin{cases}
		1/L,\quad \text{if}\ L\in (0,+\infty),\\[1mm]
		0,\qquad \ \text{if}\ L=0, +\infty.
	\end{cases}
	$$
	For any $(y,y_{\Gamma})\in \mathcal{H}^{1}_{L,0}$, $L\in [0,+\infty)$, we define
	\begin{align}
		\Vert(y,y_{\Gamma})\Vert_{\mathcal{H}^{1}_{L,0}}:=((y,y_{\Gamma}),(y,y_{\Gamma}))_{\mathcal{H}^{1}_{L,0}}^{1/2}= [ a_{L}((y,y_{\Gamma}),(y,y_{\Gamma}))]^{1/2}.
		\label{norm-hL}
	\end{align}
For $(y,y_{\Gamma})\in \mathcal{V}_{(0)}^1\subseteq\mathcal{H}^{1}_{L,0}$, $\Vert(y,y_{\Gamma})\Vert_{\mathcal{H}^{1}_{L,0}}$ does not depend on $L$, since the third term in $a_L$ simply vanishes.
	The following Poincar\'{e} type inequality has been proved in \cite[Lemma A.1]{KL}:
	\begin{lemma}
		There exists a constant $c_P>0$ depending only on $L\in [0,+\infty)$ and $\Omega$ such that
		\begin{align}
			\|(y,y_{\Gamma})\|_{\mathcal{L}^2}\leq c_P \Vert(y,y_{\Gamma})\Vert_{\mathcal{H}^{1}_{L,0}},\quad \forall\, (y,y_{\Gamma})\in \mathcal{H}^{1}_{L,0}.
			\label{Po3}
		\end{align}
	\end{lemma}
	\noindent
	Hence, for every $L\in [0,+\infty)$, $\mathcal{H}^{1}_{L,0}$ is a Hilbert space with the inner product $(\cdot,\cdot)_{\mathcal{H}^{1}_{L,0}}^{1/2}$. The induced norm $\Vert\cdot\Vert_{\mathcal{H}^{1}_{L,0}}$ prescribed in \eqref{norm-hL} is equivalent to the standard one $\Vert\cdot\Vert_{\mathcal{H}^{1}}$ on $\mathcal{H}^{1}_{L,0}$.

	For $L\in[0,+\infty)$, let us consider the following elliptic boundary value problem
	\begin{align}
		\begin{cases}
			-\Delta u = y,&\quad \text{in }\Omega,\\
			-\Delta_{\Gamma}u_\Gamma +\partial_{\mathbf{n}}u= y_{\Gamma},&\quad \text{on }\Gamma,\\
			L\partial_{\mathbf{n}}u =u_\Gamma-u|_{\Gamma}, &\quad \mathrm{on\;}\Gamma.
	\end{cases}
		\label{ell2.2}
	\end{align}
	Define the space
	$$
	\mathcal{H}_{L,0}^{-1}=
	\begin{cases}
		\mathcal{H}_{(0)}^{-1} :=\big\{(y,y_{\Gamma})\in(\mathcal{H}^{1})'\;:\;\overline{m}(y,y_{\Gamma})=0\big\},\quad \text{if}\ L\in(0,+\infty),\\[1mm]
		\mathcal{V}_{(0)}^{-1} :=\big\{(y,y_{\Gamma})\in(\mathcal{V}^{1})'\;:\;\overline{m}(y,y_{\Gamma})=0\big\},\quad \ \,\text{if}\ L=0,
	\end{cases}
	$$
	where $\overline{m}$ is given by \eqref{gmean} if $L\in (0,+\infty)$, and for $L=0$, we take
	$$
	\overline{m}(y,y_{\Gamma})=\frac{\langle (y,y_{\Gamma}),(1,1)\rangle_{(\mathcal{V}^{1})', \mathcal{V}^{1}}}{|\Omega|+|\Gamma|}.
	$$
Then the chain of inclusions holds
$$
\mathcal{H}^{1}_{L,0} \subset \mathcal{L}_{(0)}^2\subset \mathcal{H}_{L,0}^{-1} \subset (\mathcal{H}_{L}^{1})'\subset (\mathcal{H}^{1}_{L,0})'.
$$
	It has been shown in \cite[Theorem 3.3]{KL} that for every $(y,y_{\Gamma})\in\mathcal{H}_{L,0}^{-1}$, problem \eqref{ell2.2} admits a unique weak solution $(u,u_\Gamma)\in\mathcal{H}_{L,0}^{1}$ satisfying the weak formulation
	\begin{align}
		a_L((u,u_{\Gamma}),(\zeta,\zeta_{\Gamma})) = \langle(y,y_{\Gamma}),(\zeta,\zeta_{\Gamma})\rangle_{(\mathcal{H}_L^{1})',\mathcal{H}_L^{1}},
		\quad \forall\, (\zeta,\zeta_{\Gamma})\in\mathcal{H}_L^{1},
		\notag
	\end{align}
	and the $\mathcal{H}^{1}$-estimate
	\begin{align}
		\|(u,u_{\Gamma})\|_{\mathcal{H}^1}\leq C\|(y,y_{\Gamma})\|_{(\mathcal{H}^{1}_L)'},\notag
	\end{align}
	for some constant $C>0$ depending only on $L$ and $\Omega$. Furthermore, if the domain $\Omega$ is of class $C^{k+2}$ and $(y,y_{\Gamma})\in \mathcal{H}_{L,0}^{k}$, $k\in \mathbb{N}$, the following regularity estimate holds
	\begin{align}
		\|(u,u_{\Gamma})\|_{\mathcal{H}^{k+2}}\leq C\|(y,y_{\Gamma})\|_{\mathcal{H}^{k}}.
		\notag
	\end{align}
	The above facts enable us to define the solution operator $$\mathfrak{S}^{L}:\mathcal{H}_{L,0}^{-1}\rightarrow\mathcal{H}_{L,0}^{1},\quad(y,y_{\Gamma})\mapsto (u,u_\Gamma)=\mathfrak{S}^{L}(y,y_{\Gamma})=(\mathfrak{S}^{L}_{\Omega}(y,y_{\Gamma}),\mathfrak{S}^{L}_{\Gamma}(y,y_{\Gamma})).
	$$
	We mention that similar results for the special case $L=0$ have been presented in \cite{CF15}. A direct calculation yields that
	$$
	((u,u_{\Gamma}), (z,z_\Gamma))_{\mathcal{L}^2}
	=((u,u_{\Gamma}), \mathfrak{S}^{L}(z,z_\Gamma))_{\mathcal{H}^1_{L,0}},\quad \forall\, (u,u_{\Gamma})\in \mathcal{H}_{L,0}^1,\ (z,z_\Gamma)\in \mathcal{L}^2_{(0)}.
	$$
	Thanks to \cite[Corollary 3.5]{KL}, we can introduce the inner product on $\mathcal{H}_{L,0}^{-1}$  as
	\begin{align}
		((y,y_{\Gamma}),(z,z_{\Gamma}))_{L,0,*}&:=(\mathfrak{S}^{L}(y,y_{\Gamma}),\mathfrak{S}^{L}(z,z_{\Gamma}))_{\mathcal{H}^{1}_{L,0}},
		\quad \forall\, (y,y_{\Gamma}), (z,z_{\Gamma})\in \mathcal{H}_{L,0}^{-1}.\notag
	\end{align}
	The associated norm $\Vert(y,y_{\Gamma})\Vert_{L,0,*} :=((y,y_{\Gamma}),(y,y_{\Gamma}))_{L,0,*}^{1/2}$
	is equivalent to the standard dual norm $\|\cdot\|_{(\mathcal{H}_L^1)'}$ on $\mathcal{H}_{L,0}^{-1}$.
	Then it follows that
	\begin{align}
		\|(y,y_{\Gamma})\|_{L,\ast}&:=(\Vert(y,y_{\Gamma})-\overline{m}(y,y_{\Gamma}) \mathbf{1}\Vert_{L,0,*}^2+ |\overline{m}(y,y_{\Gamma})|^2)^{1/2},
		\quad \forall\, (y,y_{\Gamma})\in (\mathcal{H}_L^{1})',\notag
	\end{align}
	is equivalent to the usual dual norm $\|\cdot\|_{(\mathcal{H}_L^1)'}$ on $(\mathcal{H}_L^{1})'$.

	\subsection{Problem setting}
	For any arbitrary but given final time $T\in(0,+\infty)$, we set $$Q_{T}:=\Omega\times(0,T),\quad \Sigma_{T}:=\Gamma\times(0,T).$$ If $T=+\infty$, we simply set $Q:=\Omega\times(0,+\infty)$ and $\Sigma:=\Gamma\times(0,+\infty)$.
	In view of the decomposition for the bulk and surface potentials
	$$F=\widehat{\beta}+\widehat{\pi},\quad G=\widehat{\beta}_{\Gamma}+\widehat{\pi}_{\Gamma},$$
	we reformulate our target problem \eqref{ch1}--\eqref{ch7} as follows:
	\begin{align}
		&\partial_{t}\varphi=\Delta\mu,&&\text{in }Q_{T},\label{model1}\\
		&\mu=a_{\Omega}\varphi-J\ast\varphi+\beta(\varphi)+\pi(\varphi),&&\text{in }Q_{T},\label{model2}\\
		&\partial_{t}\psi=\Delta_{\Gamma}\theta-\partial_{\mathbf{n}}\mu,&&\text{on }\Sigma_{T},\label{model3}\\
		&\theta=a_{\Gamma}\psi-K\circledast\psi+\beta_{\Gamma}(\psi)+\pi_{\Gamma}(\psi),&&\text{on }\Sigma_{T},\label{model4}\\
		&\begin{cases}
			L\partial_{\mathbf{n}}\mu=\theta-\mu,&\text{if }L\in[0,+\infty),\\
			\partial_{\mathbf{n}}\mu=0,&\text{if }L=+\infty,
		\end{cases}
		&&\text{on }\Sigma_{T},\label{model5}\\
		&\varphi|_{t=0}=\varphi_{0},&&\text{in }\Omega,\label{phiini}\\
		&\psi|_{t=0}=\psi_{0},&&\text{on }\Gamma.\label{psiini}
	\end{align}
 In this setting, the total free energy of system \eqref{model1}--\eqref{model5} can be expressed equivalently as
	\begin{align}
		E(\boldsymbol{\varphi})&=\frac{1}{2}\int_{\Omega}a_{\Omega}\varphi^{2}\,\mathrm{d}x-\frac{1}{2}\int_{\Omega}(J\ast\varphi)\varphi\,\mathrm{d}x+\int_{\Omega}(\widehat{\beta}(\varphi)+\widehat{\pi}(\varphi))\,\mathrm{d}x\notag\\
		&\quad+\frac{1}{2}\int_{\Gamma}a_{\Gamma}\psi^{2}\,\mathrm{d}S-\frac{1}{2}\int_{\Gamma}(K\circledast\psi)\psi\,\mathrm{d}S+\int_{\Gamma}(\widehat{\beta}_{\Gamma}(\psi)+\widehat{\pi}_{\Gamma}(\psi))\,\mathrm{d}S.\notag
	\end{align}
	Throughout this paper, we make the following assumptions.
	\begin{description}
		\item[(A1)] The convolution kernels $J,K:\mathbb{R}^{d}\rightarrow\mathbb{R}$ are even (i.e., $J(x)=J(-x)$ and $K(x)=K(-x)$ for almost all $x\in\mathbb{R}^{d}$), nonnegative almost everywhere and satisfy $J\in W^{1,1}(\mathbb{R}^{d})$ and $K\in W^{2,r}(\mathbb{R}^{d})$ with $r>1$. Note that the regularity assumption on $K$ is higher than that on $J$ since the traces $K(x-\cdot)|_\Gamma$ and $\nabla_{\Gamma}K(x-\cdot)|_\Gamma$ must exist and belong to $L^r(\Gamma)$ for all $x\in\Gamma$ (see \cite{KS}). In addition, we suppose that
		\begin{align}
		&a_{\ast}:=\inf_{x\in\Omega}\int_{\Omega}J(x-y)\,\mathrm{d}y>0, &&a_{\circledast}:=\inf_{x\in\Gamma}\int_{\Gamma}K(x-y)\,\mathrm{d}S_{y}>0, \label{2.1}\\
	&a^{\ast}:=\sup_{x\in\Omega}\int_{\Omega}J(x-y)\,\mathrm{d}y<+\infty, &&a^{\circledast}:=\sup_{x\in\Gamma}\int_{\Gamma}K(x-y)\,\mathrm{d}S_{y}<+\infty, \label{2.2}\\
			&b^{\ast}:=\sup_{x\in\Omega}\int_{\Omega}|\nabla J(x-y)|\,\mathrm{d}y<+\infty,&&b^{\circledast} :=\sup_{x\in\Gamma}\int_{\Gamma}|\nabla_{\Gamma}K(x-y)|\,\mathrm{d}S_{y}<+\infty. \label{2.3}
		\end{align}
		\item[(A2)] The nonlinear convex functions $\widehat{\beta}$, $\widehat{\beta}_{\Gamma}$ belong to $C([-1,1])\cap C^{2}(-1,1)$. Their derivatives are denoted by $\beta=\widehat{\beta}'$, $\beta_{\Gamma}=\widehat{\beta}_{\Gamma}'$ such that $\beta$, $\beta_{\Gamma}\in C^{1}(-1,1)$ are monotone increasing functions satisfying
		\begin{align*}
			&\lim_{s\rightarrow-1}\beta(s)=-\infty,\quad\lim_{s\rightarrow1}\beta(s)=+\infty,\\
				&\lim_{s\rightarrow-1}\beta_{\Gamma}(s)=-\infty, \quad\lim_{s\rightarrow1}\beta_{\Gamma}(s)=+\infty.
		\end{align*}
		Besides, their derivatives $\beta'$, $\beta_{\Gamma}'$ fulfill
		$$\beta'(s)\geq\alpha,\quad\beta_{\Gamma}'(s)\geq\alpha,\quad\forall\,s\in(-1,1)$$
		for some constant $\alpha>0$.
		We extend $\widehat{\beta}(s)=\widehat{\beta}_\Gamma(s)=+\infty$ for $s\notin[-1,1]$.
		Without loss of generality, we also assume $\widehat{\beta}(0)=\widehat{\beta}_\Gamma(0)=\beta(0)=\beta_\Gamma(0)=0$.
		This implies that $\widehat{\beta}(s)$, $\widehat{\beta}_\Gamma(s)\geq0$ for $s\in[-1,1]$.
		\item[(A3)] $\widehat{\pi}$, $\widehat{\pi}_{\Gamma}\in C^{1}(\mathbb{R})$ and $\pi:=\widehat{\pi}'$, $\pi_{\Gamma}:=\widehat{\pi}_{\Gamma}'$ are differentiable and Lipschitz continuous on $\mathbb{R}$ with Lipschitz constants $\gamma_{1}$ and $\gamma_{2}$, respectively.
		Furthermore, $\gamma_{1}$ and $\gamma_{2}$ satisfy
		\begin{align*}
		0<\gamma_{1}<a_{\ast}+\frac{\alpha}{1+\alpha},\quad0<\gamma_{2}<a_{\circledast}+\frac{\alpha}{1+\alpha}.
		\end{align*}
		\item[(A4)] The initial datum $\boldsymbol{\varphi}_{0}:=(\varphi_{0},\psi_{0})\in\mathcal{L}^{2}$ satisfies $\widehat{\beta}(\varphi_{0})\in L^{1}(\Omega)$, $\widehat{\beta}_{\Gamma}(\psi_{0})\in L^{1}(\Gamma)$ and
		\begin{align*}
		\begin{cases}
		\overline{m}_{0}=\overline{m}(\boldsymbol{\varphi}_{0})\in(-1,1),&\text{if }L\in[0,+\infty),\\[1mm]
		m_{\Omega,0}:=\langle\varphi_{0}\rangle_{\Omega}\in(-1,1),\quad	m_{\Gamma,0}:=\langle\psi_{0}\rangle_{\Gamma}\in(-1,1),&\text{if }L=+\infty.
		\end{cases}
		\end{align*}
	\end{description}
	
	\begin{remark}\rm
	\label{kernel}
	The assumptions on the kernel $J$ in $\mathbf{(A1)}$ are typical in the literature (cf. \cite{DST20,KS}) and sometimes a weaker assumption $J\in W^{1,1}_{\text{loc}}(\mathbb{R}^d)$ is used (see, e.g., \cite{GGG}). As explained in \cite[Remark 2.1]{KS}, demanding $J\in  W^{1,1}(\mathbb{R}^d)$ instead of $J\in W^{1,1}_{\text{loc}}(\mathbb{R}^d)$ is not a real restriction. Analogous arguments can be applied for the assumptions on $K$.
	\end{remark}
	
Below we report some basic properties of the kernels $J$, $K$ and the  corresponding convolutions (see \cite{KS}). As a direct consequence of \eqref{2.2} and \eqref{2.3}, we find for all functions $u\in L^{1}(\Omega)$ and $v\in L^{1}(\Gamma)$, it holds that $J\ast u\in W^{1,1}(\Omega)$ and $K\circledast v\in W^{1,1}(\Gamma)$ with
   \begin{align}
   	&\Vert J\ast u\Vert_{L^{1}(\Omega)}\leq a^{\ast}\Vert u\Vert_{L^{1}(\Omega)},\notag\\
   	&\Vert K\circledast v\Vert_{L^{1}(\Gamma)}\leq a^{\circledast}\Vert v\Vert_{L^{1}(\Gamma)},\notag\\
   	&\Vert\nabla(J\ast u)\Vert_{L^{1}(\Omega)}=\Vert(\nabla J)\ast u\Vert_{L^{1}(\Omega)}\leq b^{\ast}\Vert u\Vert_{L^{1}(\Omega)},\notag\\
   	&\Vert\nabla_{\Gamma}(K\circledast v)\Vert_{L^{1}(\Gamma)}=\Vert(\nabla_{\Gamma}K)\circledast v\Vert_{L^{1}(\Gamma)}\leq b^{\circledast}\Vert v\Vert_{L^{1}(\Gamma)}.\notag
   \end{align}

	\subsection{Statement of results}
	We now state our main results in this paper. The first part concerns the well-posedness of problem \eqref{model1}--\eqref{psiini}.
	\begin{definition}
	\label{weakdefn}
	Let $T\in(0,+\infty)$ be an arbitrary but given final time and $L\in[0,+\infty]$. Suppose that the assumptions $\mathbf{(A1)}$--$\mathbf{(A4)}$ are satisfied. The pair $(\boldsymbol{\varphi},\boldsymbol{\mu})$ with $
\boldsymbol{\varphi}:=(\varphi, \psi)$, $\boldsymbol{\mu}:=(\mu, \theta)$,
is called a weak solution to problem \eqref{model1}--\eqref{psiini} on $[0,T]$,
	if the following conditions are fulfilled:
	\begin{description}
	\item[(i)] The functions $(\boldsymbol{\varphi},\boldsymbol{\mu})$ satisfy the regularity properties
	\begin{align}
	&\begin{cases}
		\varphi\in H^{1}(0,T;V')\cap L^{\infty}(0,T;H)\cap L^{2}(0,T;V),\\[1mm]
		\psi\in H^{1}(0,T;V_{\Gamma}')\cap L^{\infty}(0,T;H_{\Gamma})\cap L^{2}(0,T;V_{\Gamma}),
     \end{cases}\quad\text{if }L\in(0,+\infty],\notag\\
	&\boldsymbol{\varphi}\in H^{1}(0,T;(\mathcal{V}^{1})')\cap L^{\infty}(0,T;\mathcal{L}^{2})\cap L^{2}(0,T;\mathcal{H}^{1}),\quad\quad\text{if }L=0,\notag\\
	&\mu, \,\beta(\varphi)\in L^{2}(0,T;V),\quad\theta,\,\beta_{\Gamma}(\psi)\in L^{2}(0,T;V_{\Gamma}),\notag
	\end{align}
and
	\begin{align}
		&\varphi\in L^{\infty}(Q_T),\text{ with }|\varphi(x,t)|<1\text{ a.e. in\ }  Q_{T},\label{phiinfty}\\[1mm]
		&\psi\in L^{\infty}(\Sigma_T),\text{ with }|\psi(x,t)|<1\text{ a.e. on }  \Sigma_{T}.\label{psiinfty}
	\end{align}
	\item[(ii)] If $L\in(0,+\infty]$, the following variational formulations
	\begin{align}
	&\langle\partial_{t}\varphi,z\rangle_{V',V} =-\int_{\Omega}\nabla\mu\cdot\nabla z\,\mathrm{d}x+\chi(L)\int_{\Gamma}(\theta-\mu)z\,\mathrm{d}S, \quad\forall\,z\in V,\label{weak1}\\
	&\langle\partial_{t}\psi,z_{\Gamma}\rangle_{V_{\Gamma}',V_{\Gamma}} =-\int_{\Gamma}\nabla_{\Gamma}\theta\cdot\nabla_{\Gamma} z_{\Gamma}\,\mathrm{d}S-\chi(L) \int_{\Gamma}(\theta-\mu)z_{\Gamma}\,\mathrm{d}S, \quad\forall\,z_\Gamma \in V_\Gamma,\label{weak2}
	\end{align}
hold for almost all $t\in(0,T)$,
while for $L=0$, the following variational formulation
	\begin{align}
	\langle\partial_{t}\boldsymbol{\varphi}, \boldsymbol{z}\rangle_{(\mathcal{V}^{1})',\mathcal{V}^{1}} =-\int_{\Omega}\nabla\mu\cdot\nabla z\,\mathrm{d}x-\int_{\Gamma}\nabla_{\Gamma}\theta\cdot\nabla_{\Gamma} z_{\Gamma}\,\mathrm{d}S,\quad\forall\,\boldsymbol{z}\in\mathcal{V}^1, \label{0weak1}
	\end{align}
	holds for almost all $t\in(0,T)$.
	The bulk and boundary chemical potentials $\mu$, $\theta$ satisfy
	\begin{align}
	&\mu=a_{\Omega}\varphi-J\ast\varphi+\beta(\varphi)+\pi(\varphi),\quad\text{a.e. in }Q_{T},\label{weak3}\\[1mm]
	&\theta=a_{\Gamma}\psi-K\circledast\psi+\beta_{\Gamma}(\psi)+\pi_{\Gamma}(\psi),\quad\text{a.e. on }\Sigma_{T}.\label{weak4}
	\end{align}
	Furthermore, the initial conditions $\varphi|_{t=0}=\varphi_0$ and $\psi|_{t=0}=\psi_0$ are satisfied almost everywhere in $\Omega$ and on $\Gamma$, respectively.
	\item[(iii)] The energy equality
	\begin{align}
		E(\boldsymbol{\varphi}(t))+\int_{0}^{t}\Big(\Vert\nabla\mu(s)\Vert_{H}^{2}+\Vert\nabla_{\Gamma}\theta(s)\Vert_{H_{\Gamma}}^{2}+\chi(L)\Vert\theta(s)-\mu(s)\Vert_{H_{\Gamma}}^{2}\Big)\,\mathrm{d}s= E(\boldsymbol{\varphi}_{0})\label{energyeq}
	\end{align}
	holds for all $t\in[0,T]$.
	\end{description}
	\end{definition}
Then we have
	\begin{theorem}[Existence of global weak solutions]
	\label{exist}
	Let $\Omega\subset\mathbb{R}^{d}$ $(d\in\{2,3\})$ be a bounded domain with smooth boundary $\Gamma=\partial\Omega$ and $T\in(0,+\infty)$ be an arbitrary but given final time.
	Suppose that the assumptions $\mathbf{(A1)}$--$\mathbf{(A4)}$ are satisfied. Then problem \eqref{model1}--\eqref{psiini} admits a global weak solution $(\boldsymbol{\varphi},\boldsymbol{\mu})$ on $[0,T]$ in the sense of Definition \ref{weakdefn}.
	\end{theorem}
	
	\begin{remark}\rm
	\label{L2-conti}
	By virtue of the regularity of $\boldsymbol{\varphi}$ and $\partial_{t}\boldsymbol{\varphi}$,
	we have $\boldsymbol{\varphi}\in C([0,T];\mathcal{L}^{2})$.
	This fact together with Lemma \ref{lem-subsequent} implies that $J\ast\varphi\in C([0,T];V)$
	and $K\circledast\psi\in C([0,T];V_{\Gamma})$.
	\end{remark}
	
	\begin{theorem}[Continuous dependence estimates and uniqueness]
		\label{contidependence}
	Suppose that the assumptions of Theorem \ref{exist} are satisfied.
	Let $(\boldsymbol{\varphi}_{i},\boldsymbol{\mu}_{i})$
	be two weak solutions to problem \eqref{model1}--\eqref{psiini}
	corresponding to the initial data $\boldsymbol{\varphi}_{0,i}$ $(i\in\{1,2\})$ with
	\begin{align*}
		\begin{cases}
	 \overline{m}(\boldsymbol{\varphi}_{0,1}) =\overline{m}(\boldsymbol{\varphi}_{0,2})=\overline{m}_{0},&\text{if }L\in[0,+\infty),\\[1mm]
	 \langle\varphi_{0,1}\rangle_{\Omega}=\langle\varphi_{0,2}\rangle_{\Omega} =m_{\Omega,0},\quad \langle\psi_{0,1}\rangle_{\Gamma}=\langle\psi_{0,2}\rangle_{\Gamma} =m_{\Gamma,0},&\text{if }L=+\infty.
	 \end{cases}
	\end{align*}
	 Then for almost all $t\in[0,T]$, we have
	\begin{align}
		&\Vert\boldsymbol{\varphi}_{1}(t)-\boldsymbol{\varphi}_{2}(t)\Vert_{L,0,\ast}^{2} +\int_{0}^{t}\Vert\boldsymbol{\varphi}_{1}(s)-\boldsymbol{\varphi}_{2}(s)\Vert_{\mathcal{L}^{2}}^{2}\,\mathrm{d}s\leq C\Vert\boldsymbol{\varphi}_{0,1}-\boldsymbol{\varphi}_{0,2}\Vert_{L,0,\ast}^{2}, \quad\text{if }L\in[0,+\infty),\label{contiesti}
	\end{align}
	and
	\begin{align}
		&\Vert\varphi_{1}(t)-\varphi_{2}(t)\Vert_{V^{\ast}_{0}}^{2} +\Vert\psi_{1}(t)-\psi_{2}(t)\Vert_{V^{\ast}_{\Gamma,0}}^{2}+\int_{0}^{t}\Vert\boldsymbol{\varphi}_{1}(s)-\boldsymbol{\varphi}_{2}(s)\Vert_{\mathcal{L}^{2}}^{2}\,\mathrm{d}s\notag\\
		&\qquad\leq C\big(\Vert\varphi_{0,1}-\varphi_{0,2}\Vert_{V^{\ast}_{0}}^{2} +\Vert\psi_{0,1}-\psi_{0,2}\Vert_{V^{\ast}_{\Gamma,0}}^{2}\big),\quad\text{if }L=+\infty,\label{0contiesti}
	\end{align}
	where the positive constant $C>0$ depends on the coefficients in the  assumptions, $\Omega$, $\Gamma$ and $T$. As a consequence, the weak solution $(\bm{\varphi},\bm{\mu})$ obtained in  Theorem \ref{exist} is unique.
	\end{theorem}
	
	
	The second part is related to the convergence rates. In the construction of weak solutions, two different limiting processes may take place:
	(i) in the Yosida approximation: passing to the limit as $\varepsilon\to0$;
	(ii) in the asymptotic limits: passing to the limits as $L\to0$ or $L\to+\infty$.
	In both cases, the existence results are achieved by the compactness argument, which does not provide any information on the convergence rate.
	We fill this gap under the following additional assumptions on the interaction kernel $J$ and Lipschitz constants for $\pi$, $\pi_\Gamma$:
		\begin{description}
		\item[(A5)] $J\in W^{1,1}(\mathbb{R}^{d})\cap L^{2}(\mathbb{R}^{d})$.
		
		\item[(A6)] The constants $\gamma_{1}$ and $\gamma_{2}$ in $\mathbf{(A3)}$ satisfy
		\begin{align*}
			0<\gamma_{1}<a_{\ast},\quad0<\gamma_{2}<a_{\circledast}.
		\end{align*}
	\end{description}

	\begin{theorem}[Convergence rate in the Yosida approximation]
		\label{rate}
		Suppose that the assumptions of Theorem \ref{exist} are satisfied and $L\in(0,+\infty)$.
		Assume in addition, $\mathbf{(A5)}$ and $\mathbf{(A6)}$ hold.
Let $\varepsilon^\ast$ be a given constant satisfying
\begin{align}
0<\varepsilon^\ast\leq\min\left\{\frac{1}{2\Vert J\Vert_{L^{1}(\Omega)}+2\gamma_{1}+1},\,\frac{1}{2\Vert K\Vert_{L^{1}(\Gamma)}+2\gamma_{2}+1}\right\}<1.
\label{epsi-star}
\end{align}
For any $\varepsilon\in(0,\varepsilon^\ast)$, we denote by $\boldsymbol{\varphi}_{\varepsilon}$ the weak solution to the approximating problem \eqref{appro1}--\eqref{psiappro} corresponding to the approximating parameter $\varepsilon$. Besides, let $\boldsymbol{\varphi}$ be the unique weak solution to the original problem \eqref{model1}--\eqref{psiini}.
		Then we have
		\begin{align}
			\Vert\boldsymbol{\varphi}_{\varepsilon} -\boldsymbol{\varphi}\Vert_{L^{\infty}(0,T;\mathcal{H}_{L,0}^{-1})} +\Vert\boldsymbol{\varphi}_{\varepsilon} -\boldsymbol{\varphi}\Vert_{L^{2}(0,T;\mathcal{L}^{2})}\leq C\sqrt{\varepsilon},\quad\text{as }\varepsilon\to0,\notag
		\end{align}
		where the positive constant $C$ depends on the coefficients in the assumptions, $\mathcal{L}^{2}$-norm of the initial datum $\boldsymbol{\varphi}_{0}$ and $\varepsilon^\ast$, but not on $\varepsilon$.
	\end{theorem}
	\begin{theorem}[Convergence rate in the asymptotic limits]
		\label{weak-convergence}
	Suppose that the assumptions of Theorem \ref{exist} are satisfied,
	$(\boldsymbol{\varphi}^{L},\boldsymbol{\mu}^{L})$ is the unique weak solution to problem \eqref{model1}--\eqref{psiini} corresponding to $L\in[0,+\infty]$ and $T\in (0,+\infty)$.
	Assume in addition, $\mathbf{(A5)}$ holds. Then, we have
	\begin{align}
	\|\boldsymbol{\varphi}^{L}- \boldsymbol{\varphi}^{0}\|_{L^{\infty}(0,T;\mathcal{V}_{(0)}^{-1})} +\|\boldsymbol{\varphi}^{L}-\boldsymbol{\varphi}^{0}\|_{L^{2}(0,T;\mathcal{L}^{2})}\leq C_{0}\sqrt{L},\quad \text{as }L\to0,\label{to0rate}
	\end{align}
	and
	\begin{align}
	&\|\varphi^{L}-\varphi^{\infty}\|_{L^{\infty}(0,T;V')} +\|\psi^{L}-\psi^{\infty}\|_{L^{\infty}(0,T;V_{\Gamma}')}\notag\\
	&\qquad+\|\boldsymbol{\varphi}^{L}- \boldsymbol{\varphi}^{\infty}\|_{L^{2}(0,T;\mathcal{L}^{2})}\leq \frac{C_{\infty}}{L^{1/4}},\quad \text{as }L\to+\infty,\label{toinftyrate}
	\end{align}
	where the positive constant $C_{0}$ is independent of $L\in(0,1]$ and the positive constant $C_{\infty}$ may depend on some large constant $L_0\gg 1$, but is independent of $L\in[L_0,+\infty)$.
	\end{theorem}

   Lastly, we state results on the regularity propagation and the strict separation property. The following result implies that every weak solution instantaneously regularizes for $t>0$.
   \begin{theorem}[Regularity of weak solutions]
   	\label{regularize}
   	Suppose that $\Omega\subset\mathbb{R}^{d}$ $(d\in\{2,3\})$ is a bounded domain with smooth boundary $\Gamma$,
   	the assumptions $(\mathbf{A1})$--$(\mathbf{A4})$ are satisfied and $L\in[0,+\infty]$.
   	Assume in addition, assumption $(\mathbf{A5})$ holds when $L=0$. Let $(\boldsymbol{\varphi},\boldsymbol{\mu})$ be the unique global weak solution to problem \eqref{model1}--\eqref{psiini} obtained in Theorem \ref{exist}.
   	For any $\tau>0$, it holds
   	\begin{align*}
   	&\boldsymbol{\varphi}\in L^{\infty}(\tau,+\infty;\mathcal{H}^{1}),\quad\partial_{t}\boldsymbol{\varphi}\in L^{2}_{\mathrm{uloc}}(\tau,+\infty;\mathcal{L}^{2}),\\
   	&\boldsymbol{\mu}\in L^{\infty}(\tau,+\infty;\mathcal{H}^{1})\cap L_{\mathrm{uloc}}^{2}(\tau,+\infty;\mathcal{H}^{2}),\\
   	&\begin{cases}
   	\partial_{t}\boldsymbol{\varphi}\in L^{\infty}(\tau,+\infty;\mathcal{H}_{L,0}^{-1}),&\text{if }L\in[0,+\infty),\\[1mm]
   	\partial_{t}\varphi\in L^{\infty}(\tau,+\infty;V_{0}^{\ast}),\quad\partial_{t}\psi\in L^{\infty}(\tau,+\infty;V_{\Gamma,0}^{\ast}),&\text{if }L=+\infty.
   	\end{cases}
   	\end{align*}
   	Moreover, we have $(\beta(\varphi), \beta_{\Gamma}(\psi)) \in L^{\infty}(\tau,+\infty;\mathcal{H}^{1})$
   	and thus
   	\begin{align}
   		&\beta(\varphi)\in L^{\infty}(\tau,+\infty;L^{p}(\Omega)),\quad
   		\begin{cases}
   		p\in[2,+\infty),&\text{if }d=2,\\[1mm]
   		p\in[2,6],&\text{if }d=3,
   		\end{cases}\label{regularize1}\\
   		&\beta_{\Gamma}(\psi)\in L^{\infty}(\tau,+\infty;L^{q}(\Gamma)),\quad q\in[2,+\infty).\label{regularize2}
   	\end{align}
   \end{theorem}
   \begin{remark}\rm
   	When we establish the regularity propagation of global weak solutions for the case $L=0$,
   	we regard problem \eqref{model1}--\eqref{psiini} with $L\in(0,+\infty)$ as an approximating problem
   	for problem \eqref{model1}--\eqref{psiini} with $L=0$.
   	In this process, we have to derive uniform estimates with respect to $L\in(0,1)$.
   	To deal the following term (cf. \eqref{quotient5})
   	$$\langle\partial_{t}^{h}\boldsymbol{\varphi}^L_{\varepsilon}, (J\ast\partial_{t}^{h}\varphi^L_{\varepsilon},K\circledast\partial_{t}^{h}\psi^L_{\varepsilon})\rangle_{(\mathcal{H}^{1})',\mathcal{H}^{1}},$$
   	we need the additional assumption $\mathbf{(A5)}$. We also remark that, the weak solution $(\boldsymbol{\varphi},\boldsymbol{\mu})$ becomes a strong one for $t>0$ since $\tau>0$ is arbitrary.
   \end{remark}

	Finally, we show the instantaneous strict separation property in both two and three dimensions.
    To this end,
    we make the following additional assumptions on singular potentials:
	\begin{description}
	\item[(A7)] The bulk and boundary potentials coincide, i.e., $\beta=\beta_\Gamma$.
	
	\item[(A8)] 
	As $\delta\rightarrow0$, for some constant $\kappa$, there holds
	\begin{align*}
		&	\frac{1}{\beta(1-2\delta)}=O\Big(\frac{1}{|\mathrm{ln}(\delta)|^{\kappa}}\Big),\quad\frac{1}{|\beta(-1+2\delta)|}=O\Big(\frac{1}{|\mathrm{ln}(\delta)|^{\kappa}}\Big),
	\end{align*}
	with $\kappa>0$ if $d=3$ and $\kappa>1/2$ if $d=2$.
	
	\item[(A9)] There exist $\widetilde{\delta}\in(0,1/2)$ and $\widetilde{C}_0\geq1$ such that for any $\delta\in (0,\widetilde{\delta})$, it holds
	\begin{align*}
		&\frac{1}{\beta'(1-2\delta)}\leq \widetilde{C}_0\delta,\quad\frac{1}{\beta'(-1+2\delta)}\leq \widetilde{C}_0\delta.
	\end{align*}
		
	\item[(A10)] There exists $\delta_{1}\in (0,1)$ such that $\beta'$ is monotone non-decreasing on $[1-\delta_{1},1)$ and monotone non-increasing in $(-1,-1+\delta_{1}]$. 
	\end{description}

	\begin{remark}\rm
		\label{about kappa}
	In order to deal with an additional term $\mathcal{K}_3$ (cf. \eqref{k3}) caused by the boundary condition \eqref{model5} with $L\in(0,+\infty)$, we need the compatibility condition $\mathbf{(A7)}$. Furthermore, we make assumptions $\mathbf{(A8)}$--$\mathbf{(A10)}$ in order to apply a suitable De Giorgi's iteration scheme, adapting the idea in \cite{GP,Gior,Po}. In view of  \eqref{241202-1}, we only need to assume $\kappa>0$ in $\mathbf{(A8)}$ in the three-dimensional case, which is different from the assumption $\kappa=1$ made in \cite{Gior,Po}. This relaxation is due to the other two stronger assumptions $\mathbf{(A9)}$ and $\mathbf{(A10)}$.
	\end{remark}

	\begin{theorem}[Strict separation property]
	\label{separation}
	Suppose that $\Omega\subset\mathbb{R}^{d}$ $(d\in\{2,3\})$ is a bounded domain with smooth boundary $\Gamma$,
the assumptions $(\mathbf{A1})$--$(\mathbf{A4})$ are satisfied and $L\in(0,+\infty)$. Let $(\boldsymbol{\varphi},\boldsymbol{\mu})$ be the unique global weak solution to problem \eqref{model1}--\eqref{psiini} obtained in Theorem \ref{exist}. In addition, we assume
	\begin{itemize}
	\item[(1)] If $d=2$, $\mathbf{(A7)}$ and $\mathbf{(A8)}$ hold.
	
	\item[(2)] If $d=3$, $\mathbf{(A7)}$--$\mathbf{(A10)}$ hold.
	\end{itemize}
	Then, for any $\tau>0$, there exists $\delta\in(0,1)$ such that the unique global solution to problem \eqref{model1}--\eqref{psiini} satisfies
	\begin{align}
	&|\varphi(x,t)|\leq1-\delta,\quad\text{for a.e. }(x,t)\in\Omega\times[\tau,+\infty),\label{sepa1}\\[1mm]
	&|\psi(x,t)|\leq1-\delta,\quad\text{for a.e. }(x,t)\in\Gamma\times[\tau,+\infty).\label{sepa2}
	\end{align}
	Moreover, there exists a positive constant $\widehat{C}_0$ such that
	\begin{align}
	\sup_{t\geq\tau}\|\boldsymbol{\mu}(t)\|_{\mathcal{L}^\infty}\leq \widehat{C}_0,\quad\sup_{t\geq\tau}\|\partial_t\boldsymbol{\mu}\|_{L^2(t,t+1;\mathcal{L}^2)}\leq\widehat{C}_0.\label{mu-high}
	\end{align}
	\end{theorem}

	\begin{remark}\rm
		\label{L=infty}
	When $L=+\infty$, the subsystems in the bulk and on the boundary are completely decoupled. Hence, we can establish the strict separation property of $\varphi$ and $\psi$ separately. In this case, we do not  need to assume $\beta=\beta_\Gamma$ as $\mathbf{(A7)}$. Precisely, for the singular potential $\beta$, we assume  
	\begin{itemize}
	\item [(1)] If $d=2$, $\mathbf{(A2)}$ and $\mathbf{(A8)}$ hold.
	\item [(2)]	If $d=3$, $\mathbf{(A2)}$ and $\mathbf{(A8)}$--$\mathbf{(A10)}$ hold.
	\end{itemize}
    While for the singular potential $\beta_\Gamma$, since the boundary $\Gamma$ is a $(d-1)$-dimensional compact smooth manifold, with $d-1\leq2$, we only need to assume that $\beta_\Gamma$ satisfies $\mathbf{(A8)}$ with $\kappa>1/2$. For futher details, we refer to \cite{Gior,Po,GP} for the strict separation property of the bulk phase variable $\varphi$, while for the boundary phase variable $\psi$, the same argument applies after a minor modification.
	\end{remark}

	\begin{remark}\rm
	\label{except0}
	In order to prove Theorem \ref{separation}, we make use of a suitable De Giorgi's iteration scheme for parabolic equations, which involves integration-by-parts of the chemical potential $\mu$. More precisely, we take test functions $\varphi_n\eta_n^2$ in \eqref{weak1} and $\psi_n\eta_n^2$ in \eqref{weak2} (see the proof of Theorem \ref{separation} for their definitions) and use the boundary condition \eqref{model5}. For this purpose, we need $\mathbf{(A7)}$ to deal with the difficulty caused by \eqref{model5}. When $L=0$, we observe that in the weak formula \eqref{0weak1}, the test function $\boldsymbol{z}$ has to belong to $\mathcal{V}^1$, but the global weak solution $\boldsymbol{\varphi}$ to problem \eqref{model1}--\eqref{psiini} only belongs to $\mathcal{H}^1$. This makes it impossible for us to construct suitable test function in \eqref{0weak1} as for the case $L\in (0,+\infty)$. Therefore, we cannot deal with the case $L=0$ in the current framework.
	\end{remark}
	
	\section{Well-posedness for the Case $L\in(0,+\infty)$}
	\setcounter{equation}{0}
	\subsection{The approximating problem}
	To prove the existence of weak solutions to problem \eqref{model1}--\eqref{psiini} with general singular potentials for the case $L\in(0,+\infty)$,
	we regard the derivatives $\beta$ and $\beta_{\Gamma}$ as maximal monotone graphs in $\mathbb{R}\times\mathbb{R}$
	and approximate the maximal monotone operators by means of suitable Yosida approximations.
	For every $\varepsilon\in (0,1)$, set $\beta_{\varepsilon}$, $\beta_{\Gamma,\varepsilon}:\mathbb{R}\rightarrow\mathbb{R}$, along with the associated resolvent operators $J_{\varepsilon}$, $J_{\Gamma,\varepsilon}:\mathbb{R}\rightarrow\mathbb{R}$ by
	\begin{align*}
		&\beta_{\varepsilon}(s):=\frac{1}{\varepsilon}(s-J_{\varepsilon}(s)) :=\frac{1}{\varepsilon}(s-(I+\varepsilon\beta)^{-1}(s)), \\
		&\beta_{\Gamma,\varepsilon}(s):=\frac{1}{\varepsilon}(s-J_{\Gamma,\varepsilon}(s)) :=\frac{1}{\varepsilon}(s-(I+\varepsilon\beta_{\Gamma})^{-1}(s)),
	\end{align*}
	for all $s\in\mathbb{R}$.
	The related Moreau--Yosida regularizations $\widehat{\beta}_{\varepsilon}$, $\widehat{\beta}_{\Gamma,\varepsilon}$ of $\widehat{\beta}$, $\widehat{\beta}_{\Gamma}:\mathbb{R}\rightarrow\mathbb{R}$ are then given by (see, e.g., \cite{R.E.S})
	\begin{align}
		&\widehat{\beta}_{\varepsilon}(s) :=\inf_{\tau\in\mathbb{R}}\left\{\frac{1}{2\varepsilon}|s-\tau|^{2} +\widehat{\beta}(\tau)\right\}=\frac{1}{2\varepsilon}|s-J_{\varepsilon}(s)|^{2} +\widehat{\beta}(J_{\varepsilon}(s)) =\int_{0}^{s}\beta_{\varepsilon}(\tau)\,\mathrm{d}\tau,
		\notag\\
		&\widehat{\beta}_{\Gamma,\varepsilon}(s) :=\inf_{\tau\in\mathbb{R}}\left\{\frac{1}{2\varepsilon}|s-\tau|^{2} +\widehat{\beta}_{\Gamma}(\tau)\right\}=\int_{0}^{s}\beta_{\Gamma,\varepsilon}(\tau)\,\mathrm{d}\tau.
		\notag
	\end{align}
	From $\mathbf{(A2)}$, we find $\beta_{\varepsilon}(0)=\beta_{\Gamma,\varepsilon}(0)=0$ and
	  $\beta_{\varepsilon}$, $\beta_{\Gamma,\varepsilon}$ are Lipschitz continuous with the Lipschitz constant $1/\varepsilon$ (see \cite[Propositions 2.6 and 2.7]{Br}).
	Thus, it follows that $\widehat{\beta}_{\varepsilon}$ and $\widehat{\beta}_{\Gamma,\varepsilon}$ are nonnegative convex functions with (at most) quadratic growth and (see, e.g., \cite{B,R.E.S})
	\begin{align}
		&|\beta_{\varepsilon}(s)|\leq|\beta(s)|,\quad|\beta_{\Gamma,\varepsilon}(s)|\leq|\beta_{\Gamma}(s)|, \quad \forall\ s\in (-1,1),
		\notag\\[1mm]
		&0\leq\widehat{\beta}_{\varepsilon}(s)\leq\widehat{\beta}(s),\quad 0\leq\widehat{\beta}_{\Gamma,\varepsilon}(s)\leq\widehat{\beta}_{\Gamma}(s),
		\quad\forall\,s\in\mathbb{R}.\notag
	\end{align}
	Additionally, the following inequalities hold for $\varepsilon\in (0,1)$ (see \cite[Section 5]{GMS09}):
	for any given $s_0\in (-1,1)$,
	\begin{align}
		\beta_{\varepsilon}(s)(s-s_{0})\geq\delta_{0}|\beta_{\varepsilon}(s)|-c_{1},
		\quad\beta_{\Gamma,\varepsilon}(s)(s-s_{0})\geq\delta_{0}|\beta_{\Gamma,\varepsilon}(s)|-c_{1},\quad\forall\;s\in\mathbb{R},
		\label{eq2.8}
	\end{align}
	where $\delta_{0}$, $c_{1}$ are positive constants depending on $s_0$, but independent of $\varepsilon$.
	As shown in \cite{GGG}, for any $\varepsilon\in(0,1)$, $\beta_{\varepsilon}'(s)$ and $\beta_{\Gamma,\varepsilon}'(s)$ exist for all $s\in\mathbb{R}$ and
	\begin{align}
		\beta_{\varepsilon}'(s)\geq \frac{\alpha}{1+\alpha},\quad\beta_{\Gamma,\varepsilon}'(s) \geq\frac{\alpha}{1+\alpha},\quad\forall\,s\in\mathbb{R}.\label{unibelow}
	\end{align}
	Furthermore, $\widehat{\beta}_\varepsilon$ and $\widehat{\beta}_{\Gamma,\varepsilon}$ are bounded from below uniformly with respect to $\varepsilon\in(0,\varepsilon^\ast)$, that is,
	\begin{align}
		\widehat{\beta}_{\varepsilon}(s) \geq\frac{1}{4\varepsilon^{\ast}}s^{2}-\widetilde{C},\quad\widehat{\beta}_{\Gamma,\varepsilon}(s)\geq\frac{1}{4\varepsilon^{\ast}}s^{2}-\widetilde{C},\quad\forall\,s\in\mathbb{R},\quad\forall\,0<\varepsilon\leq\varepsilon^{\ast},\label{Funibelow}
	\end{align}
	where $\varepsilon^\ast>0$ is given by \eqref{epsi-star}
	and $\widetilde{C}>0$ depends only on $\varepsilon^{\ast}$ but is independent of $\varepsilon\in(0,\varepsilon^{\ast})$.

	Let us consider the following approximating problem with $\varepsilon\in(0,\varepsilon^\ast)$:
	\begin{align}
		&\partial_{t}\varphi=\Delta\mu,&&\text{in }Q_{T},\label{appro1}\\
		&\mu=a_{\Omega}\varphi-J\ast\varphi+\beta_{\varepsilon}(\varphi)+\pi(\varphi),&&\text{in }Q_{T},\label{appro2}\\
		&\partial_{t}\psi=\Delta_{\Gamma}\theta-\partial_{\mathbf{n}}\mu,&&\text{on }\Sigma_{T},\label{appro3}\\
		&\theta=a_{\Gamma}\psi-K\circledast\psi+\beta_{\Gamma,\varepsilon}(\psi)+\pi_{\Gamma}(\psi),&&\text{on }\Sigma_{T},\label{appro4}\\
		&L\partial_{\mathbf{n}}\mu=\theta-\mu,\quad L\in(0,+\infty),&&\text{on }\Sigma_{T},\label{appro5}\\
		&\varphi|_{t=0}=\varphi_{0},&&\text{in }\Omega,\label{phiappro}\\
		&\psi|_{t=0}=\psi_{0},&&\text{on }\Gamma.\label{psiappro}
	\end{align}
	\begin{definition}
		\label{approweak}
		Let $\Omega\subset\mathbb{R}^d$ $(d\in\{2,3\})$ be a bounded domain with smooth boundary $\Gamma$,
		$T\in(0,+\infty)$ be an arbitrary but given final time and $(\varepsilon,L)\in(0,\varepsilon^\ast)\times(0,+\infty)$.
		The pair $(\boldsymbol{\varphi}_\varepsilon,\boldsymbol{\mu}_\varepsilon)$ with
$\boldsymbol{\varphi}_\varepsilon:=(\varphi_\varepsilon,\psi_\varepsilon)$, $\bm{\mu}_\varepsilon:=(\mu_\varepsilon,\theta_\varepsilon)$,
		is called a weak solution to approximating problem \eqref{appro1}--\eqref{psiappro} on $[0,T]$,
		if the following conditions are fulfilled:
		\begin{description}
			\item[(i)] The pair $(\boldsymbol{\varphi}_\varepsilon,\boldsymbol{\mu}_\varepsilon)$ satisfies the following regularity properties
			\begin{align}
				\begin{array}{l}
					\varphi_\varepsilon\in H^{1}(0,T;V')\cap L^{\infty}(0,T;H)\cap L^{2}(0,T;V),\\[1mm]
					\psi_\varepsilon\in H^{1}(0,T;V_{\Gamma}')\cap L^{\infty}(0,T;H_{\Gamma})\cap L^{2}(0,T;V_{\Gamma}),\\[1mm]
					\mu_\varepsilon\in L^{2}(0,T;V),\quad\theta_\varepsilon\in L^{2}(0,T;V_{\Gamma}).
				\end{array}\label{approregu}
			\end{align}
			\item[(ii)] The weak formulations
			\begin{align}
				&\langle\partial_{t}\varphi_\varepsilon,z\rangle_{V',V}=-\int_{\Omega}\nabla\mu_\varepsilon\cdot\nabla z\,\mathrm{d}x+\frac{1}{L}\int_{\Gamma}(\theta_\varepsilon-\mu_\varepsilon)z\,\mathrm{d}S,\quad\forall\,z\in V,\label{approweak1}\\
				&\langle\partial_{t}\psi_\varepsilon,z_{\Gamma}\rangle_{V_{\Gamma}',V_{\Gamma}}=-\int_{\Gamma}\nabla_{\Gamma}\theta_\varepsilon\cdot\nabla_{\Gamma} z_{\Gamma}\,\mathrm{d}S-\frac{1}{L}\int_{\Gamma}(\theta_\varepsilon-\mu_\varepsilon)z_{\Gamma}\,\mathrm{d}S,\quad\forall\,z_\Gamma\in V_\Gamma,\label{approweak2}\\
				&\int_{\Omega}\mu_\varepsilon\zeta\,\mathrm{d}x=\int_{\Omega}(a_{\Omega}\varphi_\varepsilon-J\ast\varphi_\varepsilon+\beta_{\varepsilon}(\varphi_\varepsilon)+\pi(\varphi_\varepsilon))\zeta\,\mathrm{d}x,\quad\forall\,\zeta\in H,\label{approweak3}\\
				&\int_{\Gamma}\theta_\varepsilon\zeta_{\Gamma}\,\mathrm{d}S=\int_{\Gamma}(a_{\Gamma}\psi_\varepsilon-K\circledast\psi_\varepsilon+\beta_{\Gamma,\varepsilon}(\psi_\varepsilon)+\pi_{\Gamma}(\psi_\varepsilon))\zeta_{\Gamma}\,\mathrm{d}S,\quad\forall\,\zeta_\Gamma\in H_\Gamma,\label{approweak4}
			\end{align}
			hold for almost all $t\in(0,T)$.
			Moreover, the initial conditions $\varphi_\varepsilon|_{t=0}=\varphi_{0}$
			and $\psi_\varepsilon|_{t=0}=\psi_{0}$ are satisfied almost everywhere in $\Omega$ and on $\Gamma$, respectively.
			\item[(iii)] The energy inequality
			\begin{align}
				E_{\varepsilon}(\boldsymbol{\varphi}_\varepsilon(t))+\int_{0}^{t}\Big(\Vert\nabla\mu_\varepsilon(s)\Vert_{H}^{2}+\Vert\nabla_{\Gamma}\theta_\varepsilon(s)\Vert_{H_{\Gamma}}^{2}+\frac{1}{L}\Vert\theta_\varepsilon(s)-\mu_\varepsilon(s)\Vert_{H_{\Gamma}}^{2}\Big)\,\mathrm{d}s\leq E(\boldsymbol{\varphi}_{0})\label{approenergyineq}
			\end{align}
			holds for all $t\in[0,T]$,
			where the total free energy $E_{\varepsilon}(\boldsymbol{\varphi}_\varepsilon)$ of the approximating problem is defined by
			\begin{align}
				E_{\varepsilon}(\boldsymbol{\varphi}_\varepsilon) &=\frac{1}{2}\int_{\Omega}a_{\Omega}\varphi_\varepsilon^{2}\,\mathrm{d}x -\frac{1}{2}\int_{\Omega}(J\ast\varphi_\varepsilon)\varphi_\varepsilon \,\mathrm{d}x+\int_{\Omega}(\widehat{\beta}_{\varepsilon}(\varphi_\varepsilon) +\widehat{\pi}(\varphi_\varepsilon))\,\mathrm{d}x\notag\\
				&\quad+\frac{1}{2}\int_{\Gamma}a_{\Gamma}\psi_\varepsilon^{2}\,\mathrm{d}S -\frac{1}{2}\int_{\Gamma}(K\circledast\psi_\varepsilon)\psi_\varepsilon \,\mathrm{d}S+\int_{\Gamma}(\widehat{\beta}_{\Gamma,\varepsilon}(\psi_\varepsilon) +\widehat{\pi}_{\Gamma}(\psi_\varepsilon))\,\mathrm{d}S. \notag
			\end{align}
		\end{description}
	\end{definition}
	
	\begin{remark}\rm
	By the regularity \eqref{approregu} and the weak formulations \eqref{approweak3}, \eqref{approweak4}, we find
	\begin{align}
		&\mu_\varepsilon=a_{\Omega}\varphi_\varepsilon-J\ast\varphi_\varepsilon+\beta_{\varepsilon}(\varphi_\varepsilon)+\pi(\varphi_\varepsilon),\quad\text{a.e. in }Q_{T},\label{appropointmu}\\
		&\theta_\varepsilon=a_{\Gamma}\psi_\varepsilon-K\circledast\psi_\varepsilon+\beta_{\Gamma,\varepsilon}(\psi_\varepsilon)+\pi_{\Gamma}(\psi_\varepsilon),\quad\text{a.e. on }\Sigma_{T}.\label{appropointtheta}
	\end{align}
	\end{remark}
	
   \begin{proposition}
   	\label{approexist}
   	Let $\Omega\subset\mathbb{R}^{d}$ $(d\in\{2,3\})$ be a bounded domain with smooth boundary $\Gamma=\partial\Omega$,
   	$T\in(0,+\infty)$ be an arbitrary but given final time.
   	Assumptions $\mathbf{(A1)}$--$\mathbf{(A4)}$ hold and $(\varepsilon,L)\in (0,\varepsilon^\ast)\times(0,+\infty)$.
   	Then, the approximating problem \eqref{appro1}--\eqref{psiappro} admits a unique global weak solution $(\boldsymbol{\varphi}_\varepsilon,\boldsymbol{\mu}_\varepsilon)$ on $[0,T]$
   	in the sense of Definition \ref{approweak}.
   \end{proposition}
   \begin{proof}
   	\textbf{Step 1. The Faedo--Galerkin scheme.} It is well-known that the problems
   	\begin{align*}
   		\begin{cases}
   		-\Delta u=\lambda_{\Omega}u,&\text{in }\Omega,\\
   		\partial_{\mathbf{n}}u=0,&\text{on }\Gamma,
   		\end{cases}
   		\quad\text{and}\quad-\Delta_{\Gamma} v=\lambda_{\Gamma}v\quad\text{on }\Gamma,
   	\end{align*}
   	have countable many eigenvalues,
   	which can be written as increasing sequences $\{\lambda_{\Omega}^{j}\}_{j\in\mathbb{Z}^+}$ and $\{\lambda_{\Gamma}^{j}\}_{j\in\mathbb{Z}^+}$, respectively.
   	The associated eigenfunctions $\{u_{j}\}_{j\in\mathbb{Z}^+}\subset V$ and $\{v_{j}\}_{j\in\mathbb{Z}^+}\subset V_{\Gamma}$ form an orthonormal basis of $H$ and $H_{\Gamma}$, respectively.
   	In particular, we fix the eigenfunctions associated to the first eigenvalues $\lambda_{\Omega}^{1}=\lambda_{\Gamma}^{1}=0$ as $u_{1}=|\Omega|^{-1/2}$ and $v_{1}=|\Gamma|^{-1/2}$.
   	Moreover, the eigenfunctions $\{u_{j}\}_{j\in\mathbb{Z}^+}$ and $\{v_{j}\}_{j\in\mathbb{Z}^+}$ also form an orthogonal basis of $V$ and $V_{\Gamma}$, respectively.
   Since in the definition of weak solutions (recall Definition \ref{weakdefn}), the phase variable $\varphi$ and chemical potential $\mu$ in the bulk only belong to $V$, their normal derivatives $\partial_\mathbf{n}\varphi$, $\partial_\mathbf{n}\mu$ on $\Gamma$ are not necessarily to be well defined. Hence, we can apply the idea for the local Cahn--Hilliard equation with a local dynamic boundary condition in \cite{CKSS,KS24} and construct weak solutions using a Faedo--Galerkin scheme involving the basis $\{u_{j}\}_{j\in\mathbb{Z}^+}$ and $\{v_{j}\}_{j\in\mathbb{Z}^+}$.
   	
   	For any $m\in\mathbb{Z}^+$, we introduce the finite-dimensional subspaces
   	\begin{align*}
   		\mathcal{A}_{m}=\text{span}\{u_{1},...,u_{m}\}\subset V,\quad\mathcal{B}_{m}=\text{span}\{v_{1},...,v_{m}\}\subset V_{\Gamma},
   	\end{align*}
   	along with the orthogonal $L^{2}(\Omega)$-projection $\mathbb{P}_{\mathcal{A}_{m}}$ and the orthogonal $L^{2}(\Gamma)$-projection $\mathbb{P}_{\mathcal{B}_{m}}$.
   	There exist constants $\widetilde{C}_{\Omega}$, $\widetilde{C}_{\Gamma}>0$ depending on $\Omega$ and $\Gamma$, respectively,
   	such that for all $u\in V$ and $v\in V_{\Gamma}$,
   	\begin{align*}
   		\Vert\mathbb{P}_{\mathcal{A}_{m}}u\Vert_{V}\leq \widetilde{C}_{\Omega}\Vert u\Vert_{V}\quad\text{and}\quad\Vert\mathbb{P}_{\mathcal{B}_{m}}v\Vert_{V_{\Gamma}}\leq \widetilde{C}_{\Gamma}\Vert v\Vert_{V_{\Gamma}}.
   	\end{align*}
   	To construct a sequence of approximating solutions, we consider the ansatz
   	\begin{align}
   	&\varphi_{m}(x,t):=\sum_{j=1}^{m}a_{j}^{m}(t)u_{j}(x), \quad\psi_{m}(x,t):=\sum_{j=1}^{m}b_{j}^{m}(t)v_{j}(x),\label{FG1}\\
   	&\mu_{m}(x,t):=\sum_{j=1}^{m}c_{j}^{m}(t)u_{j}(x), \quad\theta_{m}(x,t):=\sum_{j=1}^{m}d_{j}^{m}(t)v_{j}(x),\label{FG2}
   	\end{align}
   for every $m\in\mathbb{Z}^+$, which satisfy the discretized weak formulations
   	\begin{align}
   		&\langle\partial_{t}\varphi_{m},z\rangle_{V',V} =-\int_{\Omega}\nabla\mu_{m}\cdot\nabla z\,\mathrm{d}x+\frac{1}{L}\int_{\Gamma}(\theta_{m} -\mu_{m})z\,\mathrm{d}S,\label{FGweak1}\\
   		&\langle\partial_{t}\psi_{m},z_{\Gamma} \rangle_{V_{\Gamma}',V_{\Gamma}} =-\int_{\Gamma}\nabla_{\Gamma}\theta_{m}\cdot\nabla_{\Gamma} z_{\Gamma}\,\mathrm{d}S-\frac{1}{L}\int_{\Gamma}(\theta_{m} -\mu_{m})z_{\Gamma}\,\mathrm{d}S,\label{FGweak2}\\
   		&\int_{\Omega}\mu_{m}\zeta\,\mathrm{d}x=\int_{\Omega} (a_{\Omega}\varphi_{m}-J\ast\varphi_{m}+\beta_{\varepsilon}(\varphi_{m})+\pi(\varphi_{m}))\zeta\,\mathrm{d}x,\label{FGweak3}\\
   		&\int_{\Gamma}\theta_{m}\zeta_{\Gamma}\,\mathrm{d}S =\int_{\Gamma}(a_{\Gamma}\psi_{m}-K\circledast\psi_{m} +\beta_{\Gamma,\varepsilon}(\psi_{m})+\pi_{\Gamma}(\psi_{m})) \zeta_{\Gamma}\,\mathrm{d}S,\label{FGweak4}
   	\end{align}
   	on $[0,T]$ for all $z,\zeta\in\mathcal{A}_{m}$ and $z_{\Gamma},\zeta_{\Gamma}\in \mathcal{B}_{m}$, supplemented with the initial conditions
   	\begin{align}
   	\varphi_{m}(0)=\mathbb{P}_{\mathcal{A}_{m}}(\varphi_{0})\text{ a.e. in }\Omega,\quad\psi_{m}(0)=\mathbb{P}_{\mathcal{B}_{m}}(\psi_{0})\text{ a.e. on }\Gamma.\label{FGini}
   	\end{align}
   	The time-dependent coefficient vectors
   	$\boldsymbol{a}^{m}:=(a_{1}^{m},...,a_{m}^{m})$, $\boldsymbol{b}^{m}:=(b_{1}^{m},...,b_{m}^{m})$,
     $\boldsymbol{c}^{m}:=(c_{1}^{m},...,c_{m}^{m})$ and $\boldsymbol{d}^{m}:=(d_{1}^{m},...,d_{m}^{m})$ are assumed to be continuously differentiable.
   	Taking $z=u_{j}$ in \eqref{FGweak1} and $z_{\Gamma}=v_{j}$ in \eqref{FGweak2}, $j\in\{1,...,m\}$,
   	we find that $(\boldsymbol{a}^{m},\boldsymbol{b}^{m})^{\top}$ can be  determined
   	by a system of $2m$ ordinary differential equations,
   	whose right-hand side depends continuously on $\boldsymbol{a}^{m}$, $\boldsymbol{b}^{m}$, $\boldsymbol{c}^{m}$ and $\boldsymbol{d}^{m}$.
   	In view of \eqref{FGini}, this system is subject to the initial conditions
   	\begin{align*}
   	& a_{j}^{m}(0)=(\varphi_{0},u_{j})_{H},\quad\text{for  }j\in\{1,...,m\},\\
   	& b_{j}^{m}(0)=(\psi_{0},v_{j})_{H_{\Gamma}},\quad\text{for  }j\in\{1,...,m\}.
   	\end{align*}
   	Moreover, taking $\zeta=u_j$ in \eqref{FGweak3} and $\zeta_\Gamma=v_j$ in \eqref{FGweak4}, $j\in\{1,...,m\}$,
   	we infer that $(\boldsymbol{c}^{m},\boldsymbol{d}^{m})^{\top}$ is explicitly given by a system of $2m$ algebraic equations,
   	whose right-hand side is Lipschitz continuously with respect to  $(\boldsymbol{a}^{m},\boldsymbol{b}^{m})^{\top}$.
   Inserting $(\boldsymbol{c}^{m},\boldsymbol{d}^{m})^{\top}$ into the right-hand side of the aforementioned ODE system for $(\boldsymbol{a}^{m},\boldsymbol{b}^{m})^{\top}$,
   	we obtain a closed $2m$-dimensional ODE system,
   	whose right-hand side is Lipschitz continuously with respect to $(\boldsymbol{a}^{m},\boldsymbol{b}^{m})^{\top}$ itself.
   	Applying the Cauchy--Lipschitz theorem, we obtain a unique local solution $(\boldsymbol{a}^{m},\boldsymbol{b}^{m})^{\top}:\,[0,T_{m}] \rightarrow\mathbb{R}^{2m}$
   	 to the corresponding initial value problem for some $T_{m}\in (0, T]$.
   	This enables us to construct $(\boldsymbol{c}^{m},\boldsymbol{d}^{m})^{\top}\,:[0,T_{m}] \rightarrow\mathbb{R}^{2m}$
   	by the aforementioned algebraic equations.
   	In view of \eqref{FG1}, \eqref{FG2}, we obtain a unique approximating solution  $(\boldsymbol{\varphi}_m, \boldsymbol{\mu}_m)$ with
   	\begin{align*}
   		\boldsymbol{\varphi}_m:=(\varphi_{m},\psi_{m})\in C^{1}([0,T_{m}];\mathcal{H}^{1}),\quad\boldsymbol{\mu}_m:=(\mu_{m},\theta_{m})\in C([0,T_{m}];\mathcal{H}^{1}),
   	\end{align*}
   	which solves the discretized system \eqref{FGweak1}--\eqref{FGweak4} on the interval $[0,T_{m}]$ subject to the initial conditions \eqref{FGini}.
   \medskip
   	
   	\textbf{Step 2. Uniform estimates with respect to $m\in\mathbb{Z}^+$.}
   	We derive suitable estimates for approximating solutions $(\boldsymbol{\varphi}_{m},\boldsymbol{\mu}_{m})$,
   	which are uniform with respect to the approximating parameter  $m\in\mathbb{Z}^+$ and $T_m$.
   	In the following, we denote by $C$ a generic non-negative constant
   	depending only on the initial datum and the constants introduced in the assumptions including the final time $T$, but not on $m$ or $T_{m}$.
   	
   Taking $z=\mu_{m}$ in \eqref{FGweak1} and $z_{\Gamma}=\theta_{m}$ in \eqref{FGweak2},
  we infer from  \eqref{FGweak3} and \eqref{FGweak4} that
   	\begin{align}
   		\frac{\mathrm{d}}{\mathrm{d}t}E_\varepsilon(\boldsymbol{\varphi}_m(s)) +\int_{\Omega}|\nabla\mu_{m}(s)|^{2}\,\mathrm{d}x +\int_{\Gamma}|\nabla_{\Gamma}\theta_{m}(s)|^{2}\,\mathrm{d}S +\frac{1}{L}\int_{\Gamma}(\theta_{m}(s)-\mu_{m}(s))^{2}\,\mathrm{d}S =0\label{FGuni1}
   	\end{align}
   	for all $s\in(0,T_{m})$. Integrating \eqref{FGuni1} over  $[0,t]\subset[0,T_m]$ yields
   	\begin{align}
   		&E_\varepsilon(\boldsymbol{\varphi}_m(t))+\int_{0}^{t}\int_{\Omega}|\nabla\mu_{m}(s)|^{2}\,\mathrm{d}x\,\mathrm{d}s+\int_{0}^{t}\int_{\Gamma}|\nabla_{\Gamma}\theta_{m}(s)|^{2}\,\mathrm{d}S\,\mathrm{d}s\notag\\
   		&\qquad+\frac{1}{L}\int_{0}^{t}\int_{\Gamma}(\theta_{m}(s) -\mu_{m}(s))^{2}\,\mathrm{d}S\,\mathrm{d}s =E_\varepsilon(\boldsymbol{\varphi}_m(0)).\label{FGuni2}
   	\end{align}
   	Thanks to $\mathbf{(A3)}$, there exist two positive constants $C_{1}$ and $C_{2}$,
   	independent of $m\in\mathbb{Z}^+$ and $\varepsilon\in(0,\varepsilon^\ast)$, such that
   	\begin{align*}
   	|\widehat{\pi}(s)|\leq \frac{\gamma_{1}}{2}r^{2}+C_{1},\quad|\widehat{\pi}_{\Gamma}(s)|\leq\frac{\gamma_{2}}{2}s^{2}+C_{2},\quad\forall\,s\in\mathbb{R}.
   	\end{align*}
   	It follows from \eqref{2.1}, \eqref{2.2} and \eqref{Funibelow} that
   	\begin{align}
   		&E_{\varepsilon}(\boldsymbol{\varphi}) \geq\frac{a_{\ast}+a_{\circledast}}{2}\Vert\boldsymbol{\varphi} \Vert_{\mathcal{L}^{2}}^{2}-(\widetilde{C}+C_{1}+C_{2})(|\Omega|+|\Gamma|), \label{ineq1}\\
   		&E_{\varepsilon}(\boldsymbol{\varphi}) \leq\frac{1}{2}\Big(a^{\ast}+a^{\circledast}+\frac{1}{\varepsilon} +\Vert J\Vert_{L^{1}(\Omega)}+\Vert K\Vert_{L^{1}(\Gamma)}+\gamma_{1}+\gamma_{2}\Big) \Vert\boldsymbol{\varphi}\Vert_{\mathcal{L}^{2}}^{2} +(C_{1}+C_{2})(|\Omega|+|\Gamma|).\label{ineq2}
   	\end{align}
   	From \eqref{FGuni2}, \eqref{ineq1} and \eqref{ineq2}, we find  $E_\varepsilon(\boldsymbol{\varphi}_m(0))$  can be controlled by a constant independent of $m$, so that 	
   \begin{align}
   		&\Vert\varphi_{m}\Vert_{L^{\infty}(0,T_{m};H)} +\Vert\psi_{m}\Vert_{L^{\infty}(0,T_{m};H_{\Gamma})} +\Vert\nabla\mu_{m}\Vert_{L^{2}(0,T_{m};H)}
   \notag\\
   &\qquad+\Vert\nabla_{\Gamma}\theta_{m}\Vert_{L^{2}(0,T_{m};H_{\Gamma})} +\frac{1}{\sqrt{L}}\Vert\theta_{m}-\mu_{m}\Vert_{L^{2}(0,T_{m};H_{\Gamma})} \leq C.\label{FGuni3}
   	\end{align}
   	Taking $\zeta=1$ in \eqref{FGweak3} and $\zeta_{\Gamma}=1$ in \eqref{FGweak4},
   	using \eqref{FGuni3}, we infer
   	\begin{align}
   	\Big|\int_{\Omega}\mu_{m}\,\mathrm{d}x\Big|&=\Big|\int_{\Omega}(a_{\Omega}\varphi_{m}-J\ast\varphi_{m}+\beta_{\varepsilon}(\varphi_{m})+\pi(\varphi_{m}))\,\mathrm{d}x\Big|\leq C(1+\Vert\varphi_{m}\Vert_{L^{1}(\Omega)})\leq C,\label{FGuni4}\\
   	\Big|\int_{\Gamma}\theta_{m}\,\mathrm{d}S\Big|&=\Big|\int_{\Gamma}(a_{\Gamma}\psi_{m}-K\circledast\psi_{m}+\beta_{\Gamma,\varepsilon}(\psi_{m})+\pi_{\Gamma}(\psi_{m}))\,\mathrm{d}S\Big|\leq C(1+\Vert\psi_{m}\Vert_{L^{1}(\Gamma)})\leq C,\label{FGuni5}
   	\end{align}
   	for almost all $t\in[0,T_{m}]$.
   	Then, by \eqref{FGuni3}, \eqref{FGuni4}, \eqref{FGuni5} and the generalized Poincar\'e's inequality \eqref{Po3}, we obtain
   	\begin{align}
   	\Vert\mu_{m}\Vert_{L^{2}(0,T_{m};V)}+\Vert\theta_{m}\Vert_{L^{2}(0,T_{m};V_{\Gamma})}\leq C.\label{FGuni6}
   	\end{align}
   	For any $z^{\ast}\in V$ and $z_{\Gamma}^{\ast}\in V_{\Gamma}$, we set  $z=\mathbb{P}_{\mathcal{A}_{m}}z^{\ast}$ and $z_{\Gamma}=\mathbb{P}_{\mathcal{B}_{m}}z_{\Gamma}^{\ast}$.
   	Then, by \eqref{FGweak1} and \eqref{FGweak2}, we find
   	\begin{align*}
   	|\langle\partial_{t}\varphi_{m},z^{\ast}\rangle_{V',V}|&=|\langle\partial_{t}\varphi_{m},z\rangle_{V',V}|\\
   	&=\Big|-\int_{\Omega}\nabla\mu_{m}\cdot\nabla z\,\mathrm{d}x+\frac{1}{L}\int_{\Gamma}(\theta_{m}-\mu_{m})z\,\mathrm{d}S\Big|\\
   	&\leq C\big(\Vert\nabla\mu_{m}\Vert_{H}+\frac{1}{L}\Vert\theta_{m}-\mu_{m}\Vert_{H_{\Gamma}}\big)\Vert z\Vert_{V}\\
   	&\leq C\big(\Vert\nabla\mu_{m}\Vert_{H}+\frac{1}{L}\Vert\theta_{m}-\mu_{m}\Vert_{H_{\Gamma}}\big)\Vert z^{\ast}\Vert_{V}
   	\end{align*}
   	and similarly,
   	\begin{align*}
   		|\langle\partial_{t}\psi_{m},z_{\Gamma}^{\ast}\rangle_{V_{\Gamma}',V_{\Gamma}}|&=|\langle\partial_{t}\psi_{m},z_{\Gamma}\rangle_{V_{\Gamma}',V_{\Gamma}}|\\
   		&\leq C\big(\Vert\nabla_{\Gamma}\theta_{m}\Vert_{H_{\Gamma}}+\frac{1}{L}\Vert\theta_{m}-\mu_{m}\Vert_{H_{\Gamma}}\big)\Vert z_{\Gamma}^{\ast}\Vert_{V_{\Gamma}}.
   	\end{align*}
   Hence, it follows from \eqref{FGuni3} that
   	\begin{align}
   		\Vert\partial_{t}\varphi_{m}\Vert_{L^{2}(0,T_{m};V')} +\Vert\partial_{t}\psi_{m}\Vert_{L^{2}(0,T_{m};V_{\Gamma}')} \leq C.\label{FGuni6-2}
   	\end{align}
   	Finally, we derive higher order estimates for $\varphi_{m}$ and $\psi_{m}$.
   	Taking $\zeta=-\Delta\varphi_{m}$ in \eqref{FGweak3}, using integration by parts, we obtain
   	\begin{align*}
   		\int_{\Omega}(a_{\Omega}+\beta_{\varepsilon}'(\varphi_{m}))|\nabla\varphi_{m}|^{2}\,\mathrm{d}x&=\int_{\Omega}\nabla\mu_{m}\cdot\nabla\varphi_{m}\,\mathrm{d}x+\int_{\Omega}(\nabla J\ast\varphi_{m})\cdot\nabla\varphi_{m}\,\mathrm{d}x\\
   		&\quad-\int_{\Omega}\varphi_{m}\nabla a_{\Omega}\cdot\nabla\varphi_{m}\,\mathrm{d}x-\int_{\Omega}\pi'(\varphi_{m})|\nabla\varphi_{m}|^{2}\,\mathrm{d}x.
   	\end{align*}
   	According to $\mathbf{(A3)}$, \eqref{unibelow} and Young's inequality for convolution, we get
   	\begin{align}
   	\frac{\beta_{1}}{2}\Vert\nabla\varphi_{m}\Vert_{H}^{2} \leq\frac{3}{2\beta_{1}}\Vert\nabla\mu_{m}\Vert_{H}^{2} +\frac{3}{2\beta_{1}}\Vert\nabla J\Vert_{L^{1}(\Omega)}^{2}\Vert\varphi_{m}\Vert_{H}^{2} +\frac{3}{2\beta_{1}}\Vert\nabla a_{\Omega}\Vert_{L^{\infty}(\Omega)}^{2} \Vert\varphi_{m}\Vert_{H}^{2},\label{FGuni7}
   	\end{align}
   	where $\beta_{1}=\alpha/(1+\alpha)+a_{\ast}-\gamma_{1}$.
   	Analogously, taking $\zeta_{\Gamma}=-\Delta_{\Gamma}\psi_{m}$ in \eqref{FGweak4}, we find
   	\begin{align*}
   		\int_{\Gamma}(a_{\Gamma} +\beta_{\Gamma,\varepsilon}'(\psi_{m}))|\nabla_{\Gamma}\psi_{m}|^{2}\, \mathrm{d}S&=\int_{\Gamma}\nabla_{\Gamma}\theta_{m} \cdot\nabla_{\Gamma}\psi_{m}\,\mathrm{d}S +\int_{\Gamma}(\nabla_{\Gamma}K\circledast\psi_{m}) \cdot\nabla_{\Gamma}\psi_{m}\,\mathrm{d}S\\
   		&\quad-\int_{\Gamma}\psi_{m}\nabla_{\Gamma}a_{\Gamma} \cdot\nabla_{\Gamma}\psi_{m}\,\mathrm{d}S -\int_{\Gamma}\pi_{\Gamma}'(\psi_{m})|\nabla_{\Gamma}\psi_{m}|^{2} \,\mathrm{d}S,
   	\end{align*}
   which yields
   	\begin{align}
   	\frac{\beta_{2}}{2}\Vert\nabla_{\Gamma}\psi_{m}\Vert_{H_{\Gamma}}^{2} \leq\frac{3}{2\beta_{2}}\Vert\nabla_{\Gamma}\theta_{m}\Vert_{H_{\Gamma}}^{2} +\frac{3}{2\beta_{2}}\Vert\nabla_{\Gamma} K\Vert_{L^{1}(\Gamma)}^{2}\Vert\psi_{m}\Vert_{H_{\Gamma}}^{2} +\frac{3}{2\beta_{2}}\Vert\nabla a_{\Gamma}\Vert_{L^{\infty}(\Gamma)}^{2}\Vert\psi_{m} \Vert_{H_{\Gamma}}^{2},\label{FGuni8}
   	\end{align}
   	where $\beta_{2}=\alpha/(1+\alpha)+a_{\circledast}-\gamma_{2}$.
   	Then, by \eqref{2.3}, \eqref{FGuni3}, \eqref{FGuni7} and \eqref{FGuni8}, we can conclude
   	\begin{align}
   		\Vert\varphi_{m}\Vert_{L^{2}(0,T_{m};V)}+\Vert\psi_{m}\Vert_{L^{2}(0,T_{m};V_{\Gamma})}\leq C.\label{FGuni9}
   	\end{align}

   	\textbf{Step 3. Extension of the approximating solution onto $[0,T]$.} Using the definition of the approximating solution \eqref{FG1}, \eqref{FG2} and the uniform estimate \eqref{FGuni3},
   	we obtain for all $t\in[0,T_{m}]$ and $j\in\{1,...,m\}$
   	\begin{align*}
   	|a_{j}^{m}(t)|+|b_{j}^{m}(t)|&=\Big|\int_{\Omega}\varphi_{m}(t)u_{j}\,\mathrm{d}x\Big|+\Big|\int_{\Gamma}\psi_{m}(t)v_{j}\,\mathrm{d}S\Big|\\
   	&\leq\Vert\varphi_{m}\Vert_{L^{\infty}(0,T_{m};H)} +\Vert\psi_{m}\Vert_{L^{\infty}(0,T_{m};H_{\Gamma})}\leq C.
   	\end{align*}
   	This yields that the solution $(\boldsymbol{a}^{m},\boldsymbol{b}^{m})^{\top}$ is bounded on the time interval $[0,T_{m}]$
   	by a positive constant that is independent of $m$ and $T_{m}$.
   	Hence, the classical ODE theory enables us to extend the local solution  $(\boldsymbol{a}^{m},\boldsymbol{b}^{m})^{\top}$ on the whole interval $[0,T]$.
   	Since the coefficients $(\boldsymbol{c}^{m},\boldsymbol{d}^{m})^{\top}$ can be reconstructed from $(\boldsymbol{a}^{m},\boldsymbol{b}^{m})^{\top}$,
   	we see that $(\boldsymbol{c}^{m},\boldsymbol{d}^{m})^{\top}$ also exist  on $[0,T]$.
   	As a consequence, the approximating solution $(\boldsymbol{\varphi}_m,\boldsymbol{\mu}_m)$ exists on $[0,T]$
   	and satisfies the weak formulations \eqref{FGweak1}--\eqref{FGweak4} for each $m\in\mathbb{Z}^+$.
   	Moreover, the uniform estimates \eqref{FGuni3}, \eqref{FGuni6}, \eqref{FGuni6-2} and \eqref{FGuni9} established in Step 2
   	remain true after replacing $T_{m}$ with $T$.
   	In summary, for each $m\in\mathbb{Z}^+$, the approximating solution $(\boldsymbol{\varphi}_m,\boldsymbol{\mu}_m)$ satisfies
   	\begin{align}
   	&\Vert\boldsymbol{\varphi}_m\Vert_{L^{\infty}(0,T;\mathcal{L}^2)} +\Vert\boldsymbol{\varphi}_m\Vert_{L^{2}(0,T;\mathcal{H}^1)} +\Vert\boldsymbol{\mu}_{m}\Vert_{L^{2}(0,T;\mathcal{H}^1)}\notag\\
   	&\quad+\frac{1}{\sqrt{L}}\Vert\theta_{m}-\mu_{m} \Vert_{L^{2}(0,T;H_{\Gamma})} +\Vert\partial_{t}\varphi_{m}\Vert_{L^{2}(0,T;V')}	+\Vert\partial_{t}\psi_{m}\Vert_{L^{2}(0,T;V_{\Gamma}')}\leq C.\label{FGuni}
   	\end{align}

   	\textbf{Step 4. Existence and uniqueness.} Thanks to the uniform estimate \eqref{FGuni}, the Banach--Alaoglu theorem as well as the Aubin--Lions--Simon lemma (cf. Lemma \ref{ALS}), there exist limit functions $\boldsymbol{\varphi}_\varepsilon=(\varphi_\varepsilon,\psi_\varepsilon)$, $\boldsymbol{\mu}_\varepsilon=(\mu_\varepsilon,\theta_\varepsilon)$ such that
   	\begin{align}
   		\varphi_{m}&\rightarrow\varphi_\varepsilon&&\text{weakly in }H^{1}(0,T;V')\cap L^{2}(0,T;V),\notag\\
   		& &&\text{strongly in }C([0,T];H)\ \text{and a.e. in }Q_T,\notag\\
   		\psi_{m}&\rightarrow\psi_\varepsilon&&\text{weakly in }H^{1}(0,T;V_{\Gamma}')\cap L^{2}(0,T;V_{\Gamma}),\notag\\
   		& &&\text{strongly in }C([0,T];H_{\Gamma})\ \text{and a.e. on } \Sigma_T,\notag\\
   		\mu_{m}&\rightarrow\mu_\varepsilon&&\text{weakly in }L^{2}(0,T;V),\notag\\
   		\theta_{m}&\rightarrow\theta_\varepsilon&&\text{weakly in }L^{2}(0,T;V_{\Gamma}),\notag\\
   		\theta_{m}-\mu_{m}&\rightarrow\theta_\varepsilon-\mu_\varepsilon&&\text{weakly in }L^{2}(0,T;H_{\Gamma}),\notag
   	\end{align}
   	as $m\rightarrow+\infty$, along a non-relabeled subsequence.
   	From these convergence results, we can deduce the desired regularity properties \eqref{approregu} of  $(\boldsymbol{\varphi}_\varepsilon,\boldsymbol{\mu}_\varepsilon)$.
   	Hence, Definition \ref{approweak}-(i) is fulfilled.
   	Next, we can pass to the limit as $m\rightarrow+\infty$ in \eqref{FGweak1}--\eqref{FGweak4}
   	to show the existence of a weak solution to the approximating problem  \eqref{appro1}--\eqref{psiappro} on $[0,T]$.
   	Furthermore, according to the above convergence properties, and passing to the limit $m\rightarrow+\infty$ in \eqref{FGuni2},
   	the weak solution satisfies the energy inequality \eqref{approenergyineq} for almost all $t\in [0,T]$.
   Since the time integrals in \eqref{approenergyineq} are continuous with respect to $t$ and the mapping $t\to E_\varepsilon(\boldsymbol{\varphi}_\varepsilon(t))$  is lower
semicontinuous, we conclude that \eqref{approenergyineq} actually holds   for all $t\in[0,T]$. The argument is similar to those in \cite{GGG,KS} and we omit the details here.
   	The continuous dependence on the initial data can be easily achieved by the standard energy method,
   	which yields the uniqueness of weak solutions to \eqref{appro1}--\eqref{psiappro}.
   	The proof of Proposition \ref{approexist} is complete.
   \end{proof}

   \subsection{Existence and uniqueness of weak solutions}
     \textbf{Proof of Theorem \ref{exist} with} $L\in(0,+\infty)$.  We divide the proof into several steps.
    \smallskip

   \textbf{Step 1. Uniform estimates with respect to $\varepsilon\in(0,\varepsilon^\ast)$.} Let $(\boldsymbol{\varphi}_{\varepsilon},\boldsymbol{\mu}_{\varepsilon})$
     be the unique weak solution to approximating problem \eqref{appro1}--\eqref{psiappro}
   obtained in Proposition \ref{approexist} corresponding to $\varepsilon\in(0,\varepsilon^\ast)$.
   In the following, the symbol $C$ denotes a generic positive constant
   independent of $\varepsilon\in(0,\varepsilon^\ast)$.
   First, by the energy inequality \eqref{approenergyineq} and \eqref{ineq1}, we find, for any $\varepsilon\in(0,\varepsilon^\ast)$ and all $t\in[0,T]$,
   \begin{align}
   	&\frac{a_{\ast}+a_{\circledast}}{2}\Vert\boldsymbol{\varphi}_\varepsilon(t) \Vert_{\mathcal{L}^{2}}^{2} +\int_{0}^{t}\int_{\Omega}|\nabla\mu_{\varepsilon}(s)|^{2} \,\mathrm{d}x\,\mathrm{d}s+\int_{0}^{t}\int_{\Gamma}|\nabla_{\Gamma} \theta_{\varepsilon}(s)|^{2}\,\mathrm{d}S\,\mathrm{d}s\notag\\
   	&\quad+\frac{1}{L}\int_{0}^{t}\int_{\Gamma}(\theta_{\varepsilon}(s) -\mu_{\varepsilon}(s))^{2}\,\mathrm{d}S\,\mathrm{d}s\leq E(\boldsymbol{\varphi}_0)+ (\widetilde{C}+C_{1}+C_{2})(|\Omega|+|\Gamma|),\notag
   \end{align}
   which yields
   \begin{align}
  &\Vert\boldsymbol{\varphi}_{\varepsilon}\Vert_{L^{\infty}(0,T;\mathcal{L}^2)} +\Vert\nabla\mu_{\varepsilon}\Vert_{L^{2}(0,T;H)} +\Vert\nabla_{\Gamma}\theta_{\varepsilon}\Vert_{L^{2}(0,T;H_{\Gamma})} +\frac{1}{\sqrt{L}}\Vert\theta_{\varepsilon}-\mu_{\varepsilon} \Vert_{L^{2}(0,T;H_{\Gamma})}\leq C.\label{approuni1}
   \end{align}
   Next, taking $\zeta=\varphi_{\varepsilon}-\overline{m}_{0}$ in \eqref{approweak3} and $\zeta_{\Gamma}=\psi_{\varepsilon}-\overline{m}_{0}$ in \eqref{approweak4}, we obtain
   \begin{align}
   	&\int_{\Omega}\beta_{\varepsilon}(\varphi_{\varepsilon})(\varphi_{\varepsilon} -\overline{m}_{0})\,\mathrm{d}x+\int_{\Gamma}\beta_{\Gamma,\varepsilon}(\psi_{\varepsilon})(\psi_{\varepsilon}-\overline{m}_{0})\,\mathrm{d}S\notag\\
   	&\quad=\int_{\Omega}\mu_{\varepsilon}(\varphi_{\varepsilon} -\overline{m}_{0})\,\mathrm{d}x+\int_{\Gamma}\theta_{\varepsilon}(\psi_{\varepsilon}-\overline{m}_{0})\,\mathrm{d}S\notag\\
   	&\qquad+\int_{\Omega}(J\ast\varphi_{\varepsilon})(\varphi_{\varepsilon} -\overline{m}_{0})\,\mathrm{d}x+\int_{\Gamma}(K\circledast\psi_{\varepsilon})(\psi_{\varepsilon}-\overline{m}_{0})\,\mathrm{d}S\notag\\
   	&\qquad-\int_{\Omega}(a_{\Omega}\varphi_{\varepsilon} +\pi(\varphi_{\varepsilon}))(\varphi_{\varepsilon}-\overline{m}_{0})\,\mathrm{d}x-\int_{\Gamma} (a_{\Gamma}\psi_{\varepsilon}+\pi_{\Gamma}(\psi_{\varepsilon}))(\psi_{\varepsilon}-\overline{m}_{0})\,\mathrm{d}S.\label{approuni2}
   \end{align}
   Concerning the terms on the left-hand side of \eqref{approuni2}, we infer from \eqref{eq2.8} that
   \begin{align}
   	&\int_{\Omega}\beta_{\varepsilon}(\varphi_{\varepsilon})(\varphi_{\varepsilon} -\overline{m}_{0})\,\mathrm{d}x+\int_{\Gamma}\beta_{\Gamma,\varepsilon} (\psi_{\varepsilon})(\psi_{\varepsilon}-\overline{m}_{0})\,\mathrm{d}S\notag\\
   	&\quad\geq \delta_{0}\Big(\int_{\Omega}|\beta_{\varepsilon}(\varphi_{\varepsilon})| \,\mathrm{d}x+\int_{\Gamma}|\beta_{\Gamma,\varepsilon}(\psi_{\varepsilon})| \,\mathrm{d}S\Big)-c_{1}(|\Omega|+|\Gamma|).\label{approuni4}
   \end{align}
   For the first line on the right-hand side of \eqref{approuni2}, by Young's inequality and the generalized Poincar\'e's inequality \eqref{Po3}, we obtain
   \begin{align}
   	&\int_{\Omega}\mu_{\varepsilon}(\varphi_{\varepsilon}-\overline{m}_{0})\,\mathrm{d}x+\int_{\Gamma}\theta_{\varepsilon}(\psi_{\varepsilon}-\overline{m}_{0})\,\mathrm{d}S\notag\\
   	&\quad=\int_{\Omega}(\mu_{\varepsilon}-\overline{m}(\boldsymbol{\mu}_\varepsilon))(\varphi_{\varepsilon}-\overline{m}_{0})\,\mathrm{d}x+\int_{\Gamma}(\theta_{\varepsilon}-\overline{m}(\boldsymbol{\mu}_\varepsilon))(\psi_{\varepsilon}-\overline{m}_{0})\,\mathrm{d}S\notag\\
   	&\quad=\int_{\Omega}(\mu_{\varepsilon}-\overline{m}(\boldsymbol{\mu}_\varepsilon))\varphi_{\varepsilon}\,\mathrm{d}x+\int_{\Gamma}(\theta_{\varepsilon}-\overline{m}(\boldsymbol{\mu}_\varepsilon))\psi_{\varepsilon}\,\mathrm{d}S\notag\\
   	&\quad\leq\Vert\mu_{\varepsilon}-\overline{m}(\boldsymbol{\mu}_\varepsilon)\Vert_{H}\Vert\varphi_{\varepsilon}\Vert_{H}+\Vert\theta_{\varepsilon}-\overline{m}(\boldsymbol{\mu}_\varepsilon)\Vert_{H_{\Gamma}}\Vert\psi_{\varepsilon}\Vert_{H_{\Gamma}}\notag\\
   	&\quad\leq C(\Vert\nabla\mu_{\varepsilon}\Vert_{H}+\Vert\nabla_{\Gamma}\theta_{\varepsilon}\Vert_{H_{\Gamma}}+\frac{1}{\sqrt{L}}\Vert\theta_{\varepsilon}-\mu_{\varepsilon}\Vert_{H_{\Gamma}})(\Vert\varphi_{\varepsilon}\Vert_{H}+\Vert\psi_{\varepsilon}\Vert_{H_{\Gamma}}).\label{approuni3}
   \end{align}
   For the remaining terms on the right-hand side of \eqref{approuni2}, using $\mathbf{(A1)}$, $\mathbf{(A3)}$, H\"older's inequality and Young's inequality for convolution, we have
   \begin{align}
   	&\int_{\Omega}(J\ast\varphi_{\varepsilon})(\varphi_{\varepsilon}-\overline{m}_{0})\,\mathrm{d}x+\int_{\Gamma}(K\circledast\psi_{\varepsilon})(\psi_{\varepsilon}-\overline{m}_{0})\,\mathrm{d}S\notag\\[1mm]
   	&\quad\leq\Vert J\ast\varphi_{\varepsilon}\Vert_{H} \Vert\varphi_{\varepsilon}-\overline{m}_{0}\Vert_{H} +\Vert K\circledast\psi_{\varepsilon}\Vert_{H_{\Gamma}} \Vert\psi_{\varepsilon}-\overline{m}_{0}\Vert_{H_{\Gamma}} \notag\\[1mm]
   	&\quad\leq \Vert J\Vert_{L^{1}(\Omega)}\Vert\varphi_{\varepsilon}\Vert_{H}\Vert\varphi_{\varepsilon}-\overline{m}_{0}\Vert_{H}+\Vert K\Vert_{L^{1}(\Gamma)}\Vert\psi_{\varepsilon}\Vert_{H_{\Gamma}}\Vert\psi_{\varepsilon}-\overline{m}_{0}\Vert_{H_{\Gamma}}\label{approuni5-1}
   	\end{align}
   	and
   	\begin{align}
   	&-\int_{\Omega}(a_{\Omega}\varphi_{\varepsilon}+\pi(\varphi_{\varepsilon}))(\varphi_{\varepsilon}-\overline{m}_{0})\,\mathrm{d}x-\int_{\Gamma}(a_{\Gamma}\psi_{\varepsilon}+\pi_{\Gamma}(\psi_{\varepsilon}))(\psi_{\varepsilon}-\overline{m}_{0})\,\mathrm{d}S\notag\\[1mm]
   	&\quad\leq(\Vert\pi(\varphi_{\varepsilon})\Vert_{H}+\Vert a_{\Omega}\Vert_{L^{\infty}(\Omega)}\Vert\varphi_{\varepsilon}\Vert_{H})\Vert\varphi_{\varepsilon}-\overline{m}_{0}\Vert_{H}\notag\\[1mm]
   	&\qquad+(\Vert \pi_{\Gamma}(\psi_{\varepsilon})\Vert_{H_{\Gamma}}+\Vert a_{\Gamma}\Vert_{L^{\infty}(\Gamma)}\Vert\psi_{\varepsilon}\Vert_{H_{\Gamma}}) \Vert\psi_{\varepsilon}-\overline{m}_{0}\Vert_{H_{\Gamma}}\notag\\[1mm]
   	&\quad\leq((a^{\ast}+\gamma_{1})\Vert\varphi_{\varepsilon}\Vert_{H}+|\pi(0)|\,|\Omega|^{1/2})\Vert\varphi_{\varepsilon}-\overline{m}_{0}\Vert_{H}\notag\\[1mm]
   	&\qquad+((a^{\circledast}+\gamma_{2})\Vert\psi_{\varepsilon}\Vert_{H_{\Gamma}}+|\pi_{\Gamma}(0)|\,|\Gamma|^{1/2})\Vert\psi_{\varepsilon}-\overline{m}_{0}\Vert_{H_{\Gamma}}.\label{approuni5}
   \end{align}
   Combining \eqref{approuni1}--\eqref{approuni5}, we can conclude that
   \begin{align}
   \Vert \beta_{\varepsilon}(\varphi_{\varepsilon})\Vert_{L^{2}(0,T;L^{1}(\Omega))}+\Vert \beta_{\Gamma,\varepsilon}(\psi_{\varepsilon})\Vert_{L^{2}(0,T;L^{1}(\Gamma))}\leq C.\label{approuni6}
   \end{align}
   Taking $\zeta=1$ in \eqref{approweak3} and $\zeta_{\Gamma}=1$ in \eqref{approweak4} yields
   \begin{align*}
   	\Big|\int_{\Omega}\mu_{\varepsilon}\,\mathrm{d}x\Big| +\Big|\int_{\Gamma}\theta_{\varepsilon}\,\mathrm{d}S\Big|&\leq \Vert J\ast\varphi_{\varepsilon}\Vert_{L^{1}(\Omega)}+\Vert \beta_{\varepsilon}(\varphi_{\varepsilon})\Vert_{L^{1}(\Omega)}+\Vert K\circledast\psi_{\varepsilon}\Vert_{L^{1}(\Gamma)}+\Vert \beta_{\Gamma,\varepsilon}(\psi_{\varepsilon})\Vert_{L^{1}(\Gamma)}.
   \end{align*}
   This together with \eqref{approuni6} implies
   \begin{align}
   	\Big\Vert\int_{\Omega}\mu_{\varepsilon}\,\mathrm{d}x \Big\Vert_{L^{2}(0,T)}+\Big\Vert\int_{\Gamma}\theta_{\varepsilon} \,\mathrm{d}S\Big\Vert_{L^{2}(0,T)}\leq C.\label{approuni7}
   \end{align}
   Then we infer from \eqref{approuni1}, \eqref{approuni7} and the generalized Poincar\'e's inequality \eqref{Po3} that
   \begin{align}
   	\Vert\mu_{\varepsilon}\Vert_{L^{2}(0,T;V)} +\Vert\theta_{\varepsilon}\Vert_{L^{2}(0,T;V_{\Gamma})}\leq C.\label{approuni8}
   \end{align}
   By comparison in \eqref{appropointmu} and \eqref{appropointtheta}, we further get
   \begin{align}
   	\Vert \beta_{\varepsilon}(\varphi_{\varepsilon})\Vert_{L^{2}(0,T;H)}+\Vert \beta_{\Gamma,\varepsilon}(\psi_{\varepsilon})\Vert_{L^{2}(0,T;H_{\Gamma})}\leq C.\label{approuni10}
   \end{align}
   Besides, using the weak formulations \eqref{approweak1}, \eqref{approweak2} and \eqref{approuni1}, we can conclude
   \begin{align}
   	\Vert\partial_{t}\varphi_{\varepsilon}\Vert_{L^{2}(0,T;V')}+	\Vert\partial_{t}\psi_{\varepsilon}\Vert_{L^{2}(0,T;V_{\Gamma}')}\leq C.\label{approuni9}
   \end{align}
   Finally, taking gradient on both sides of \eqref{appropointmu} and surface gradient
   on both sides of \eqref{appropointtheta}, testing the resultants by $\nabla\varphi_{\varepsilon}$ and $\nabla_{\Gamma}\psi_{\varepsilon}$, respectively, and applying similar calculations for  \eqref{FGuni7}--\eqref{FGuni9}, we have
   \begin{align}
   	\Vert\varphi_{\varepsilon}\Vert_{L^{2}(0,T;V)}+\Vert\psi_{\varepsilon}\Vert_{L^{2}(0,T;V_{\Gamma})}\leq C.\label{approuni11}
   \end{align}

   \textbf{Step 2. Passage to the limit as $\varepsilon\to 0$.}   By the uniform estimates \eqref{approuni1} and \eqref{approuni7}--\eqref{approuni11},
   we can conclude that there exist functions $\boldsymbol{\varphi}=(\varphi,\psi)$, $\boldsymbol{\mu}=(\mu,\theta)$ and $\boldsymbol{\xi}=(\xi,\xi_{\Gamma})$ such that
   \begin{align}
   	\varphi_{\varepsilon}&\rightarrow\varphi&&\text{weakly in }H^{1}(0,T;V')\cap L^{2}(0,T;V),\label{approconver1}\\
   	& &&\text{strongly in }C([0,T];H)\ \text{and a.e. in\ } Q_T,\notag\\
   	\psi_{\varepsilon}&\rightarrow\psi&&\text{weakly in }H^{1}(0,T;V_{\Gamma}')\cap L^{2}(0,T;V_{\Gamma}),\label{approconver2}\\
   	& &&\text{strongly in }C([0,T];H_{\Gamma})\ \text{and a.e.on }\Sigma_T,\notag\\
   	\mu_{\varepsilon}&\rightarrow\mu&&\text{weakly in }L^{2}(0,T;V),\label{approconver3}\\
   	\theta_{\varepsilon}&\rightarrow\theta &&\text{weakly in }L^{2}(0,T;V_{\Gamma}),\label{approconver4}\\
   	\theta_{\varepsilon}-\mu_{\varepsilon}&\rightarrow\theta-\mu&&\text{weakly in }L^{2}(0,T;H_{\Gamma}),\label{approconver5}\\
   	\beta_{\varepsilon}(\varphi_{\varepsilon})&\rightarrow\xi&&\text{weakly in }L^{2}(0,T;H),\label{approconver6}\\
   	\beta_{\Gamma,\varepsilon}(\psi_{\varepsilon}) &\rightarrow\xi_{\Gamma}&&\text{weakly in }L^{2}(0,T;H_{\Gamma}),\label{approconver7}
   \end{align}
   as $\varepsilon\rightarrow0$ in the sense of a subsequence.
   Thanks to $\mathbf{(A1)}$, Lemma \ref{lem-subsequent}, \eqref{approconver1} and \eqref{approconver2}, it holds
   \begin{align*}
   	J\ast\varphi_{\varepsilon}\to J\ast\varphi&\quad\text{strongly in }L^{2}(0,T;V),\\
   	K\circledast\psi_{\varepsilon}\to K\circledast\psi&\quad\text{strongly in }L^{2}(0,T;V_{\Gamma}).
   \end{align*}
   By the same argument as in \cite{GGG}, we can show that the limit functions $\varphi$ and $\psi$ fulfill
   \begin{align*}
   	|\varphi(x,t)|<1\quad\text{a.e. in } Q_{T},\qquad|\psi(x,t)|<1\quad\text{a.e. on }  \Sigma_{T},
   \end{align*}
   which yields \eqref{phiinfty} and \eqref{psiinfty}.
   As a byproduct, it holds
   \begin{align*}
   	\beta_{\varepsilon}(\varphi_{\varepsilon})\to\beta(\varphi)\quad\text{a.e. in } Q_{T},\quad \beta_{\Gamma,\varepsilon}(\psi_{\varepsilon}) \to\beta_{\Gamma}(\psi)\quad\text{a.e. on }  \Sigma_{T},
   \end{align*}
   where we have used the pointwise convergence of $\varphi_{\varepsilon}$, $\psi_{\varepsilon}$ and the uniform convergence of $\beta_{\varepsilon}$, $\beta_{\Gamma,\varepsilon}$ to $\beta$, $\beta_{\Gamma}$.
   Moreover, by \eqref{approconver6}, \eqref{approconver7} and the strong convergence results \eqref{approconver1}, \eqref{approconver2}, we can identify
   \begin{align}
   	\xi= \beta(\varphi)\quad\text{and}\quad\xi_{\Gamma}= \beta_{\Gamma}(\psi).\notag
   \end{align}
   Passing to the limit $\varepsilon\rightarrow0$ in \eqref{approweak1}--\eqref{approweak4} (in the sense of a convergent subsequence),
    we can check that $(\boldsymbol{\varphi},\boldsymbol{\mu})$ satisfies the formulations   \eqref{weak1}--\eqref{weak4}. By comparison in \eqref{weak3} and \eqref{weak4}, we see that $\beta(\varphi)\in L^2(0,T;V)$ and $\beta_\Gamma(\psi)\in L^2(0,T;V_\Gamma)$.
   Furthermore, by the lower weak semicontinuity of norms, passing to the limit $\varepsilon\to0$ in \eqref{approenergyineq}, we can conclude the energy inequality
   \begin{align}
   	E(\boldsymbol{\varphi}(t))+\int_{0}^{t} \Big(\Vert\nabla\mu(s)\Vert_{H}^{2}+\Vert\nabla_{\Gamma}\theta(s) \Vert_{H_{\Gamma}}^{2}+\frac{1}{L}\Vert\theta(s)-\mu(s)\Vert_{H_{\Gamma}}^{2} \Big)\,\mathrm{d}s\leq E(\boldsymbol{\varphi}_{0}),
   \quad \forall\,t\in[0,T].
   \notag 
   \end{align}

   \textbf{Step 3. Energy equality.} Since $T>0$ is arbitrary, the solution $(\boldsymbol{\varphi},\boldsymbol{\mu})$ can be defined on the whole interval $[0,+\infty)$. Next, we show that for any given $\tau>0$, there exists a constant $C>0$, independent of $\tau$, such that
   \begin{align}
   	\Vert\partial_t\boldsymbol{\varphi}\Vert_{L^\infty(\tau,+\infty;\mathcal{H}_{L,0}^{-1})}^{2} +\int_{t}^{t+1}\Vert\partial_t\boldsymbol{\varphi}(s) \Vert_{\mathcal{L}^{2}}^{2}\,\mathrm{d}s \leq\frac{C}{\tau},\quad\forall\,t\geq\tau.\label{Tuni3}
   \end{align}
   We denote the difference quotient of a function $f$ at time $t$ by $\partial_{t}^{h}f:=(f(t+h)-f(t))/h$.
   Let $(\boldsymbol{\varphi}_{\varepsilon},\boldsymbol{\mu}_{\varepsilon})$ be the weak solution to the approximating problem \eqref{appro1}--\eqref{psiappro}.
   Taking the difference of \eqref{approweak1}--\eqref{approweak4} at time $t$ and $t+h$, dividing the resultant by $h$, we obtain
   \begin{align}
   	&\langle\partial_{t}\partial_{t}^{h}\varphi_{\varepsilon},z\rangle_{V',V} =-\int_{\Omega}\nabla\partial_{t}^{h}\mu_{\varepsilon}\cdot\nabla z\,\mathrm{d}x+\frac{1}{L}\int_{\Gamma}(\partial_{t}^{h} \theta_{\varepsilon}-\partial_{t}^{h}\mu_{\varepsilon})z\,\mathrm{d}S, \label{quotient1}\\
   	&\langle\partial_{t}\partial_{t}^{h}\psi_{\varepsilon},z_{\Gamma} \rangle_{V_{\Gamma}',V_{\Gamma}}=-\int_{\Gamma}\nabla_{\Gamma}\partial_{t}^{h} \theta_{\varepsilon}\cdot\nabla_{\Gamma} z_{\Gamma}\,\mathrm{d}S-\frac{1}{L}\int_{\Gamma}(\partial_{t}^{h} \theta_{\varepsilon}-\partial_{t}^{h}\mu_{\varepsilon})z_{\Gamma}\,\mathrm{d}S,\label{quotient2}\\
   	&\int_{\Omega}\partial_{t}^{h}\mu_{\varepsilon}\,\zeta\,\mathrm{d}x =\int_{\Omega}\Big(a_{\Omega}\partial_{t}^{h}\varphi_{\varepsilon}-J\ast\partial_{t}^{h} \varphi_{\varepsilon}+\partial_{t}^{h}\beta_{\varepsilon}(\varphi_{\varepsilon}) +\partial_{t}^{h}\pi(\varphi_{\varepsilon})\Big)\zeta\,\mathrm{d}x,\label{quotient3}\\
   	&\int_{\Gamma}\partial_{t}^{h}\theta_{\varepsilon}\,\zeta_{\Gamma} \,\mathrm{d}S=\int_{\Gamma}\Big(a_{\Gamma}\partial_{t}^{h} \psi_{\varepsilon}-K\circledast\partial_{t}^{h}\psi_{\varepsilon} +\partial_{t}^{h}\beta_{\Gamma,\varepsilon}(\psi_{\varepsilon}) +\partial_{t}^{h}\pi_{\Gamma}(\psi_{\varepsilon})\Big)\zeta_{\Gamma} \,\mathrm{d}S,\label{quotient4}
   \end{align}
    for almost all $t\in(0,T)$ and for all test functions $(z,z_\Gamma)\in\mathcal{H}^1$, $(\zeta,\zeta_{\Gamma})\in \mathcal{L}^2$.
   Since $\overline{m}(\partial_{t}^{h}\boldsymbol{\varphi}_{\varepsilon})=0$,
   we can take $(z,z_{\Gamma})=\mathfrak{S}^{L}(\partial_{t}^{h}\boldsymbol{\varphi}_{\varepsilon})$ in \eqref{quotient1}, \eqref{quotient2}, respectively.
   Then, using \eqref{quotient3} and \eqref{quotient4}, we deduce that
   \begin{align}
   	0&=\frac{1}{2}\frac{\mathrm{d}}{\mathrm{d}t}\Vert\partial_{t}^{h}\boldsymbol{\varphi}_{\varepsilon}\Vert_{L,0,\ast}^{2}+\int_{\Omega}\partial_{t}^{h}\mu_{\varepsilon}\,\partial_{t}^{h}\varphi_{\varepsilon}\,\mathrm{d}x+\int_{\Gamma}\partial_{t}^{h}\theta_{\varepsilon}\,\partial_{t}^{h}\psi_{\varepsilon}\,\mathrm{d}S\notag\\
   	&=\frac{1}{2}\frac{\mathrm{d}}{\mathrm{d}t}\Vert\partial_{t}^{h} \boldsymbol{\varphi}_{\varepsilon}\Vert_{L,0,\ast}^{2} +\int_{\Omega}a_{\Omega}|\partial_{t}^{h}\varphi_{\varepsilon}|^{2}\,\mathrm{d}x+\int_{\Gamma}a_{\Gamma}|\partial_{t}^{h}\psi_{\varepsilon}|^{2}\,\mathrm{d}S\notag\\
   	&\quad-\int_{\Omega}(J\ast\partial_{t}^{h}\varphi_{\varepsilon}) \partial_{t}^{h}\varphi_{\varepsilon}\,\mathrm{d}x -\int_{\Gamma}(K\circledast\partial_{t}^{h}\psi_{\varepsilon}) \partial_{t}^{h}\psi_{\varepsilon}\,\mathrm{d}S\notag\\
   	&\quad+\int_{\Omega}\frac{\beta_{\varepsilon}(\varphi_{\varepsilon}(t+h)) -\beta_{\varepsilon}(\varphi_{\varepsilon}(t))}{h}\partial_{t}^{h} \varphi_{\varepsilon}\,\mathrm{d}x+\int_{\Omega}\frac{\pi(\varphi_{\varepsilon} (t+h))-\pi(\varphi_{\varepsilon}(t))}{h}\partial_{t}^{h}\varphi_{\varepsilon} \,\mathrm{d}x\notag\\
   	&\quad+\int_{\Gamma}\frac{\beta_{\Gamma,\varepsilon}(\psi_{\varepsilon}(t+h)) -\beta_{\Gamma,\varepsilon}(\psi_{\varepsilon}(t))}{h}\partial_{t}^{h} \psi_{\varepsilon}\,\mathrm{d}S+\int_{\Gamma}\frac{\pi_{\Gamma} (\psi_{\varepsilon}(t+h))-\pi_{\Gamma}(\psi_{\varepsilon}(t))}{h} \partial_{t}^{h}\psi_{\varepsilon}\,\mathrm{d}S\notag\\
   	&\geq\frac{1}{2}\frac{\mathrm{d}}{\mathrm{d}t}\Vert\partial_{t}^{h} \boldsymbol{\varphi}_{\varepsilon}\Vert_{L,0,\ast}^{2}+\Big(a_{\ast} +\frac{\alpha}{1+\alpha}-\gamma_{1}\Big)\Vert\partial_{t}^{h} \varphi_{\varepsilon}\Vert_{H}^{2}+\Big(a_{\circledast} +\frac{\alpha}{1+\alpha}-\gamma_{2}\Big)\Vert\partial_{t}^{h} \psi_{\varepsilon}\Vert_{H_{\Gamma}}^{2}\notag\\[1mm]
   	&\quad-\langle\partial_{t}^{h}\boldsymbol{\varphi}_{\varepsilon}, (J\ast\partial_{t}^{h}\varphi_{\varepsilon},K\circledast\partial_{t}^{h} \psi_{\varepsilon})\rangle_{(\mathcal{H}^{1})',\mathcal{H}^{1}}. \label{quotient5}
   \end{align}
   With the aid of Lemma \ref{lem-subsequent}, the last line of \eqref{quotient5} can be bounded as follows:
   \begin{align}
   &\langle\partial_{t}^{h}\boldsymbol{\varphi}_{\varepsilon},(J\ast\partial_{t}^{h}\varphi_{\varepsilon},K\circledast\partial_{t}^{h}\psi_{\varepsilon})\rangle_{(\mathcal{H}^{1})',\mathcal{H}^{1}}\notag\\
   &\quad\leq\Vert\partial_{t}^{h}\boldsymbol{\varphi}_{\varepsilon}\Vert_{(\mathcal{H}^{1})'}\Vert(J\ast\partial_{t}^{h}\varphi_{\varepsilon},K\circledast\partial_{t}^{h}\psi_{\varepsilon})\Vert_{\mathcal{H}^{1}}\notag\\
   &\quad\leq C\Vert\partial_{t}^{h}\boldsymbol{\varphi}_{\varepsilon}\Vert_{L,0,\ast}(\Vert J\ast\partial_{t}^{h}\varphi_{\varepsilon}\Vert_{V}+\Vert K\circledast\partial_{t}^{h}\psi_{\varepsilon}\Vert_{V_{\Gamma}})\notag\\
   &\quad\leq C\Vert\partial_{t}^{h}\boldsymbol{\varphi}_{\varepsilon}\Vert_{L,0,\ast}\Vert\partial_{t}^{h}\boldsymbol{\varphi}_{\varepsilon}\Vert_{\mathcal{L}^{2}}\notag\\
   &\quad\leq\epsilon\Vert\partial_{t}^{h} \boldsymbol{\varphi}_{\varepsilon}\Vert_{\mathcal{L}^{2}}^{2} +C(\epsilon)\Vert\partial_{t}^{h}\boldsymbol{\varphi}_{\varepsilon} \Vert_{L,0,\ast}^{2}.\label{quotient6}
   \end{align}
   Combining \eqref{quotient5} and \eqref{quotient6}, choosing $\epsilon>0$ sufficiently small,
   we find that there exist two positive constants $C_3$ and $C_4$,
   independent of $\varepsilon\in(0,\varepsilon^\ast)$ and $h\in(0,1)$, such that
   \begin{align}
   	\frac{\mathrm{d}}{\mathrm{d}t}\Vert\partial_{t}^{h} \boldsymbol{\varphi}_{\varepsilon}\Vert_{L,0,\ast}^{2} +C_{3}\Vert\partial_{t}^{h}\boldsymbol{\varphi}_{\varepsilon} \Vert_{\mathcal{L}^{2}}^{2}\leq C_{4}\Vert\partial_{t}^{h}\boldsymbol{\varphi}_{\varepsilon} \Vert_{L,0,\ast}^{2}.\label{quotient10}
   \end{align}
   Applying the uniform Gronwall inequality, we obtain that, for all $\tau>0$, it holds
   \begin{align}
   	\Big\Vert\frac{\boldsymbol{\varphi}_{\varepsilon}(t+\tau+h) -\boldsymbol{\varphi}_{\varepsilon}(t+\tau)}{h}\Big\Vert_{L,0,\ast}^{2}\leq \frac{e^{C\tau}}{\tau}\int_{t}^{t+\tau} \Big\Vert\frac{\boldsymbol{\varphi}_{\varepsilon}(s+h) -\boldsymbol{\varphi}_{\varepsilon}(s)}{h}\Big\Vert_{L,0,\ast}^{2} \,\mathrm{d}s,\quad\forall\,t\geq0.\label{quotient7}
   \end{align}
   From \eqref{approenergyineq}, we get $\partial_t\boldsymbol{\varphi}_{\varepsilon}\in L^{2}(0,+\infty;\mathcal{H}_{L,0}^{-1})$,
   then for almost all $t>0$, it holds
   \begin{align*}
   	\Big\Vert\frac{\boldsymbol{\varphi}_{\varepsilon}(s+h) -\boldsymbol{\varphi}_{\varepsilon}(s)}{h} \Big\Vert_{L,0,\ast}\leq\frac{1}{h}\int_{t}^{t+h} \Vert\partial_t\boldsymbol{\varphi}_{\varepsilon}(s)\Vert_{L,0,\ast} \,\mathrm{d}s\rightarrow\Vert\partial_t\boldsymbol{\varphi}_{\varepsilon}(t) \Vert_{L,0,\ast}\quad\text{as }h\rightarrow0^{+},
   \end{align*}
   and for all $t\geq0$, it holds
   \begin{align}
   \int_{t}^{t+\tau}\Big\Vert\frac{\boldsymbol{\varphi}_{\varepsilon}(s+h) -\boldsymbol{\varphi}_{\varepsilon}(s)}{h}\Big\Vert_{L,0,\ast}^{2} \,\mathrm{d}s\leq\int_{t}^{t+\tau+h}\Vert\partial_t \boldsymbol{\varphi}_{\varepsilon}(s)\Vert_{L,0,\ast}^{2} \,\mathrm{d}s.\label{quotient8}
   \end{align}
   By \eqref{quotient7}, \eqref{quotient8} and the change of variable $\eta:=t+\tau$, we obtain
   \begin{align}
   \Big\Vert\frac{\boldsymbol{\varphi}_{\varepsilon}(\eta+h)-\boldsymbol{\varphi}_{\varepsilon}(\eta)}{h}\Big\Vert_{L,0,\ast}^{2}\leq\frac{e^{C\tau}}{\tau}\int_{t}^{\eta+h}\Vert\partial_t\boldsymbol{\varphi}_{\varepsilon}(s)\Vert_{L,0,\ast}^{2}\,\mathrm{d}s\leq\frac{e^{C\tau}}{\tau}(E_{\varepsilon}(\boldsymbol{\varphi}_{0})-E_{\varepsilon}(\boldsymbol{\varphi}_{\varepsilon}(\eta+h))).\label{quotient9}
   \end{align}
   Recalling \eqref{ineq1}, we get
  \begin{align*}
   	E_{\varepsilon}(\boldsymbol{\varphi}_{0}) -E_{\varepsilon}(\boldsymbol{\varphi}_{\varepsilon}(\eta+h))\leq E(\boldsymbol{\varphi}_{0})+(\widetilde{C}+C_{1}+C_{2})(|\Omega|+|\Gamma|),
   \end{align*}
   which, combined with \eqref{quotient9}, implies that
  \begin{align}
    \Big\Vert\frac{\boldsymbol{\varphi}_{\varepsilon}(\eta+h)-\boldsymbol{\varphi}_{\varepsilon}(\eta)}{h}\Big\Vert_{L,0,\ast}^{2}\leq\frac{e^{C\tau}}{\tau}(E(\boldsymbol{\varphi}_{0})+(\widetilde{C}+C_{1}+C_{2})(|\Omega|+|\Gamma|)),\label{quotient11}
   \end{align}
   and the right-hand side of \eqref{quotient11} is independent of $h$.
   Letting $h\rightarrow0^+$ gives
   \begin{align*}
   	\partial_{t}^{h}\boldsymbol{\varphi}_{\varepsilon}\rightarrow \partial_t\boldsymbol{\varphi}_{\varepsilon}\qquad\text{weakly star in }L^{\infty}(\tau,+\infty;\mathcal{H}_{L,0}^{-1}),
   \end{align*}
   and the following estimate
   \begin{align}
   	\Vert\partial_t\boldsymbol{\varphi}_{\varepsilon} \Vert_{L^\infty(\tau,+\infty;\mathcal{H}^{-1}_{L,0})}^{2}\leq \frac{e^{C\tau}}{\tau}(E(\boldsymbol{\varphi}_{0}) +(\widetilde{C}+C_{1}+C_{2})(|\Omega|+|\Gamma|)).\label{quotient12}
   \end{align}
  Since the right-hand side of \eqref{quotient12} is independent of $\varepsilon\in(0,\varepsilon^\ast)$,
   this implies that as $\varepsilon\rightarrow0$ (in the sense of a subsequence), there holds
   \begin{align*}
   	\partial_t\boldsymbol{\varphi}_{\varepsilon}\rightarrow \partial_t\boldsymbol{\varphi}\qquad\text{weakly star in }L^{\infty}(\tau,+\infty;\mathcal{H}_{L,0}^{-1}),
   \end{align*}
   with the same estimate
   \begin{align}
   	\Vert\partial_t\boldsymbol{\varphi}\Vert_{L^\infty(\tau,+\infty;\mathcal{H}_{L,0}^{-1})}^{2}\leq \frac{e^{C\tau}}{\tau}(E(\boldsymbol{\varphi}_{0})+(\widetilde{C}+C_{1}+C_{2})(|\Omega|+|\Gamma|)).\notag
   \end{align}
   Finally, integrating \eqref{quotient10} over $[t,t+1]$ with respect to $s$, we obtain
   \begin{align*}
   \int_{t}^{t+1}\Vert\partial_{t}^{h}\boldsymbol{\varphi}_{\varepsilon}(s) \Vert_{\mathcal{L}^{2}}^{2}\,\mathrm{d}s\leq C\Big(\Vert\partial_{t}^{h}\boldsymbol{\varphi}_{\varepsilon}(t) \Vert_{L,0,\ast}^{2}+\int_{t}^{t+1}\Vert\partial_{t}^{h} \boldsymbol{\varphi}_{\varepsilon}(s)\Vert_{L,0,\ast}^{2}\,\mathrm{d}s\Big),
   \quad t\geq\tau.
   \end{align*}
   Letting $h\rightarrow0^+$ and $\varepsilon\rightarrow0$, we can conclude \eqref{Tuni3}, which implies that  $\partial_t\boldsymbol{\varphi}(s)\in\mathcal{L}^{2}$, for almost all $s\geq\tau$. Testing \eqref{weak3} by $\zeta=\partial_t\varphi(s)$ and \eqref{weak4} by $\zeta_{\Gamma}=\partial_t\psi(s)$, using \eqref{weak1} and \eqref{weak2}, we get
   \begin{align}
   	&\frac{\mathrm{d}}{\mathrm{ds}}E(\boldsymbol{\varphi}(s)) +\int_{\Omega}|\nabla\mu(s)|^{2}\,\mathrm{d}x +\int_{\Gamma}|\nabla_{\Gamma}\theta(s)|^{2}\,\mathrm{d}S +\frac{1}{L}\int_{\Gamma}|\theta(s)-\mu(s)|^{2}\,\mathrm{d}S=0,\quad\text{for a.a. }s\geq\tau.\label{Lv-1}
   \end{align}
   Integrating \eqref{Lv-1} over $[\tau,t]$ for any $0<\tau<t\leq T$, we obtain
   \begin{align}
   &E(\boldsymbol{\varphi}(t))-E(\boldsymbol{\varphi}(\tau))\notag\\
   &\quad=-\int_{\tau}^{t}\int_{\Omega}|\nabla\mu(s)|^{2}\,\mathrm{d}x\,\mathrm{d}s-\int_{\tau}^{t}\int_{\Gamma}|\nabla_{\Gamma}\theta(s)|^{2}\,\mathrm{d}S\,\mathrm{d}s-\frac{1}{L}\int_{\tau}^{t}\int_{\Gamma}|\theta(s)-\mu(s)|^{2}\,\mathrm{d}S\,\mathrm{d}s\notag\\
   &\quad=-\int_{\tau}^{t}\Vert\partial_t\boldsymbol{\varphi}(s)\Vert_{L,0,*}^{2}\,\mathrm{d}s.\label{Lv-2}
   \end{align}
  	Hence, for any given $\tau>0$, the mapping $t\mapsto E(\boldsymbol{\varphi}(t))$ is absolutely continuous and non-increasing for all $t\in[\tau,T]$.
  	From Remark \ref{L2-conti} and Lebesgue's dominated convergence theorem,
  	we have $\lim_{\tau\to0}E(\boldsymbol{\varphi}(\tau))= E(\boldsymbol{\varphi}(0))$.
  	This allows us to pass to the limit as $\tau\to0$ in \eqref{Lv-2}
  	to conclude the energy equality \eqref{energyeq} for the case with $L\in(0,+\infty)$.
  \hfill $\square$
  \medskip

   \textbf{Proof of Theorem \ref{contidependence} with $L\in(0,+\infty)$.}
   Suppose that $(\boldsymbol{\varphi}_{i},\boldsymbol{\mu}_{i})$ is the unique weak solution to problem \eqref{model1}--\eqref{psiini}
   corresponding to the initial datum $\boldsymbol{\varphi}_{0,i}$ $(i\in\{1,2\})$
   and $\overline{m}(\boldsymbol{\varphi}_{0,1})=\overline{m}(\boldsymbol{\varphi}_{0,2})=\overline{m}_{0}$.
   Denote
   $$
   \boldsymbol{\varphi}=(\varphi,\psi) :=(\varphi_{1}-\varphi_{2},\psi_{1}-\psi_{2}),
   \quad
   \boldsymbol{\varphi}_0:=\boldsymbol{\varphi}_{0,1} -\boldsymbol{\varphi}_{0,2},
   $$
   then $\varphi$ and $\psi$ satisfy \eqref{weak1} and \eqref{weak2}, with
   \begin{align}
   	&\mu:=\mu_{1}-\mu_{2}=a_{\Omega}\varphi-J\ast\varphi +\beta(\varphi_{1})-\beta(\varphi_{2}) +\pi(\varphi_{1})-\pi(\varphi_{2}),\quad\text{ a.e. in } Q_T,\label{difference-mu}\\[1mm]
   	&\theta:=\theta_{1}-\theta_{2}=a_{\Gamma}\psi-K\circledast\psi +\beta_{\Gamma}(\psi_{1})-\beta_{\Gamma}(\psi_{2}) +\pi_{\Gamma}(\psi_{1})-\pi_{\Gamma}(\psi_{2}),\quad\text{ a.e. on } \Sigma_T.\label{difference-theta}
   \end{align}
   Since $\overline{m}(\boldsymbol{\varphi})=0$,
   we can take $\boldsymbol{z}=\mathfrak{S}^{L}(\boldsymbol{\varphi})$ in \eqref{weak1} and \eqref{weak2}.
   By the definition of $\mathfrak{S}^{L}$, \eqref{difference-mu} and \eqref{difference-theta}, we find
   \begin{align}
   	0&=\frac{1}{2}\frac{\mathrm{d}}{\mathrm{d}t} \Vert\boldsymbol{\varphi}\Vert_{L,0,*}^{2}+\int_{\Omega}\mu\varphi\,\mathrm{d}x+\int_{\Gamma}\theta\psi\,\mathrm{d}S\notag\\
   	&=\frac{1}{2}\frac{\mathrm{d}}{\mathrm{d}t} \Vert\boldsymbol{\varphi}\Vert_{L,0,*}^{2}+\int_\Omega a_\Omega|\varphi|^2\,\mathrm{d}x+\int_{\Omega}(\beta(\varphi_{1}) -\beta(\varphi_{2}))\varphi\,\mathrm{d}x-\int_{\Omega}(J\ast\varphi) \varphi\,\mathrm{d}x+\int_{\Omega}(\pi(\varphi_{1})-\pi(\varphi_{2}))\varphi\,\mathrm{d}x\notag\\
   	&\quad+\int_\Gamma a_\Gamma |\psi|^2\,\mathrm{d}S+\int_{\Gamma}(\beta_{\Gamma}(\psi_{1}) -\beta_{\Gamma}(\psi_{2}))\psi\,\mathrm{d}S-\int_{\Gamma} (K\circledast\psi)\psi\,\mathrm{d}S+\int_{\Gamma}(\pi_{\Gamma}(\psi_{1}) -\pi_{\Gamma}(\psi_{2}))\psi\,\mathrm{d}S\notag\\
   	&\geq\frac{1}{2}\frac{\mathrm{d}}{\mathrm{d}t}\Vert\boldsymbol{\varphi} \Vert_{L,0,*}^{2}+(\alpha+a_{\ast}-\gamma_{1})\Vert\varphi\Vert_{H}^{2} +(\alpha+a_{\circledast}-\gamma_{2})\Vert\psi\Vert_{H_{\Gamma}}^{2} -\langle\boldsymbol{\varphi},(J\ast\varphi,K\circledast\psi) \rangle_{(\mathcal{H}^1)',\mathcal{H}^1},\notag
   \end{align}
   which yields
   \begin{align}
   	&\frac{1}{2}\frac{\mathrm{d}}{\mathrm{d}t} \Vert\boldsymbol{\varphi}\Vert_{L,0,*}^{2} +(\alpha+a_{\ast}-\gamma_{1})\Vert\varphi\Vert_{H}^{2} +(\alpha+a_{\circledast}-\gamma_{2})\Vert\psi\Vert_{H_{\Gamma}}^{2}\notag\\
   	&\quad\leq\Vert\boldsymbol{\varphi}\Vert_{(\mathcal{H}^1)'}\Vert (J\ast\varphi,K\circledast\psi)\Vert_{\mathcal{H}^1}\notag\\
   	&\quad\leq C\Vert\boldsymbol{\varphi}\Vert_{L,0,\ast} \|\boldsymbol{\varphi}\|_{\mathcal{L}^2}\notag\\
   	&\quad\leq C\Vert\boldsymbol{\varphi}\Vert_{L,0,*}^{2} +\frac{\alpha+a_{\ast}-\gamma_{1}}{2}\Vert\varphi\Vert_{H}^{2} +\frac{\alpha+a_{\circledast}-\gamma_{2}}{2} \Vert\psi\Vert_{H_{\Gamma}}^{2}.\notag
   \end{align}
   Applying Gronwall's inequality, we can conclude that
   \begin{align*}
   \Vert\boldsymbol{\varphi}\Vert_{L^{\infty}(0,T;\mathcal{H}_{L,0}^{-1})}+\Vert\boldsymbol{\varphi}\Vert_{L^{2}(0,T;\mathcal{L}^{2})}\leq C\Vert\boldsymbol{\varphi}_{0}\Vert_{L,0,*}.
   \end{align*}
   This leads to the continuous dependence estimate \eqref{contiesti} for the case with $L\in(0,+\infty)$.
   \hfill $\square$

   \subsection{Convergence rate of the Yosida approximation}
   Inspired by \cite{GST}, we establish convergence rate of the Yosida approximation as $\varepsilon\to 0$.
   To this end, we introduce the linear operator $\mathbb{J}:\,\mathcal{L}^{2}\rightarrow\mathcal{L}^{2}$ as follows
   \begin{align*}
   	\mathbb{J}(\boldsymbol{z})(x,y) :=\Big(\int_{\Omega}J(x-p)z(p)\,\mathrm{d}p,\int_{\Gamma}K(y-q)z_{\Gamma}(q) \,\mathrm{d}S_{q}\Big),\quad\forall\,(x,y)\in\Omega\times\Gamma, \quad\forall\,\boldsymbol{z}\in\mathcal{L}^{2}.
   \end{align*}
   Clearly, $\mathbb{J}$ is bounded and self-adjoint.
   According to \cite{KL}, $\mathfrak{S}^{L}$ admits eigenfunctions $\{\boldsymbol{\phi}_{k}\}_{k\geq1}\subset\mathcal{L}^{2}$ that form an orthonormal basis of $\mathcal{L}^{2}$.
   Besides, the following lemma will be useful to derive the convergence rate.
   \begin{lemma}
   	\label{HS}
   	The operator $\mathbb{J}:\,\mathcal{L}^{2}\rightarrow\mathcal{L}^{2}$ is a Hilbert--Schmidt operator with
   	$$
   \sum_{k=1}^{+\infty}\Vert\mathbb{J}(\boldsymbol{\phi}_{k}) \Vert_{\mathcal{L}^{2}}^{2}<+\infty.
   $$
   \end{lemma}
   \begin{proof}
   Recall that the eigenfunctions
    $\{u_{k}\}_{k\in\mathbb{Z}^+}$, $\{v_{k}\}_{k\in\mathbb{Z}^+}$
   form an orthonormal basis of $H$ and $H_{\Gamma}$, respectively.
   Define $\boldsymbol{u}_{k}=(u_{k},v_{k})$. Since $J\in L^2(\mathbb{R}^d)$ and $K\in W^{2,r}(\mathbb{R}^d)$ with $r>1$,
   we have $J(x-\cdot)\in H$ and $K(y-\cdot)\in H_\Gamma$ for all $x\in\Omega$ and $y\in\Gamma$. In particular, it holds
   \begin{align*}
   	J(x-\cdot)=\sum_{k=1}^{+\infty}(J(x-\cdot),u_{k})_{H}u_{k},\quad K(y-\cdot)=\sum_{k=1}^{+\infty}(K(y-\cdot),v_{k})_{H_{\Gamma}}v_{k}.
   \end{align*}
   Then we have
   \begin{align*}
   	\sum_{k=1}^{+\infty}\Vert\mathbb{J}(\boldsymbol{u}_{k}) \Vert_{\mathcal{L}^{2}}^{2}&=\sum_{k=1}^{+\infty}\int_{\Omega} \Big(\int_{\Omega}J(x-p)u_{k}(p)\,\mathrm{d}p\Big)^{2}\,\mathrm{d}x +\sum_{k=1}^{+\infty}\int_{\Gamma}\Big(\int_{\Gamma}K(y-q)v_{k}(q) \,\mathrm{d}S_{q}\Big)^{2}\,\mathrm{d}S_{y}\\
   	&=\sum_{k=1}^{+\infty}\int_{\Omega}|(J(x-\cdot),u_{k})_{H}|^{2} \,\mathrm{d}x+\sum_{k=1}^{+\infty}\int_{\Gamma} |(K(y-\cdot),v_{k})_{H_\Gamma}|^{2}\,\mathrm{d}S_{y}\\
   	&=\int_{\Omega}\sum_{k=1}^{+\infty}|(J(x-\cdot),u_{k})_{H}|^{2} \,\mathrm{d}x+\int_{\Gamma}\sum_{k=1}^{+\infty}|(K(y-\cdot),v_{k} )_{H_\Gamma}|^{2}\,\mathrm{d}S_{y}\\
   	&=\int_{\Omega}\Vert J(x-\cdot)\Vert_{H}^{2}\,\mathrm{d}x+\int_{\Gamma}\Vert K(y-\cdot)\Vert_{H_{\Gamma}}^{2}\,\mathrm{d}S_{y}\\
   	&\leq C(|\Omega|\Vert J\Vert_{L^{2}(\mathbb{R}^{d})}^{2}+|\Gamma|\Vert K\Vert_{W^{2,r}(\mathbb{R}^{d})}^{2})<+\infty.
   \end{align*}
   This implies
   \begin{align*}
   \sum_{k=1}^{+\infty}\Vert\mathbb{J}(\boldsymbol{\phi}_{k}) \Vert_{\mathcal{L}^{2}}^{2}=\sum_{k=1}^{+\infty}\sum_{j=1}^{+\infty} |(\mathbb{J}(\boldsymbol{\phi}_{k}),\boldsymbol{u}_{j} )_{\mathcal{L}^{2}}|^{2}=\sum_{j=1}^{+\infty} \sum_{k=1}^{+\infty}|(\boldsymbol{\phi}_{k}, \mathbb{J}(\boldsymbol{u}_{j}))_{\mathcal{L}^{2}}|^{2} =\sum_{j=1}^{+\infty}\Vert\mathbb{J}(\boldsymbol{u}_{j}) \Vert_{\mathcal{L}^{2}}^{2}<+\infty.
   \end{align*}
   The proof of Lemma \ref{HS} is complete.
   \end{proof}

   \noindent \textbf{Proof of Theorem \ref{rate}.}
   We first consider the approximating problem \eqref{appro1}--\eqref{psiappro}. Let $\boldsymbol{\varphi}_{\varepsilon_{i}}:=(\varphi_{\varepsilon_{i}},\psi_{\varepsilon_{i}})$, $\boldsymbol{\mu}_{\varepsilon_{i}}:=(\mu_{\varepsilon_{i}},\theta_{\varepsilon_{i}})$ be the unique weak solution to problem \eqref{appro1}--\eqref{psiappro}
   corresponding to the parameter $\varepsilon_{i}$ $(i\in\{1,2\})$.
   We denote their differences by $\boldsymbol{\varphi}^{\sharp}:=\boldsymbol{\varphi}_{\varepsilon_{1}} -\boldsymbol{\varphi}_{\varepsilon_{2}}$, $\boldsymbol{\mu}^{\sharp}:=\boldsymbol{\mu}_{\varepsilon_{1}} -\boldsymbol{\mu}_{\varepsilon_{2}}$. Then $(\boldsymbol{\varphi}^{\sharp},\boldsymbol{\mu}^{\sharp})$ satisfies the following weak formulations
   \begin{align}
   	\langle\partial_{t}\boldsymbol{\varphi}^{\sharp},\boldsymbol{z} \rangle_{(\mathcal{H}^{1})',\mathcal{H}^{1}} &=-\int_{\Omega}\nabla\mu^{\sharp}\cdot\nabla z\,\mathrm{d}x-\int_{\Gamma}\nabla_{\Gamma}\theta^{\sharp} \cdot\nabla_{\Gamma}z_{\Gamma}\,\mathrm{d}S-\frac{1}{L}\int_{\Gamma} (\theta^{\sharp}-\mu^{\sharp})(z_{\Gamma}-z)\,\mathrm{d}S,\label{error1}\\
   	\int_{\Omega}\mu^{\sharp}\zeta\,\mathrm{d}x&=\int_{\Omega} a_{\Omega}\varphi^{\sharp}\zeta\,\mathrm{d}x-\int_{\Omega} (J\ast\varphi^{\sharp})\zeta\,\mathrm{d}x+\int_{\Omega} (\beta_{\varepsilon_{1}}(\varphi_{\varepsilon_{1}}) -\beta_{\varepsilon_{2}}(\varphi_{\varepsilon_{2}}))\zeta \,\mathrm{d}x\notag\\
   	&\quad+\int_{\Omega}(\pi(\varphi_{\varepsilon_{1}}) -\pi(\varphi_{\varepsilon_{2}}))\zeta\,\mathrm{d}x,\label{error2}\\
   	\int_{\Gamma}\theta^{\sharp}\zeta_{\Gamma}\,\mathrm{d}S &=\int_{\Gamma}a_{\Gamma}\psi^{\sharp}\zeta_{\Gamma}\,\mathrm{d}S -\int_{\Gamma}(K\circledast\psi^{\sharp})\zeta_{\Gamma} \,\mathrm{d}S+\int_{\Gamma}(\beta_{\Gamma,\varepsilon_{1}} (\psi_{\varepsilon_{1}})-\beta_{\Gamma,\varepsilon_{2}} (\psi_{\varepsilon_{2}}))\zeta_{\Gamma}\,\mathrm{d}S\notag\\
   	&\quad+\int_{\Gamma}(\pi_{\Gamma}(\psi_{\varepsilon_{1}}) -\pi_{\Gamma}(\psi_{\varepsilon_{2}}))\zeta_{\Gamma}\,\mathrm{d}S, \label{error3}
   \end{align}
   for almost every $t\in[0,T]$ and any test functions $\boldsymbol{z}=(z,z_{\Gamma})\in\mathcal{H}^{1}$, $\boldsymbol{\zeta}=(\zeta,\zeta_{\Gamma})\in \mathcal{L}^2$.
   Since $\overline{m}(\boldsymbol{\varphi}^\sharp)=0$,
   we can take $\boldsymbol{z}=\mathfrak{S}^{L}(\boldsymbol{\varphi}^{\sharp})$ in \eqref{error1}
    and then integrate over $(0,t)$, using \eqref{error2}, \eqref{error3}, we obtain
   \begin{align}
   	&\frac{1}{2}\Vert\boldsymbol{\varphi}^{\sharp}(t)\Vert_{L,0,*}^{2} +\int_{0}^{t}\int_{\Omega}a_{\Omega}|\varphi^{\sharp}(s)|^{2} \,\mathrm{d}x\,\mathrm{d}s+\int_{0}^{t}\int_{\Gamma}a_{\Gamma} |\psi^{\sharp}(s)|^{2}\,\mathrm{d}S\,\mathrm{d}s\notag\\
   	&\quad=\int_{0}^{t}\int_{\Omega}(J\ast\varphi^{\sharp}(s)) \varphi^{\sharp}(s)\,\mathrm{d}x\,\mathrm{d}s +\int_{0}^{t}\int_{\Gamma}(K\circledast\psi^{\sharp}(s)) \psi^{\sharp}(s)\,\mathrm{d}S\,\mathrm{d}s\notag\\
   	&\qquad+\int_{0}^{t}\int_{\Omega}(-\beta_{\varepsilon_{1}} (\varphi_{\varepsilon_{1}}(s))+\beta_{\varepsilon_{2}} (\varphi_{\varepsilon_{2}}(s)))\varphi^{\sharp}(s) \,\mathrm{d}x\,\mathrm{d}s\notag\\
   	&\qquad+\int_{0}^{t}\int_{\Gamma}(-\beta_{\Gamma,\varepsilon_{1}} (\psi_{\varepsilon_{1}}(s))+\beta_{\Gamma,\varepsilon_{2}} (\psi_{\varepsilon_{2}}(s)))\psi^{\sharp}(s) \,\mathrm{d}S\,\mathrm{d}s\notag\\
   	&\qquad+\int_{0}^{t}\int_{\Omega}(-\pi(\varphi_{\varepsilon_{1}}(s)) +\pi(\varphi_{\varepsilon_{2}}(s)))\varphi^{\sharp}(s) \,\mathrm{d}x\,\mathrm{d}s\notag\\
   	&\qquad+\int_0^t\int_{\Gamma}(-\pi_{\Gamma}(\psi_{\varepsilon_{1}}(s) )+\pi_{\Gamma}(\psi_{\varepsilon_{2}}(s)))\psi^{\sharp}(s) \,\mathrm{d}S\,\mathrm{d}s\notag\\
   	&\quad=\sum_{i=1}^6 H_{i}.\label{error4}
   \end{align}
   Thanks to $\mathbf{(A1)}$, we have
   \begin{align}
   	\int_{0}^{t}\int_{\Omega}a_{\Omega}|\varphi^{\sharp}(s)|^{2} \,\mathrm{d}x\,\mathrm{d}s+\int_{0}^{t}\int_{\Gamma}a_{\Gamma} |\psi^{\sharp}(s)|^{2}\,\mathrm{d}S\,\mathrm{d}s\geq a_{\ast}\Vert\varphi^{\sharp}\Vert_{L^{2}(0,t;H)}^{2}+a_{\circledast} \Vert\psi^{\sharp}\Vert_{L^{2}(0,t;H_{\Gamma})}^{2}.\label{error5}
   \end{align}
   For $H_{3}$ and $H_{4}$, we can follow the standard procedure (see, e.g., \cite{CG00,GST,LvWu}) to conclude
   \begin{align}
   	H_{3}+H_{4}\leq C(\varepsilon_{1}+\varepsilon_{2}).\label{error6}
   \end{align}
   Besides, from $\mathbf{(A3)}$ we infer that
   \begin{align}
   	|H_{5}|+|H_{6}|&\leq\gamma_{1}\Vert\varphi^{\sharp} \Vert_{L^{2}(0,t;H)}^{2}+\gamma_{2}\Vert\psi^{\sharp} \Vert_{L^{2}(0,t;H_{\Gamma})}^{2}.\label{error7}
   \end{align}
   Finally, we estimate $H_{1}+H_2$. Since $\{\boldsymbol{\phi}_{k}\}_{k\geq1}\subset\mathcal{L}^{2}$ form an orthonormal basis of $\mathcal{L}^{2}$,
   then $\mathbb{J}(\boldsymbol{\varphi}^{\sharp})$ can be expanded as
   \begin{align}
   \mathbb{J}(\boldsymbol{\varphi}^{\sharp})=\sum_{k=1}^{+\infty}(\mathbb{J}(\boldsymbol{\varphi}^{\sharp}),\boldsymbol{\phi}_{k})_{\mathcal{L}^{2}}\boldsymbol{\phi}_{k}=\underbrace{\sum_{k=1}^{N}(\mathbb{J}(\boldsymbol{\varphi}^{\sharp}),\boldsymbol{\phi}_{k})_{\mathcal{L}^{2}}\boldsymbol{\phi}_{k}}_{\mathbb{J}_{N}(\boldsymbol{\varphi}^{\sharp})}+\underbrace{\sum_{k=N+1}^{+\infty}(\mathbb{J}(\boldsymbol{\varphi}^{\sharp}),\boldsymbol{\phi}_{k})_{\mathcal{L}^{2}}\boldsymbol{\phi}_{k}}_{(\mathbb{J}-\mathbb{J}_{N})(\boldsymbol{\varphi}^{\sharp})},\notag
   \end{align}
   and, similarly, $\boldsymbol{\varphi}^{\sharp}$ can be written as
   \begin{align}
   	\boldsymbol{\varphi}^{\sharp}=\boldsymbol{\varphi}^{\sharp}_{A} +\boldsymbol{\varphi}^{\sharp}_{B},\quad\text{with }\boldsymbol{\varphi}^{\sharp}_{A}\in\,A_{N} :=\text{span}\{\boldsymbol{\phi}_{1},...,\boldsymbol{\phi}_{N}\}, \,\,\boldsymbol{\varphi}^{\sharp}_{B}\in\,B_{N}:=\text{span} \{\boldsymbol{\phi}_{N+1},...\}.\notag
   \end{align}
  It follows that
   \begin{align}
   	H_{1}+H_2&=\int_{0}^{t}(\mathbb{J}(\boldsymbol{\varphi}^{\sharp}(s)), \boldsymbol{\varphi}^{\sharp}(s))_{\mathcal{L}^{2}}\,\mathrm{d}s\notag\\
   	&=\int_{0}^{t}(\mathbb{J}_{N}(\boldsymbol{\varphi}^{\sharp}(s)), \boldsymbol{\varphi}^{\sharp}(s))_{\mathcal{L}^{2}}\,\mathrm{d}s+\int_{0}^{t}((\mathbb{J}-\mathbb{J}_{N})(\boldsymbol{\varphi}^{\sharp}(s)),\boldsymbol{\varphi}^{\sharp}(s))_{\mathcal{L}^{2}}\,\mathrm{d}s\notag\\
   	&=\int_{0}^{t}(\mathbb{J}_{N}(\boldsymbol{\varphi}^{\sharp}(s)), \boldsymbol{\varphi}^{\sharp}_{A}(s))_{\mathcal{L}^{2}}\,\mathrm{d}s+\int_{0}^{t}((\mathbb{J}-\mathbb{J}_{N})(\boldsymbol{\varphi}^{\sharp}(s)),\boldsymbol{\varphi}^{\sharp}(s))_{\mathcal{L}^{2}}\,\mathrm{d}s\notag\\
   	&\leq\Vert\mathbb{J}_{N}\Vert_{\mathcal{B}(\mathcal{L}^{2})} \Vert\boldsymbol{\varphi}^{\sharp}\Vert_{L^{2}(0,t;\mathcal{L}^{2})} \Vert\boldsymbol{\varphi}_{A}^{\sharp}\Vert_{L^{2}(0,t;\mathcal{L}^{2})} +\Vert\mathbb{J}-\mathbb{J}_{N}\Vert_{\mathcal{B}(\mathcal{L}^{2})} \Vert\boldsymbol{\varphi}^{\sharp}\Vert_{L^{2}(0,t;\mathcal{L}^{2})}^{2}\notag\\
   	&\leq\Vert\mathbb{J}\Vert_{\mathcal{B}(\mathcal{L}^{2})} \Vert\boldsymbol{\varphi}^{\sharp}\Vert_{L^{2}(0,t;\mathcal{L}^{2})} \Vert\boldsymbol{\varphi}_{A}^{\sharp}\Vert_{L^{2}(0,t;\mathcal{L}^{2})} +\Vert\mathbb{J}-\mathbb{J}_{N}\Vert_{\mathcal{B}(\mathcal{L}^{2})} \Vert\boldsymbol{\varphi}^{\sharp}\Vert_{L^{2}(0,t;\mathcal{L}^{2})}^{2} \notag\\
   &\leq(\epsilon\Vert\mathbb{J}\Vert_{\mathcal{B}(\mathcal{L}^{2})}^{2} +\Vert\mathbb{J}-\mathbb{J}_{N}\Vert_{\mathcal{B}(\mathcal{L}^{2})}) \Vert\boldsymbol{\varphi}^{\sharp}\Vert_{L^{2}(0,t;\mathcal{L}^{2})}^{2} +C(\epsilon)\Vert\boldsymbol{\varphi}_{A}^{\sharp}\Vert_{L^{2} (0,t;\mathcal{L}^{2})}^{2},\label{error10}
   \end{align}
   where the notation $\Vert\cdot\Vert_{\mathcal{B}(\mathcal{L}^{2})}$
   denotes the operator norm of linear bounded operators from $\mathcal{L}^{2}$ to $\mathcal{L}^{2}$.
   For the last term on the right-hand side of \eqref{error10},
   since $A_{N}$ is a finite dimensional space, the norms on $A_{N}$ are equivalent, namely,
   \begin{align}
   \Vert\boldsymbol{\varphi}_{A}^{\sharp}\Vert_{L^{2}(0,t;\mathcal{L}^{2})}^{2}&\leq C(N)\Vert\boldsymbol{\varphi}_{A}^{\sharp}\Vert_{L^{2}(0,t;\mathcal{H}_{L,0}^{-1})}^{2}\notag\\
   &=C(N)\Vert\mathfrak{S}^{L}(\boldsymbol{\varphi}_{A}^{\sharp})\Vert_{L^{2}(0,t;\mathcal{H}_{L,0}^{1})}^{2}\notag\\
   &\leq C(N)\Vert\mathfrak{S}^{L}(\boldsymbol{\varphi}^{\sharp})\Vert_{L^{2}(0,t;\mathcal{H}_{L,0}^{1})}^{2}\notag\\
   &=C(N)\Vert\boldsymbol{\varphi}^{\sharp}\Vert_{L^{2}(0,t;\mathcal{H}_{L,0}^{-1})}^{2}.\label{error11}
   \end{align}
   Collecting \eqref{error4}--\eqref{error11}, we arrive at
   \begin{align}
   	&\frac{1}{2}\Vert\boldsymbol{\varphi}^{\sharp}\Vert_{L,0,*}^{2} +(a_{\ast}-\gamma_{1}-\epsilon\Vert\mathbb{J}\Vert_{\mathcal{B}(\mathcal{L}^{2})}^{2}-\Vert\mathbb{J}-\mathbb{J}_{N}\Vert_{\mathcal{B}(\mathcal{L}^{2})})\Vert\varphi^{\sharp}\Vert_{L^{2}(0,T;H)}^{2}\notag\\
   	&\qquad+(a_{\circledast}-\gamma_{2}-\epsilon\Vert\mathbb{J} \Vert_{\mathcal{B}(\mathcal{L}^{2})}^{2}-\Vert\mathbb{J}-\mathbb{J}_{N}\Vert_{\mathcal{B}(\mathcal{L}^{2})})\Vert\psi^{\sharp}\Vert_{L^{2}(0,T;H_{\Gamma})}^{2}\notag\\
   	&\quad\leq C(\epsilon)C(N)\int_{0}^{t}\Vert\boldsymbol{\varphi}^{\sharp}(s) \Vert_{L,0,*}^{2}\,\mathrm{d}s+C(\varepsilon_{1}+\varepsilon_{2}). \label{error12}
   \end{align}
   As a consequence of Lemma \ref{HS}, it holds $\Vert\mathbb{J}-\mathbb{J}_{N} \Vert_{\mathcal{B}(\mathcal{L}^{2})}\rightarrow0$ as $N\rightarrow+\infty$.
   We can first choose a sufficiently large $N\in\mathbb{Z}^+$,
   then a sufficiently small $\epsilon>0$ such that
   \begin{align}
   	&a_{\ast}-\gamma_{1}-\epsilon\Vert\mathbb{J} \Vert_{\mathcal{B}(\mathcal{L}^{2})}^{2}-\Vert\mathbb{J}-\mathbb{J}_{N} \Vert_{\mathcal{B}(\mathcal{L}^{2})}>0,\notag\\
   	&a_{\circledast}-\gamma_{2}-\epsilon\Vert\mathbb{J} \Vert_{\mathcal{B}(\mathcal{L}^{2})}^{2}-\Vert\mathbb{J}-\mathbb{J}_{N} \Vert_{\mathcal{B}(\mathcal{L}^{2})}>0.\notag
   \end{align}
   Then, from \eqref{error12} and Gronwall's inequality, we can conclude that
   \begin{align}
   	\Vert\boldsymbol{\varphi}_{\varepsilon_{1}} -\boldsymbol{\varphi}_{\varepsilon_{2}} \Vert_{L^{\infty}(0,T;\mathcal{H}_{L,0}^{-1})} +\Vert\boldsymbol{\varphi}_{\varepsilon_{1}} -\boldsymbol{\varphi}_{\varepsilon_{2}}\Vert_{L^{2}(0,T;\mathcal{L}^{2})} \leq C\sqrt{\varepsilon_{1}+\varepsilon_{2}}.\label{error15}
   \end{align}
   Passing to the limit as $\varepsilon_{2}\rightarrow0$ in \eqref{error15},
   using the convergence results \eqref{approconver1}, \eqref{approconver2}
   and the weak lower semicontinuity of norms, we find
   \begin{align*}
   	\Vert\boldsymbol{\varphi}_{\varepsilon_{1}} -\boldsymbol{\varphi}\Vert_{L^{\infty}(0,T;\mathcal{H}_{L,0}^{-1})} +\Vert\boldsymbol{\varphi}_{\varepsilon_{1}} -\boldsymbol{\varphi}\Vert_{L^{2}(0,T;\mathcal{L}^{2})}\leq C\sqrt{\varepsilon_{1}},
   \end{align*}
   which completes the proof of Theorem \ref{rate}.
   \hfill $\square$

   \section{Asymptotic Limits as $L\to 0$ and $L\to+\infty$}
   \setcounter{equation}{0}
   In this section, we first prove Theorems \ref{exist} and \ref{contidependence} for the limiting cases $L=0$ and $L=+\infty$.
   The existence of weak solutions follows from the asymptotic limits as $L\rightarrow0$ and $L\rightarrow+\infty$, respectively,
   whereas the continuous dependence on the initial data can be carried out by a standard energy method.
   The second part is devoted to the proof of Theorem \ref{weak-convergence} on the convergence rate of asymptotic limits.

   In the following, we denote by $(\boldsymbol{\varphi}_{\varepsilon}^{L},\boldsymbol{\mu}_{\varepsilon}^{L})$
    the unique weak solution to the approximating problem \eqref{appro1}--\eqref{psiappro}
   corresponding to the pair of parameters $(\varepsilon,L)\in(0,\varepsilon^\ast)\times(0,+\infty)$ obtained in Proposition \ref{approexist},
   and by $(\boldsymbol{\varphi}^{L},\boldsymbol{\mu}^{L})$ the unique weak solution to the original problem  \eqref{model1}--\eqref{psiini}
   corresponding to $L\in(0,+\infty)$ obtained in Theorem \ref{exist}. Owing to \eqref{phiinfty}, \eqref{psiinfty} and the energy equality \eqref{energyeq}, it holds
   \begin{align}
   	&\|\boldsymbol{\varphi}^{L}\|_{L^{\infty}(0,T;\mathcal{L}^{2})} +\Vert\nabla\mu^{L}\Vert_{L^{2}(0,T;H)} +\Vert\nabla_{\Gamma}\theta^{L}\Vert_{L^{2}(0,T;H_{\Gamma})} +\frac{1}{\sqrt{L}}\Vert\theta^{L}-\mu^{L}\Vert_{L^{2}(0,T;H_{\Gamma})}\leq C.\label{Luni1}
   \end{align}
   Besides, according to the weak formulations \eqref{weak1} and \eqref{weak2}, we can obtain
   \begin{align}
   	&\Vert\partial_{t}\varphi^{L}\Vert_{L^{2}(0,T;V')}\leq \Vert\nabla\mu^{L}\Vert_{L^{2}(0,T;H)} +\frac{1}{\sqrt{L}}\cdot\frac{1}{\sqrt{L}}\Vert\theta^{L}-\mu^{L} \Vert_{L^{2}(0,T;H_{\Gamma})}\leq C\left(1+\frac{1}{\sqrt{L}}\right),\label{0uni1}\\
   	&\Vert\partial_{t}\psi^{L}\Vert_{L^{2}(0,T;V_{\Gamma}')}\leq \Vert\nabla_{\Gamma}\theta^{L}\Vert_{L^{2}(0,T;H_{\Gamma})} +\frac{1}{\sqrt{L}}\cdot\frac{1}{\sqrt{L}}\Vert\theta^{L}-\mu^{L} \Vert_{L^{2}(0,T;H_{\Gamma})}\leq C\left(1+\frac{1}{\sqrt{L}}\right),\label{0uni2}\\
   	&\Vert\partial_{t}\boldsymbol{\varphi}^{L}\Vert_{L^{2}(0,T;(\mathcal{V}^{1})')} \leq\Vert\nabla\mu^{L}\Vert_{L^{2}(0,T;H)} +\Vert\nabla_{\Gamma}\theta^{L}\Vert_{L^{2}(0,T;H_{\Gamma})}\leq C,\label{0uni3}
   \end{align}
   where the constant $C>0$ is independent of $L\in(0,+\infty)$.

   \subsection{The case $L\to0$}
   \begin{lemma}
   	\label{L0uni1}
   	Suppose $L\in(0,1]$. Then there exists a positive constant $C$, independent of $L\in(0,1]$, such that
   	\begin{align}
   		&\Vert\boldsymbol{\varphi}^{L}\Vert_{L^{\infty}(0,T;\mathcal{L}^{2})} +\Vert\boldsymbol{\varphi}^{L}\Vert_{L^{2}(0,T;\mathcal{H}^{1})} +\Vert\boldsymbol{\mu}^{L}\Vert_{L^{2}(0,T;\mathcal{H}^{1})}\notag\\
   		&\quad+	\Vert \beta(\varphi^{L})\Vert_{L^{2}(0,T;H)}+\Vert \beta_{\Gamma}(\psi^{L})\Vert_{L^{2}(0,T;H_{\Gamma})}\leq C.\label{0uni5}
   	\end{align}
   \end{lemma}
   \begin{proof}
   For $L\in(0,1]$ we can derive a refined estimate of the first line on the right-hand side of  \eqref{approuni2}.
   Observe that the estimate \eqref{approuni1} is uniform
   with respect to $\varepsilon\in(0,\varepsilon^\ast)$ and $L\in(0,+\infty)$.
   Recalling \cite[Lemma 5.1]{LvWu}, we find there exists a constant $\widehat{C}>0$,
   independent of $L\in(0,1]$, such that
   \begin{align}
   	\Vert\boldsymbol{z}\Vert_{L,0,*} \leq\widehat{C}\Vert\boldsymbol{z}\Vert_{\mathcal{L}^{2}},\quad\text{for all }\boldsymbol{z}\in\mathcal{L}_{(0)}^{2}.\label{Lv0905}
   \end{align}
   Thank to \eqref{approuni1} and \eqref{Lv0905}, the first line on the right-hand side of \eqref{approuni2} can be estimated as follows
   \begin{align}
  	&\int_{\Omega}\mu^{L}_{\varepsilon}(\varphi^{L}_{\varepsilon} -\overline{m}_{0})\,\mathrm{d}x+\int_{\Gamma}\theta^{L}_{\varepsilon}(\psi^{L}_{\varepsilon}-\overline{m}_{0})\,\mathrm{d}S\notag\\
  	&\quad=\langle\boldsymbol{\varphi}_{\varepsilon}^{L} -\overline{m}_{0}\boldsymbol{1},\boldsymbol{\mu}_{\varepsilon}^{L}-\overline{m}(\boldsymbol{\mu}_{\varepsilon}^{L})\rangle_{\mathcal{H}_{L,0}^{-1},\mathcal{H}_{L,0}^{1}}\notag\\
  	&\quad\leq\Vert\boldsymbol{\varphi}_{\varepsilon}^{L} -\overline{m}_{0}\boldsymbol{1}\Vert_{L,0,*}\Vert\boldsymbol{\mu}_{\varepsilon}^{L}-\overline{m}(\boldsymbol{\mu}_{\varepsilon}^{L})\Vert_{\mathcal{H}_{L,0}^{1}}\notag\\
  	&\quad\leq \widehat{C}\Vert\boldsymbol{\varphi}_{\varepsilon}^{L} -\overline{m}_{0}\boldsymbol{1}\Vert_{\mathcal{L}^{2}}\Vert\boldsymbol{\mu}_{\varepsilon}^{L}-\overline{m}(\boldsymbol{\mu}_{\varepsilon}^{L})\Vert_{\mathcal{H}_{L,0}^{1}}\notag\\
  	&\quad\leq C\Vert\boldsymbol{\mu}_{\varepsilon}^{L} -\overline{m}(\boldsymbol{\mu}_{\varepsilon}^{L})\Vert_{\mathcal{H}_{L,0}^{1}} ,\notag
  \end{align}
  where the constant $C>0$ is independent of $\varepsilon\in(0,\varepsilon^\ast)$ and $L\in(0,1]$.
  Thus, repeating the process for \eqref{approuni2}--\eqref{approuni11}, and passing to the limit as $\varepsilon\rightarrow0$, we can obtain the uniform estimate \eqref{0uni5}.
   \end{proof}

    \textbf{Proof of Theorems \ref{exist} and \ref{contidependence} with $L=0$.}
   	From \eqref{Luni1}, \eqref{0uni3}, \eqref{0uni5} and Aubin--Lions--Simon lemma (see Lemma \ref{ALS}),
   	we can conclude that there exist functions $\boldsymbol{\varphi}^{0}:=(\varphi^{0}$, $\psi^{0})$, $\boldsymbol{\mu}^{0}:=(\mu^{0}$, $\theta^{0})$ and $\boldsymbol{\xi}^{0}:=(\xi^{0},\xi^{0}_{\Gamma})$ such that
   	\begin{align}
   		\boldsymbol{\varphi}^{L}&\rightarrow\boldsymbol{\varphi}^{0}&&\text{weakly in }H^{1}(0,T;(\mathcal{V}^{1})')\cap L^{2}(0,T;\mathcal{H}^{1}),\label{0conver1}\\
   		& &&\text{strongly in }C([0,T];\mathcal{L}^{2}),\notag\\
   		\mu^{L}&\rightarrow\mu^{0}&&\text{weakly in }L^{2}(0,T;V),\label{0conver2}\\
   		\theta^{L}&\rightarrow\theta^{0}&&\text{weakly in }L^{2}(0,T;V_{\Gamma}),\label{0conver3}\\
   		\theta^{L}-\mu^{L}&\rightarrow0&&\text{strongly in }L^{2}(0,T;H_{\Gamma}),\label{0conver4}\\
   		\beta(\varphi^{L})&\rightarrow\xi^{0}&&\text{weakly in }L^{2}(0,T;H),\label{0conver5}\\
   		\beta_{\Gamma}(\psi^{L})&\rightarrow\xi^{0}_{\Gamma}&&\text{weakly in }L^{2}(0,T;H_{\Gamma}),\label{0conver6}
   	\end{align}
   	as $L\rightarrow0$ in the sense of a subsequence. By the strong convergence \eqref{0conver1} and Lemma \ref{lem-subsequent}, it holds
   	\begin{align*}
   		J\ast\varphi^{L}\to J\ast\varphi^{0}&\qquad\text{strongly in }L^{2}(0,T;V),\\
   		K\circledast\psi^{L}\to K\circledast\psi^{0}&\qquad\text{strongly in }L^{2}(0,T;V_{\Gamma}).
   	\end{align*} 	
   Applying the same argument as \cite{GGG} again, we can conclude
   \begin{align*}
   |\varphi^{0}(x,t)|<1\quad\text{a.e. in } Q_{T},\quad|\psi^{0}(x,t)|<1\quad\text{a.e. on } \Sigma_{T},
   \end{align*}
   as well as
   	\begin{align*}
   	\xi^{0}= \beta(\varphi^{0}),\qquad\xi_{\Gamma}^{0}= \beta_{\Gamma}(\psi^{0}).
   	\end{align*}
   	By \eqref{weak3} and \eqref{weak4}, the following formulations
   	\begin{align}
   	&\int_{\Omega}\mu^{L}\zeta\,\mathrm{d}x =\int_{\Omega}(a_{\Omega}\varphi^{L}-J\ast\varphi^{L}+\beta(\varphi^{L})+\pi(\varphi^{L}))\zeta\,\mathrm{d}x,\label{weak3'}\\
   	&\int_{\Gamma}\theta^{L}\zeta_{\Gamma}\,\mathrm{d}S =\int_{\Gamma}(a_{\Gamma}\psi^{L}-K\circledast\psi^{L}+\beta_{\Gamma}(\psi^{L})+\pi_{\Gamma}(\psi^{L}))\zeta_{\Gamma}\,\mathrm{d}S,\label{weak4'}
   	\end{align}
   hold for almost all $t\in[0,T]$ and for all test functions $\boldsymbol{\zeta}=(\zeta,\zeta_{\Gamma})\in \mathcal{L}^2$.
   	Then, we can pass to the limit as $L\rightarrow0$ in \eqref{weak1}, \eqref{weak2}, \eqref{weak3'} and \eqref{weak4'}
   	to obtain the existence of a weak solution to problem \eqref{model1}--\eqref{psiini} on $[0,T]$ for the limiting case $L=0$. By comparison in \eqref{weak3} and \eqref{weak4}, we see that $\beta(\varphi^0)\in L^2(0,T;V)$ and $\beta_\Gamma(\psi^0)\in L^2(0,T;V_\Gamma)$.
   	The continuous dependence estimate can be established similar to the proof of Theorem \ref{contidependence} for the case $L\in(0,+\infty)$.
   	Finally, in order to establish the energy equality \eqref{energyeq},
   	we need to show the instantaneous regularity of $\partial_t\boldsymbol{\varphi}^{0}$ (cf. \eqref{Tuni3}).
   	This can be achieved as in Step 3 of the proof of Theorem \ref{exist} for the case $L\in(0,+\infty)$. Here, we should take the difference quotient of the weak formulations \eqref{0weak1}--\eqref{weak4} instead of the approximating problem. The details are omitted.
   \hfill $\square$

   \subsection{The case $L\to +\infty$}
   \begin{lemma}
   	\label{Linfuni1}
   	Suppose $L\geq L_{0}$, with $L_{0}\geq1$ being a large constant to be determined. Then, there exists a positive constant $C$, independent of $L\in[L_{0},+\infty)$, such that
   	\begin{align}
   		\Vert \beta(\varphi^{L})\Vert_{L^{2}(0,T;L^{1}(\Omega))}+\Vert \beta_{\Gamma}(\psi^{L})\Vert_{L^{2}(0,T;L^{1}(\Gamma))}\leq C.\label{infuni1}
   	\end{align}
   \end{lemma}
   \begin{proof}
    Taking $z=\mathcal{N}_{\Omega}(\varphi^{L}-\langle\varphi^{L}\rangle_{\Omega})$ in \eqref{weak1} and $z=\mathcal{N}_{\Gamma}(\psi^{L}-\langle\psi^{L}\rangle_{\Gamma})$ in \eqref{weak2}, we obtain
   	\begin{align}
   	0&=\langle\partial_{t}\varphi^{L},\mathcal{N}_{\Omega}(\varphi^{L}-\langle\varphi^{L}\rangle_{\Omega})\rangle_{V',V}+\langle\partial_{t}\psi^{L},\mathcal{N}_{\Gamma}(\psi^{L}-\langle\psi^{L}\rangle_{\Gamma})\rangle_{V_{\Gamma}',V_{\Gamma}}\notag\\
   	&\quad+\int_{\Omega}\mu^{L}(\varphi^{L}-\langle\varphi^{L}\rangle_{\Omega})\,\mathrm{d}x+\int_{\Gamma}\theta^{L}(\psi^{L}-\langle\psi^{L}\rangle_{\Gamma})\,\mathrm{d}S\notag\\
   	&\quad-\frac{1}{L}\int_{\Gamma}(\theta^{L}-\mu^{L})(\mathcal{N}_{\Omega}(\varphi^{L}-\langle\varphi^{L}\rangle_{\Omega})-\mathcal{N}_{\Gamma}(\psi^{L}-\langle\psi^{L}\rangle_{\Gamma}))\,\mathrm{d}S\notag\\
   	&=\sum_{i=1}^{5}J_{i}.\label{infuni2}
   	\end{align}
   It follows from the definition of $\mathcal{N}_{\Omega}$ and $\mathcal{N}_{\Gamma}$ that
   	\begin{align}
   	|J_{1}|+|J_{2}|&\leq\Vert\partial_{t}\varphi^{L}\Vert_{V'} \Vert\mathcal{N}_{\Omega}(\varphi^{L} -\langle\varphi^{L}\rangle_{\Omega})\Vert_{V} +\Vert\partial_{t}\psi^{L}\Vert_{V_{\Gamma}'}\Vert\mathcal{N}_{\Gamma}(\psi^{L}-\langle\psi^{L}\rangle_{\Gamma})\Vert_{V_{\Gamma}}\notag\\[1mm]
   	&\leq C\Vert\partial_{t}\varphi^{L}\Vert_{V'}\Vert\varphi^{L} -\langle\varphi^{L}\rangle_{\Omega}\Vert_{H} +C\Vert\partial_{t}\psi^{L}\Vert_{V_{\Gamma}'}\Vert\psi^{L} -\langle\psi^{L}\rangle_{\Gamma}\Vert_{H_{\Gamma}}.\label{infuni3}
   	\end{align}
   	Next, from the trace theorem, H\"older's inequality and the definition of $\mathcal{N}_{\Omega}$ and $\mathcal{N}_{\Gamma}$, we obtain
   	\begin{align}
   	|J_{5}|&\leq\frac{1}{L}\Vert\theta^{L}-\mu^{L}\Vert_{H_{\Gamma}} \Vert\mathcal{N}_{\Omega}(\varphi^{L}-\langle\varphi^{L}\rangle_{\Omega})-\mathcal{N}_{\Gamma}(\psi^{L}-\langle\psi^{L}\rangle_{\Gamma})\Vert_{H_{\Gamma}}\notag\\
   	&\leq\frac{C}{L}\Vert\theta^{L}-\mu^{L}\Vert_{H_{\Gamma}} (\Vert\mathcal{N}_{\Omega}(\varphi^{L}-\langle\varphi^{L}\rangle_{\Omega})\Vert_{V}+\Vert\mathcal{N}_{\Gamma}(\psi^{L}-\langle\psi^{L}\rangle_{\Gamma})\Vert_{V_{\Gamma}})\notag\\
   	&\leq\frac{C}{L}\Vert\theta^{L}-\mu^{L}\Vert_{H_{\Gamma}} (\Vert\varphi^{L}-\langle\varphi^{L}\rangle_{\Omega}\Vert_{H} +\Vert\psi^{L}-\langle\psi^{L}\rangle_{\Gamma}\Vert_{H_{\Gamma}}). \label{infuni4}
   	\end{align}
   	Using \eqref{weak3} and \eqref{weak4}, we rewrite $J_{3}+J_{4}$ as follows
   	\begin{align}
   	J_{3}+J_{4}&=\int_{\Omega}(a_{\Omega}\varphi^{L} -J\ast\varphi^{L})(\varphi^{L} -\langle\varphi^{L}\rangle_{\Omega})\,\mathrm{d}x +\int_{\Gamma}(a_{\Gamma}\psi^{L}-K\circledast\psi^{L}) (\psi^{L}-\langle\psi^{L}\rangle_{\Gamma})\,\mathrm{d}S\notag\\
   	&\quad+\int_{\Omega}\beta(\varphi^{L})(\varphi^{L} -\langle\varphi^{L}\rangle_{\Omega})\,\mathrm{d}x +\int_{\Gamma}\beta_{\Gamma}(\psi^{L})(\psi^{L} -\langle\psi^{L}\rangle_{\Gamma})\,\mathrm{d}S\notag\\
   	&\quad+\int_{\Omega}\pi(\varphi^{L})(\varphi^{L} -\langle\varphi^{L}\rangle_{\Omega})\,\mathrm{d}x +\int_{\Gamma}\pi_{\Gamma}(\psi^{L})(\psi^{L} -\langle\psi^{L}\rangle_{\Gamma})\,\mathrm{d}S\notag\\
   	&=\sum_{i=1}^{6}K_{i}.\label{infuni5}
   	\end{align}
   	By H\"older's inequality, Young's inequality for convolution, $\mathbf{(A1)}$ and $\mathbf{(A3)}$, we deduce that
   	\begin{align}
   	&|K_{1}|+|K_{2}|+|K_{5}|+|K_{6}|\notag\\[1mm]
   	&\quad\leq(\Vert a_{\Omega}\varphi^{L}\Vert_{H}+\Vert J\ast\varphi^{L}\Vert_{H})\Vert\varphi^{L}-\langle\varphi^{L}\rangle_{\Omega}\Vert_{H}\notag\\[1mm]
   	&\qquad+(\Vert a_{\Gamma}\psi^{L}\Vert_{H_{\Gamma}}+\Vert K\circledast\psi^{L}\Vert_{H_{\Gamma}})\Vert\psi^{L}-\langle\psi^{L}\rangle_{\Gamma}\Vert_{H_{\Gamma}}\notag\\[1mm]
   	&\qquad+\Vert \pi(\varphi^{L})\Vert_{H}\Vert\varphi^{L}-\langle\varphi^{L}\rangle_{\Omega}\Vert_{H}+\Vert \pi_{\Gamma}(\psi^{L})\Vert_{H_{\Gamma}}\Vert\psi^{L}-\langle\psi^{L}\rangle_{\Gamma}\Vert_{H_{\Gamma}}\notag\\[1mm]
   	&\quad\leq(\Vert a_{\Omega}\Vert_{L^{\infty}(\Omega)}+\Vert J\Vert_{L^{1}(\Omega)})\Vert\varphi^{L}\Vert_{H}\Vert\varphi^{L}-\langle\varphi^{L}\rangle_{\Omega}\Vert_{H}\notag\\[1mm]
   	&\qquad+(\Vert a_{\Gamma}\Vert_{L^{\infty}(\Gamma)}+\Vert K\Vert_{L^{1}(\Gamma)})\Vert\psi^{L}\Vert_{H_{\Gamma}}\Vert\psi^{L}-\langle\psi^{L}\rangle_{\Gamma}\Vert_{H_{\Gamma}}\notag\\[1mm]
   	&\qquad+\Vert \pi(\varphi^{L})\Vert_{H}\Vert\varphi^{L}-\langle\varphi^{L}\rangle_{\Omega}\Vert_{H}+\Vert \pi_{\Gamma}(\psi^{L})\Vert_{H_{\Gamma}}\Vert\psi^{L}-\langle\psi^{L}\rangle_{\Gamma}\Vert_{H_{\Gamma}}\notag\\[1mm]
   	&\quad\leq C(1+\Vert\varphi^{L}\Vert_{H}^{2}+\Vert\psi^{L}\Vert_{H_{\Gamma}}^{2}).\label{infuni6}
   	\end{align}
   	Thanks to $\mathbf{(A2)}$, for any $s_0\in(-1,1)$, there exist positive constants $\widetilde{\delta}_0$ and $\widetilde{c}_1$ such that (see \cite{MZ04,GMS09})
   	\begin{align}
   	&\widetilde{\delta}_0|\beta(s)|\leq \beta(s)(s-s_0)+\widetilde{c}_1,\quad\forall\,s\in(-1,1),\label{L1control1}\\
   	&\widetilde{\delta}_0|\beta_\Gamma(s)|\leq \beta_\Gamma(s)(s-s_0)+\widetilde{c}_1,\quad\forall\,s\in(-1,1).\label{L1control2}
   	\end{align}
   	Using \eqref{L1control1} and \eqref{L1control2}, we can control  $K_{3}+K_{4}$ as follows
   	\begin{align}
   	K_{3}+K_{4}&=\int_{\Omega}\beta(\varphi^{L})(\varphi^{L} -\langle\varphi_{0}\rangle_{\Omega})\,\mathrm{d}x +\int_{\Gamma}\beta_{\Gamma}(\psi^{L})(\psi^{L} -\langle\psi_{0}\rangle_{\Gamma})\,\mathrm{d}S\notag\\
   &\quad+\int_{\Omega}\beta(\varphi^L)(\langle\varphi_{0}\rangle_{\Omega} -\langle\varphi^{L}\rangle_{\Omega})\,\mathrm{d}x+\int_{\Gamma} \beta_{\Gamma}(\psi^L)(\langle\psi_{0}\rangle_{\Gamma} -\langle\psi^{L}\rangle_{\Gamma})\,\mathrm{d}S\notag\\
   	&\geq \widetilde{\delta}_{0}\Big(\int_{\Omega}|\beta(\varphi^{L})| \,\mathrm{d}x+\int_{\Gamma}|\beta_{\Gamma}(\psi^{L})|\,\mathrm{d}S\Big) -\widetilde{c}_{1}(|\Omega|+|\Gamma|)\notag\\
   	&\quad-|\langle\varphi_{0}\rangle_{\Omega}-\langle\varphi^{L} \rangle_{\Omega}|\int_{\Omega}|\beta(\varphi^{L})|\,\mathrm{d}x\notag\\
   	&\quad-|\langle\psi_{0}\rangle_{\Gamma} -\langle\psi^{L}\rangle_{\Gamma}|\int_{\Gamma}|\beta_{\Gamma}(\psi^{L})| \,\mathrm{d}S.\label{infuni7}
   	\end{align}
   	It follows from \eqref{weak1} and \eqref{weak2} that
   	\begin{align}
   		|\langle\varphi_{0}\rangle_{\Omega} -\langle\varphi^{L}(t)\rangle_{\Omega}|
   		&=\frac{1}{|\Omega|}\left|\int_0^t\langle \partial_t\varphi^{L}(s),1\rangle_{V',V}\,\mathrm{d}s\right|
   		\notag\\
   		&\leq \frac{1}{|\Omega|}\frac{1}{L}\int_{0}^{t}\|\mu^L(s) -\theta^L(s)\|_{L^1(\Gamma)}\,\mathrm{d}s
   		\notag\\
   		&\leq\frac{|\Gamma|^{1/2}|T|^{1/2}}{|\Omega|}\frac{1}{\sqrt{L}} \left(\frac{1}{\sqrt{L}}\Vert\mu^L-\theta^L\Vert_{L^{2}(0,T;H_{\Gamma})} \right)
   		\notag\\
   		&\leq\frac{C'|\Gamma|^{1/2}|T|^{1/2}}{|\Omega|}\frac{1}{\sqrt{L}}, \quad \forall\, t\in [0,T],
   		\label{infuni8}
   	\end{align}
   	and in a similar manner,
   	\begin{align}
   		|\langle\psi_{0}\rangle_{\Gamma}-\langle\psi^{L}(t)\rangle_{\Gamma}|
   		\leq\frac{C''|T|^{1/2}}{|\Gamma|^{1/2}}\frac{1}{\sqrt{L}},\quad \forall\, t\in [0,T],
   		\label{infuni9}
   	\end{align}
   	where the constants $C',C''>0$ in \eqref{infuni8}, \eqref{infuni9} are independent of $L$.
   	Then, there exists some $L_{0}\geq1$ sufficiently large such that
   	\begin{align}
   	\frac{C'|\Gamma|^{1/2}|T|^{1/2}}{|\Omega|}\frac{1}{\sqrt{L}} \leq\frac{\widetilde{\delta}_{0}}{2},\quad\frac{C''|T|^{1/2}}{|\Gamma|^{1/2}} \frac{1}{\sqrt{L}}\leq\frac{\widetilde{\delta}_{0}}{2},\qquad\forall\,L\geq L_{0}.\label{infuni10}
   	\end{align}
   	Combining \eqref{infuni7}--\eqref{infuni10}, we can conclude that
   	\begin{align}
   	K_{3}+K_{4}\geq\frac{\widetilde{\delta}_{0}}{2} \Big(\int_{\Omega}|\beta(\varphi^{L})|\,\mathrm{d}x +\int_{\Gamma}|\beta_{\Gamma}(\psi^{L})|\,\mathrm{d}S\Big) -\widetilde{c}_{1}(|\Omega|+|\Gamma|).\label{infuni11}
   	\end{align}
   	In summary, from the estimates \eqref{infuni2}--\eqref{infuni5}, \eqref{infuni11}
   	as well as \eqref{Luni1}--\eqref{0uni2}, we can conclude \eqref{infuni1}.
   \end{proof}

   \begin{lemma}
   	Suppose $L\in[L_{0},+\infty)$. Then there exists a positive constant $C$,
   	independent of $L\in[L_{0},+\infty)$, such that
   	\begin{align}
    &\Vert\boldsymbol{\varphi}^{L}\Vert_{L^{\infty}(0,T;\mathcal{L}^{2})} +\Vert\boldsymbol{\varphi}^{L}\Vert_{L^{2}(0,T;\mathcal{H}^{1})} +\Vert\boldsymbol{\mu}^{L}\Vert_{L^{2}(0,T;\mathcal{H}^{1})}+\Vert \beta(\varphi^{L})\Vert_{L^{2}(0,T;H)}+\Vert \beta_{\Gamma}(\psi^{L})\Vert_{L^{2}(0,T;H_{\Gamma})}\leq C.\label{infuni12}
   	\end{align}
   \end{lemma}
   \begin{proof}
   With the aid of Lemma \ref{Linfuni1}, we can conclude \eqref{infuni12} by proceeding similarly as for \eqref{approuni7}--\eqref{approuni11} and then passing to the limit as $\varepsilon\rightarrow0$.
   \end{proof}

    \textbf{Proof of Theorems \ref{exist} and \ref{contidependence} with $L=+\infty$.}
    Thanks to the uniform estimates \eqref{Luni1}, \eqref{0uni1}, \eqref{0uni2}, \eqref{0uni3} and \eqref{infuni12} with respect to $L\in[L_0,+\infty)$,
   the existence of a weak solution for the case $L=+\infty$ can be obtained by taking the limit $L\to +\infty$ (in the sense of a subsequence), similar to the proof of Theorem \ref{exist} for the case $L=0$.
  Concerning the continuous dependence estimate \eqref{0contiesti}, let $\boldsymbol{\varphi}_i$ be the weak solution to problem \eqref{model1}--\eqref{psiini}
   corresponding to initial data $\boldsymbol{\varphi}_{0,i}$ $(i\in\{1,2\})$,
   with $\langle\varphi_{0,1}\rangle_\Omega=\langle\varphi_{0,2}\rangle_\Omega =m_{\Omega,0}$
   and $\langle\psi_{0,1}\rangle_\Gamma=\langle\psi_{0,2}\rangle_\Gamma =m_{\Gamma,0}$.
   Denote $\boldsymbol{\varphi}_0=\boldsymbol{\varphi}_{0,1}- \boldsymbol{\varphi}_{0,2}$
   and $\boldsymbol{\varphi}=(\varphi,\psi)=\boldsymbol{\varphi}_1- \boldsymbol{\varphi}_2$.
   Then it holds
   \begin{align}
   	&\langle\partial_{t}\varphi,z\rangle_{V',V} =-\int_{\Omega}\nabla\mu\cdot\nabla z\,\mathrm{d}x,\quad\forall\,z\in V,\label{diffweak1}\\
   	&\langle\partial_{t}\psi,z_{\Gamma}\rangle_{V_{\Gamma}', V_{\Gamma}}=-\int_{\Gamma}\nabla_{\Gamma}\theta\cdot\nabla_{\Gamma} z_{\Gamma}\,\mathrm{d}S,\quad\forall\,z_\Gamma \in V_\Gamma,\label{diffweak2}
   \end{align}
   for almost all $t\in(0,T)$, with
   \begin{align}
   	&\mu:=\mu_{1}-\mu_{2}=a_{\Omega}\varphi-J\ast\varphi +\beta(\varphi_{1})-\beta(\varphi_{2})+\pi(\varphi_{1}) -\pi(\varphi_{2}),\quad\text{ a.e. in }  Q_T,\label{diffweak3}\\[1mm]
   	&\theta:=\theta_{1}-\theta_{2}=a_{\Gamma}\psi-K\circledast\psi +\beta_{\Gamma}(\psi_{1})-\beta_{\Gamma}(\psi_{2}) +\pi_{\Gamma}(\psi_{1})-\pi_{\Gamma}(\psi_{2}),\quad\text{ a.e. on } \Sigma_{T}.\label{diffweak4}
   \end{align}
   Since $\langle\varphi\rangle_\Omega=0$ and $\langle\psi\rangle_\Gamma=0$,
   we can take $z=\mathcal{N}_\Omega\varphi$ in \eqref{diffweak1}
   and $z_\Gamma=\mathcal{N}_\Gamma\psi$ in \eqref{diffweak2}, and then  obtain
   \begin{align*}
   0=\frac{1}{2}\frac{\mathrm{d}}{\mathrm{d}t}\|\varphi\|_{V_0^\ast}^2+ \frac{1}{2}\frac{\mathrm{d}}{\mathrm{d}t}\|\psi\|_{V_{\Gamma,0}^\ast}^2 +\int_\Omega\mu\varphi\,\mathrm{d}x+\int_\Gamma\theta\psi\,\mathrm{d}S.
   \end{align*}
   Keeping \eqref{diffweak3} and \eqref{diffweak4} in mind,
   by a similar procedure as we did for the case $L\in(0,+\infty)$,
   we can conclude continuous dependence estimate \eqref{0contiesti}.
   Finally, the energy equality can be proven in a similar way as for the case $L=0$.
   The details are omitted.
  \hfill $\square$

    \begin{remark}\rm
   	Since the subsystems for $(\varphi,\mu)$ and $(\psi,\theta)$ are completely decoupled in the case $L=+\infty$,
   	we can solve the two systems \eqref{infch1}--\eqref{infch4} and \eqref{surch1}--\eqref{surch3} separately,
   	by applying the Yosida approximation and the Faedo--Galerkin scheme in the bulk and on the boundary.
   	In particular, concerning the subsystem \eqref{infch1}--\eqref{infch4}, we refer to \cite{GGG} for more details.
   \end{remark}

   \subsection{Convergence rate of asymptotic limits}

   The last part of this section is devoted to prove Theorem \ref{weak-convergence}
   on the convergence rate of weak solutions $(\boldsymbol{\varphi}^{L},\boldsymbol{\mu}^{L})$.
   \smallskip

    \textbf{Proof of Theorem \ref{weak-convergence}.}
   We first deal with the case as $L\to0$. Using the weak formulations \eqref{weak1}, \eqref{weak2}, \eqref{0weak1}, we deduce that
   \begin{align}
   \langle\partial_{t}(\boldsymbol{\varphi}^{L}-\boldsymbol{\varphi}^{0}),\boldsymbol{z}\rangle_{(\mathcal{V}^{1})',\mathcal{V}^{1}}=-\int_{\Omega}\nabla(\mu^{L}-\mu^{0})\cdot\nabla z\,\mathrm{d}x-\int_{\Gamma}\nabla_{\Gamma}(\theta^{L}-\theta^{0})\cdot\nabla_{\Gamma}z_{\Gamma}\,\mathrm{d}S,\label{0rate1}
   \end{align}
   for almost all $t\in[0,T]$ and for all $\boldsymbol{z}=(z,z_\Gamma)\in \mathcal{V}^{1}$.
   Thanks to \eqref{weak3}, \eqref{weak4}, it holds
   \begin{align}
    &\mu^{L}-\mu^{0}=a_{\Omega}(\varphi^{L}-\varphi^{0})-J\ast(\varphi^{L}-\varphi^{0})+\beta(\varphi^{L})-\beta(\varphi^{0})+\pi(\varphi^{L})-\pi(\varphi^{0}),\quad\text{a.e. in }Q_{T},\label{0rate2}\\[1mm]
    &\theta^{L}-\theta^{0}=a_{\Gamma}(\psi^{L}-\psi^{0})-K\circledast(\psi^{L}-\psi^{0})+\beta_{\Gamma}(\psi^{L})-\beta_{\Gamma}(\psi^{0})+\pi_{\Gamma}(\psi^{L})-\pi_{\Gamma}(\psi^{0}),\quad\text{a.e. on }\Sigma_{T}.\label{0rate3}
   \end{align}
   Since $\overline{m}(\boldsymbol{\varphi}^{L}-\boldsymbol{\varphi}^{0})=0$,
   we can take $\boldsymbol{z}=\mathfrak{S}^{0}(\boldsymbol{\varphi}^{L} -\boldsymbol{\varphi}^{0})$
   in \eqref{0rate1} and then infer from \eqref{0rate2}, \eqref{0rate3} that
   \begin{align}
   	0&=\frac{1}{2}\frac{\mathrm{d}}{\mathrm{d}t}\|\boldsymbol{\varphi}^{L} -\boldsymbol{\varphi}^{0}\|_{0,0,\ast}^{2}+\int_{\Omega}(\mu^{L}-\mu^{0}) (\varphi^{L}-\varphi^{0}) \,\mathrm{d}x+\int_{\Gamma}(\theta^{L}-\theta^{0})(\psi^{L}-\psi^{0}) \,\mathrm{d}S\notag\\
   	&\quad-\int_{\Gamma}(\theta^{L}-\mu^{L})\partial_{\mathbf{n}} \mathfrak{S}^{0}_{\Omega}(\boldsymbol{\varphi}^{L}-\boldsymbol{\varphi}^{0}) \,\mathrm{d}S\notag\\
   	&=\frac{1}{2}\frac{\mathrm{d}}{\mathrm{d}t}\|\boldsymbol{\varphi}^{L} -\boldsymbol{\varphi}^{0}\|_{0,0,\ast}^{2}+\int_{\Omega}a_{\Omega}(\varphi^{L} -\varphi^{0})^{2}\,\mathrm{d}x+\int_{\Omega}(\beta(\varphi^{L}) -\beta(\varphi^{0}))(\varphi^{L}-\varphi^{0})\,\mathrm{d}x\notag\\
   	&\quad+\int_{\Gamma}a_{\Gamma}(\psi^{L}-\psi^{0})^{2}\,\mathrm{d}S +\int_{\Gamma}(\beta_{\Gamma}(\psi^{L})-\beta_{\Gamma}(\psi^{0})) (\psi^{L}-\psi^{0})\,\mathrm{d}S\notag\\
   	&\quad+\int_{\Omega}(\pi(\varphi^{L})-\pi(\varphi^{0})) (\varphi^{L}-\varphi^{0})\,\mathrm{d}x+\int_{\Gamma} (\pi_{\Gamma}(\psi^{L})-\pi_{\Gamma}(\psi^{0}))(\psi^{L}-\psi^{0} )\,\mathrm{d}S\notag\\
   	&\quad-\int_{\Omega}(J\ast(\varphi^{L}-\varphi^{0})) (\varphi^{L}-\varphi^{0})\,\mathrm{d}x-\int_{\Gamma} (K\circledast(\psi^{L}-\psi^{0}))(\psi^{L}-\psi^{0}) \,\mathrm{d}S\notag\\
   	&\quad-\int_{\Gamma}(\theta^{L}-\mu^{L})\partial_{\mathbf{n}} \mathfrak{S}^{0}_{\Omega}(\boldsymbol{\varphi}^{L}-\boldsymbol{\varphi}^{0}) \,\mathrm{d}S\notag\\
   	&\geq\frac{1}{2}\frac{\mathrm{d}}{\mathrm{d}t}\|\boldsymbol{\varphi}^{L} -\boldsymbol{\varphi}^{0}\|_{0,0,\ast}^{2}+(a_{\ast} +\alpha-\gamma_{1})\|\varphi^{L}-\varphi^{0}\|_{H}^{2} +(a_{\circledast}+\alpha-\gamma_{2})\|\psi^{L}-\psi^{0}\|_{H_{\Gamma}}^{2} \notag\\
   &\quad-\int_{\Omega}(J\ast(\varphi^{L}-\varphi^{0}))(\varphi^{L} -\varphi^{0})\,\mathrm{d}x-\int_{\Gamma}(K\circledast(\psi^{L}-\psi^{0}))(\psi^{L}-\psi^{0})\,\mathrm{d}S\notag\\
   	&\quad-\int_{\Gamma}(\theta^{L}-\mu^{L})\partial_{\mathbf{n}} \mathfrak{S}^{0}_{\Omega}(\boldsymbol{\varphi}^{L}-\boldsymbol{\varphi}^{0}) \,\mathrm{d}S.\label{0rate4}
   \end{align}
   By similar calculations as we did in \eqref{error10}--\eqref{error11}, we find
   \begin{align}
   	&\int_{\Omega}(J\ast(\varphi^{L}-\varphi^{0})) (\varphi^{L}-\varphi^{0})\,\mathrm{d}x +\int_{\Gamma}(K\circledast(\psi^{L}-\psi^{0})) (\psi^{L}-\psi^{0})\,\mathrm{d}S\notag\\
   	&\quad\leq(\epsilon\|\mathbb{J}\|_{\mathcal{B}(\mathcal{L}^{2})}^{2} +\|\mathbb{J}-\mathbb{J}_{N}\|_{\mathcal{B}(\mathcal{L}^{2})}) \|\boldsymbol{\varphi}^{L}-\boldsymbol{\varphi}^{0}\|_{\mathcal{L}^{2}}^{2} +C(\epsilon)C(N)\|\boldsymbol{\varphi}^{L}-\boldsymbol{\varphi}^{0}\|_{0,0,\ast}^{2}.\label{0rate5}
   \end{align}
   Using H\"older's inequality and the trace theorem, we infer from \eqref{Luni1} that
   \begin{align}
    &\Big|\int_{\Gamma}(\theta^{L}-\mu^{L})\partial_{\mathbf{n}}\mathfrak{S}^{0}_{\Omega}(\boldsymbol{\varphi}^{L}-\boldsymbol{\varphi}^{0})\,\mathrm{d}S\Big|\notag\\
    &\quad\leq\|\theta^{L}-\mu^{L}\|_{H_{\Gamma}}\|\partial_{\mathbf{n}}\mathfrak{S}^{0}_{\Omega}(\boldsymbol{\varphi}^{L}-\boldsymbol{\varphi}^{0})\|_{H_{\Gamma}}\notag\\[1mm]
    &\quad\leq C\|\theta^{L}-\mu^{L}\|_{H_{\Gamma}} \|\mathfrak{S}^{0}_{\Omega}(\boldsymbol{\varphi}^{L}-\boldsymbol{\varphi}^{0})\|_{H^{2}(\Omega)}\notag\\[1mm]
    &\quad\leq C\|\theta^{L}-\mu^{L}\|_{H_{\Gamma}} \|\boldsymbol{\varphi}^{L}-\boldsymbol{\varphi}^{0}\|_{\mathcal{L}^{2}}\notag\\[1mm]
    &\quad\leq C(\widetilde{\epsilon})\|\theta^{L}-\mu^{L}\|_{H_{\Gamma}}^{2} +\widetilde{\epsilon}\|\boldsymbol{\varphi}^{L} -\boldsymbol{\varphi}^{0}\|_{\mathcal{L}^{2}}^{2}.\label{0rate6}
   \end{align}
   Taking $N\in\mathbb{Z}^{+}$ sufficiently large,
   then $\epsilon>0$ sufficiently small in \eqref{0rate5} and $\widetilde{\epsilon}>0$ sufficiently small in \eqref{0rate6}, we get
   \begin{align*}
   &a_{\ast}+\alpha-\gamma_{1}-\epsilon\|\mathbb{J}\|_{\mathcal{B}(\mathcal{L}^{2})}^{2}-\|\mathbb{J}-\mathbb{J}_{N}\|_{\mathcal{B}(\mathcal{L}^{2})}-\widetilde{\epsilon}>0,\\
   &a_{\circledast}+\alpha-\gamma_{2}-\epsilon\|\mathbb{J}\|_{\mathcal{B} (\mathcal{L}^{2})}^{2}-\|\mathbb{J}-\mathbb{J}_{N}\|_{\mathcal{B} (\mathcal{L}^{2})}-\widetilde{\epsilon}>0.
   \end{align*}
   Then it follows that
   \begin{align*}
   \frac{\mathrm{d}}{\mathrm{d}t}\|\boldsymbol{\varphi}^{L} -\boldsymbol{\varphi}^{0}\|_{0,0,\ast}^{2}+C_{5}\|\boldsymbol{\varphi}^{L} -\boldsymbol{\varphi}^{0}\|_{\mathcal{L}^{2}}^{2}\leq C_{6}\|\boldsymbol{\varphi}^{L}-\boldsymbol{\varphi}^{0}\|_{0,0,\ast}^{2} +C_{7}\|\theta^{L}-\mu^{L}\|_{H_{\Gamma}}^{2}.
   \end{align*}
   Applying Gronwall's inequality and using \eqref{Luni1}, we can conclude \eqref{to0rate}.

   Let us now turn to the case as $L\to+\infty$. From \eqref{weak1} and \eqref{weak3}, we obtain
   \begin{align}
   	&\langle\partial_{t}(\varphi^{L}-\varphi^{\infty}), z\rangle_{V',V}=-\int_{\Omega}\nabla(\mu^{L}-\mu^{\infty})\cdot\nabla z\,\mathrm{d}x+\frac{1}{L}\int_{\Gamma}(\theta^{L}-\mu^{L})z \,\mathrm{d}S\label{infrate1}
   	\end{align}
   for almost all $t\in[0,T]$ and for all $z\in V$, and
   	\begin{align}
   	&\mu^{L}-\mu^{\infty}=a_{\Omega}(\varphi^{L}-\varphi^{\infty}) -J\ast(\varphi^{L}-\varphi^{\infty})+\beta(\varphi^{L}) -\beta(\varphi^{\infty})+\pi(\varphi^{L})-\pi(\varphi^{\infty}), \quad\text{a.e. in }Q_{T}.\label{infrate2}
   \end{align}
   Taking $z=\mathcal{N}_{\Omega}(\varphi^{L}-\varphi^{\infty} -\langle\varphi^{L}-\varphi^{\infty}\rangle_{\Omega})$ in \eqref{infrate1} yields
   \begin{align}
   0&=\frac{1}{2}\frac{\mathrm{d}}{\mathrm{d}t}\|\varphi^{L} -\varphi^{\infty}-\langle\varphi^{L}-\varphi^{\infty}\rangle_{\Omega}\|_{V_{0}^{\ast}}^{2}\notag\\
   &\quad+\int_{\Omega}(\mu^{L}-\mu^{\infty})(\varphi^{L} -\varphi^{\infty}-\langle\varphi^{L}-\varphi^{\infty}\rangle_{\Omega})\,\mathrm{d}x\notag\\
   &\quad+\langle\partial_{t}\langle\varphi^{L} -\varphi^{\infty}\rangle_{\Omega},\mathcal{N}_{\Omega}(\varphi^{L} -\varphi^{\infty}-\langle\varphi^{L}-\varphi^{\infty}\rangle_{\Omega}) \rangle_{V',V}\notag\\
   &\quad-\frac{1}{L}\int_{\Gamma}(\theta^{L}-\mu^{L})\mathcal{N}_{\Omega} (\varphi^{L}-\varphi^{\infty}-\langle\varphi^{L} -\varphi^{\infty}\rangle_{\Omega})\,\mathrm{d}S\notag\\
   &=\frac{1}{2}\frac{\mathrm{d}}{\mathrm{d}t}\|\varphi^{L} -\varphi^{\infty}-\langle\varphi^{L}-\varphi^{\infty}\rangle_{\Omega}\|_{V_{0}^{\ast}}^{2} +\mathcal{I}_{1}+\mathcal{I}_{2}+\mathcal{I}_{3}.\label{infrate3}
   \end{align}
   For the term $\mathcal{I}_{1}$, we infer from \eqref{infrate1} that
   \begin{align}
   \mathcal{I}_{1}&=\int_{\Omega}a_{\Omega}(\varphi^{L}-\varphi^{\infty} -\langle\varphi^{L}-\varphi^{\infty}\rangle_{\Omega})^{2}\,\mathrm{d}x\notag\\
   &\quad+\int_{\Omega}(\beta(\varphi^{L})-\beta(\varphi^{\infty})) (\varphi^{L}-\varphi^{\infty})\,\mathrm{d}x+\int_{\Omega}(\pi(\varphi^{L}) -\pi(\varphi^{\infty}))(\varphi^{L}-\varphi^{\infty})\,\mathrm{d}x\notag\\
   &\quad-\int_{\Omega}(J\ast(\varphi^{L}-\varphi^{\infty}-\langle\varphi^{L} -\varphi^{\infty}\rangle_{\Omega}))(\varphi^{L}-\varphi^{\infty} -\langle\varphi^{L}-\varphi^{\infty}\rangle_{\Omega})\,\mathrm{d}x\notag\\
   &\quad+\langle\varphi^{L}-\varphi^{\infty}\rangle_{\Omega} \int_{\Omega}a_{\Omega}(\varphi^{L}-\varphi^{\infty}-\langle\varphi^{L}-\varphi^{\infty}\rangle_{\Omega})\,\mathrm{d}x\notag\\
   &\quad-\langle\varphi^{L}-\varphi^{\infty}\rangle_{\Omega} \int_{\Omega}(J\ast1)(\varphi^{L}-\varphi^{\infty} -\langle\varphi^{L}-\varphi^{\infty}\rangle_{\Omega})\,\mathrm{d}x\notag\\
   &\quad-\langle\varphi^{L}-\varphi^{\infty}\rangle_{\Omega} \int_{\Omega}(\beta(\varphi^{L})-\beta(\varphi^{\infty}) +\pi(\varphi^{L})-\pi(\varphi^{\infty}))\,\mathrm{d}x.\label{infrate3-1}
   \end{align}
   For the first two lines on the right-hand side of \eqref{infrate3-1},
   by $\mathbf{(A1)}$, $\mathbf{(A2)}$, $\mathbf{(A3)}$ and \eqref{infuni8}, we find
   \begin{align}
   &\int_{\Omega}a_{\Omega}(\varphi^{L}-\varphi^{\infty} -\langle\varphi^{L}-\varphi^{\infty}\rangle_{\Omega})^{2}\,\mathrm{d}x\notag\\
   	&\qquad+\int_{\Omega}(\beta(\varphi^{L})-\beta(\varphi^{\infty})) (\varphi^{L}-\varphi^{\infty})\,\mathrm{d}x+\int_{\Omega}(\pi(\varphi^{L}) -\pi(\varphi^{\infty}))(\varphi^{L}-\varphi^{\infty})\,\mathrm{d}x\notag\\
   	&\quad\geq a_{\ast}\|\varphi^{L}-\varphi^{\infty}-\langle\varphi^{L}-\varphi^{\infty} \rangle_{\Omega}\|_{H}^{2}+(\alpha-\gamma_{1})\|\varphi^{L} -\varphi^{\infty}\|_{H}^{2}\notag\\[1mm]
   	&\quad=(a_{\ast}+\alpha-\gamma_{1})\|\varphi^{L}-\varphi^{\infty} -\langle\varphi^{L}-\varphi^{\infty}\rangle_{\Omega}\|_{H}^{2} +(\alpha-\gamma_{1})|\Omega||\langle\varphi^{L}-\varphi^{\infty}\rangle_{\Omega}|^{2}\notag\\
   	&\quad\geq(a_{\ast}+\alpha-\gamma_{1})\|\varphi^{L}-\varphi^{\infty} -\langle\varphi^{L}-\varphi^{\infty}\rangle_{\Omega}\|_{H}^{2}-\frac{C}{L}. \label{infrate3-2}
   \end{align}
   Since $a_{\Omega}=J\ast1$, the sum of the forth and fifth lines on the right-hand side of \eqref{infrate3-1} simply equal to zero.
   For the third line on the right-hand side of \eqref{infrate3-1}, we define a linear operator $\widetilde{\mathbb{J}}:H\to H$ by
   $$\widetilde{\mathbb{J}}(z)[x]:=\int_{\Omega}J(x-y)z(y)\,\mathrm{d}y,\quad\forall\,x\in\Omega,\quad\forall\,z\in H.$$
   As in \cite{GST}, the operator $\widetilde{\mathbb{J}}$ is a Hilbert--Schmidt operator,
   and then, by a similar calculation as we did in \eqref{error10}--\eqref{error12}, we get
   \begin{align}
   &\Big|\int_{\Omega}(J\ast(\varphi^{L}-\varphi^{\infty}-\langle\varphi^{L} -\varphi^{\infty}\rangle_{\Omega}))(\varphi^{L}-\varphi^{\infty} -\langle\varphi^{L}-\varphi^{\infty}\rangle_{\Omega})\,\mathrm{d}x\Big|\notag\\
   &\quad\leq\Big( \epsilon\|\widetilde{\mathbb{J}}\|_{\mathcal{B}(H)}^{2} +\|\widetilde{\mathbb{J}}-\widetilde{\mathbb{J}}_{N}\|_{\mathcal{B}(H)}\Big) \|\varphi^{L}-\varphi^{\infty}-\langle\varphi^{L}-\varphi^{\infty}\rangle_{\Omega}\|_{H}^{2}\notag\\
   &\qquad+C(\epsilon)C(N)\|\varphi^{L}-\varphi^{\infty}-\langle\varphi^{L} -\varphi^{\infty}\rangle_{\Omega}\|_{V_{0}^{\ast}}^{2}.\label{infrate3-4}
   \end{align}
   By \eqref{infuni1} and \eqref{infuni8}, the last term in \eqref{infrate4} can be estimated as
   \begin{align}
   &\Big|\langle\varphi^{L}-\varphi^{\infty}\rangle_{\Omega}\int_{\Omega}(\beta(\varphi^{L})-\beta(\varphi^{\infty})+\pi(\varphi^{L})-\pi(\varphi^{\infty}))\,\mathrm{d}x\Big|\notag\\
   &\quad\leq \frac{C}{\sqrt{L}}(\|\beta(\varphi^{L})\|_{L^{1}(\Omega)}+\|\beta(\varphi^{\infty})\|_{L^{1}(\Omega)}+1).\label{infrate4-1}
   \end{align}
   In summary, we obtain
   \begin{align}
   	\mathcal{I}_{1}&\geq(a_{\ast}+\alpha-\gamma_{1} -\epsilon\|\widetilde{\mathbb{J}}\|_{\mathcal{B}(H)}^{2} -\|\widetilde{\mathbb{J}}-\widetilde{\mathbb{J}}_{N}\|_{\mathcal{B}(H)}) \|\varphi^{L}-\varphi^{\infty}-\langle\varphi^{L} -\varphi^{\infty}\rangle_{\Omega}\|_{H}^{2}\notag\\
   	&\quad-C(\epsilon)C(N)\|\varphi^{L}-\varphi^{\infty} -\langle\varphi^{L}-\varphi^{\infty}\rangle_{\Omega}\|_{V_{0}^{\ast}}^{2} -\frac{C}{L}\notag\\
   	&\quad-\frac{C}{\sqrt{L}}(\|\beta(\varphi^{L})\|_{L^{1}(\Omega)} +\|\beta(\varphi^{\infty})\|_{L^{1}(\Omega)}+1).\label{infrate4}
   \end{align}
   For the terms $\mathcal{I}_{2}$ and $\mathcal{I}_{3}$, using H\"older's inequality, the trace theorem, \eqref{weak1} and \eqref{Luni1}, we can conclude that
   \begin{align}
   |\mathcal{I}_{2}|&\leq \|\langle\partial_{t}\varphi^{L}\rangle_{\Omega}\|_{V'}\|\mathcal{N}_{\Omega}(\varphi^{L}-\varphi^{\infty}-\langle\varphi^{L}-\varphi^{\infty}\rangle_{\Omega})\|_{V}\notag\\[1mm]
   &\leq\|\langle\partial_{t}\varphi^{L}\rangle_{\Omega}\|_{V'}\|\varphi^{L}-\varphi^{\infty}-\langle\varphi^{L}-\varphi^{\infty}\rangle_{\Omega}\|_{H}\notag\\
   &\leq\widetilde{\epsilon}\|\varphi^{L}-\varphi^{\infty}-\langle\varphi^{L}-\varphi^{\infty}\rangle_{\Omega}\|_{H}^{2}+\frac{C(\widetilde{\epsilon})}{L^{2}}\|\theta^{L}-\mu^{L}\|_{H_{\Gamma}}^{2},\label{infrate5}\\
   |\mathcal{I}_{3}|&\leq\frac{1}{L}\|\theta^{L}-\mu^{L}\|_{H_{\Gamma}}\|\mathcal{N}_{\Omega}(\varphi^{L}-\varphi^{\infty}-\langle\varphi^{L}-\varphi^{\infty}\rangle_{\Omega})\|_{H_{\Gamma}}\notag\\
   &\leq\widetilde{\epsilon}\|\varphi^{L}-\varphi^{\infty}-\langle\varphi^{L}-\varphi^{\infty}\rangle_{\Omega}\|_{H}^{2}+\frac{C(\widetilde{\epsilon})}{L^2}\|\theta^{L}-\mu^{L}\|_{H_{\Gamma}}^{2}.\label{infrate6}
   \end{align}
   Let us first take $N\in\mathbb{Z}^+$ sufficiently large, then $\epsilon>0$ in \eqref{infrate4} and $\widetilde{\epsilon}>0$ in \eqref{infrate5}, \eqref{infrate6} sufficiently small. Then by \eqref{infrate3}--\eqref{infrate6}, we see that
   \begin{align}
   	&\frac{\mathrm{d}}{\mathrm{d}t}\|\varphi^{L}-\varphi^{\infty} -\langle\varphi^{L}-\varphi^{\infty}\rangle_{\Omega}\|_{V_{0}^{\ast}}^{2} +\widetilde{C}_{3}\|\varphi^{L}-\varphi^{\infty}-\langle\varphi^{L} -\varphi^{\infty}\rangle_{\Omega}\|_{H}^{2}\notag\\
   	&\quad\leq \widetilde{C}_{4}\|\varphi^{L}-\varphi^{\infty}-\langle\varphi^{L} -\varphi^{\infty}\rangle_{\Omega}\|_{V_{0}^{\ast}}^{2} +\frac{\mathcal{F}(t)}{\sqrt{L}},\label{infrate7}
   \end{align}
   with
   $$\mathcal{F}(t):=C\left(\|\beta(\varphi^{L})\|^{2}_{L^{1}(\Omega)} +\|\beta(\varphi^{\infty})\|_{L^{1}(\Omega)}^{2} +\frac{\|\theta^{L}-\mu^{L}\|_{H_{\Gamma}}^{2}}{L}+1\right)\in L^{1}(0,T),$$
   and $\|\mathcal{F}\|_{L^{1}(0,T)}\leq C$, where the constant $C$ is independent of $L\geq L_{0}$. Then, by \eqref{infrate7} and Gronwall's inequality, we can conclude
   \begin{align}
   \|\varphi^{L}-\varphi^{\infty}-\langle\varphi^{L}-\varphi^{\infty}\rangle_{\Omega}\|_{L^{\infty}(0,T;V_{0}^{\ast})}+\|\varphi^{L}-\varphi^{\infty}-\langle\varphi^{L}-\varphi^{\infty}\rangle_{\Omega}\|_{L^{2}(0,T;H)}\leq\frac{C}{L^{1/4}}.\label{infrate8}
   \end{align}
   Since $|\langle\varphi^{L}-\varphi^{\infty}\rangle_{\Omega}|\leq C/\sqrt{L}$,
   which, together with \eqref{infrate8}, implies the convergence rate \eqref{toinftyrate} for the bulk phase variable.
   The estimate for the boundary phase variable can be carried out in a similar way. We omit the details here.
    \hfill $\square$

   \section{Instantaneous Regularization and Strict Separation Property}
   \setcounter{equation}{0}

   In this section, we first establish the instantaneous regularizing property of weak solutions to problem \eqref{model1}--\eqref{psiini} with $L\in[0,+\infty]$.
   Using these global regularity results, we can show the instantaneous strict separation property by applying some known results in the literature based on a suitable De Giorgi's iteration scheme (see, e.g., \cite{GP,Gior,Po}).

   In what follows, we use the symbol $(\boldsymbol{\varphi}^{L}_\varepsilon,\boldsymbol{\mu}^{L}_\varepsilon)$
   to denote the unique global weak solution to the approximating problem \eqref{appro1}--\eqref{psiappro}
   corresponding to $(\varepsilon,L)\in(0,\varepsilon^\ast)\times(0,+\infty)$ obtained in Proposition \ref{approexist}
   and the symbol $(\boldsymbol{\varphi}^{L},\boldsymbol{\mu}^{L})$ to denote the unique global weak solution
   to problem \eqref{model1}--\eqref{psiini} corresponding to $L\in[0,+\infty]$ obtained in Theorem \ref{exist}.
   The strategy to prove Theorem \ref{regularize} can be divided into three steps:
   $(1)$ We first establish the instantaneous regularization for the case $L\in(0,+\infty)$
   since we need to take advantage of the approximating problem \eqref{appro1}--\eqref{psiappro}.
   $(2)$ We derive uniform estimates with respect to $L\in(0,1)$,
   and then pass to the limit as $L\to0$ to recover the instantaneous regularization for the case $L=0$.
   $(3)$ For the case $L=+\infty$, since the two subsystems are completely decoupled, we can verify the regularity property of the bulk and boundary variables separately by suitable approximations (see \cite{GGG} for the subsystem in the bulk).

   \begin{lemma}
   	\label{Tlemma1}
   	Suppose $L\in(0,+\infty)$. Then there exists a constant $C>0$ such that
   	\begin{align}
   \Vert\boldsymbol{\varphi}^{L}\Vert_{L^{\infty}(0,+\infty;\mathcal{L}^{2})} +\Vert\nabla\mu^{L}\Vert_{L^{2}(0,+\infty;H)}+\Vert\nabla_{\Gamma}\theta^{L} \Vert_{L^{2}(0,+\infty;H_{\Gamma})}+\frac{1}{\sqrt{L}}\Vert\theta^{L} -\mu^{L}\Vert_{L^{2}(0,+\infty;H_{\Gamma})}\leq C.\label{Tuni1}
   	\end{align}
   	Furthermore, it holds
   	\begin{align}
   	\Vert\partial_{t}\boldsymbol{\varphi}^{L}\Vert_{L^{2}(0,+\infty; \mathcal{H}_{L,0}^{-1})}\leq C.\label{Tuni2}
   	\end{align}
   \end{lemma}
   \begin{proof}
   	The estimate \eqref{Tuni1} follows from the energy equality \eqref{energyeq}, while \eqref{Tuni2}
   	can be deduced from \eqref{weak1}, \eqref{weak2}, \eqref{Tuni1} and the definition of $\mathfrak{S}^{L}$.
   \end{proof}

   \begin{lemma}
   	Suppose $L\in(0,+\infty)$. Then for any given $\tau>0$, there exists a constant $C>0$, independent of $\tau$, such that
   	\begin{align}
   	\Vert\nabla\mu^{L}(t)\Vert_{H}^{2}+\Vert\nabla_{\Gamma}\theta^{L}(t) \Vert_{H_{\Gamma}}^{2}+\frac{1}{L}\Vert\theta^{L}(t)-\mu^{L}(t)\Vert_{H_{\Gamma}}^{2}\leq\frac{C}{\tau},\quad\text{for a.a. }t\geq\tau.\label{Tuni4}
   	\end{align}
   \end{lemma}
   \begin{proof}
   The estimate \eqref{Tuni4} follows from \eqref{weak1}, \eqref{weak2}, \eqref{Tuni3} and the definition of $\mathfrak{S}^{L}$.
   \end{proof}

   \begin{lemma}
   	Suppose $L\in(0,+\infty)$. Then for any given $\tau>0$, there exists a constant $C>0$, independent of $\tau$, such that
   	\begin{align}
   	\Vert\boldsymbol{\varphi}^{L}(t)\Vert_{\mathcal{H}^{1}} \leq\frac{C}{\sqrt{\tau}},\quad\text{for a.a. }t\geq\tau.\label{Tuni11}
   	\end{align}
   \end{lemma}
   \begin{proof}
   	We start with the approximating problem \eqref{appro1}--\eqref{psiappro}.
   	Taking gradient in both sides of \eqref{appropointmu} and testing the resultant by $\nabla\varphi^{L}_{\varepsilon}$, we obtain
   	\begin{align}
   &\int_{\Omega}a_{\Omega}|\nabla\varphi^{L}_{\varepsilon}|^{2}\,\mathrm{d}x +\int_{\Omega}\beta_{\varepsilon}'(\varphi^{L}_{\varepsilon}) |\nabla\varphi^{L}_{\varepsilon}|^{2}\,\mathrm{d}x\notag\\
   	&\quad=\int_{\Omega}\nabla\mu^{L}_{\varepsilon} \cdot\nabla\varphi^{L}_{\varepsilon}\,\mathrm{d}x-\int_{\Omega}(\nabla a_{\Omega}\cdot\nabla\varphi^{L}_{\varepsilon})\varphi^{L}_{\varepsilon} \,\mathrm{d}x+\int_{\Omega}(\nabla J\ast\varphi^{L}_{\varepsilon})\cdot\nabla\varphi^{L}_{\varepsilon} \,\mathrm{d}x\notag\\
   	&\qquad-\int_{\Omega}\pi'(\varphi^{L}_{\varepsilon}) |\nabla\varphi^{L}_{\varepsilon}|^{2}\,\mathrm{d}x\notag\\
   	&\quad\leq\Vert\nabla\mu^{L}_{\varepsilon}\Vert_{H} \Vert\nabla\varphi^{L}_{\varepsilon}\Vert_{H}+\Vert\nabla a_{\Omega}\Vert_{L^{\infty}(\Omega)}\Vert\nabla\varphi^{L}_{\varepsilon} \Vert_{H}\Vert\varphi_{\varepsilon}^L\Vert_{H}\notag\\[1mm]
   	&\qquad+\Vert\nabla J\Vert_{L^{1}(\Omega)}\Vert\nabla\varphi^{L}_{\varepsilon}\Vert_{H} \Vert\varphi^{L}_{\varepsilon}\Vert_{H}+\gamma_{1}\Vert\nabla \varphi^{L}_{\varepsilon}\Vert_{H}^{2}\notag\\[1mm]
   	&\quad\leq (\epsilon+\gamma_{1})\Vert\nabla\varphi^{L}_{\varepsilon} \Vert_{H}^{2}+C(\epsilon)(\Vert\varphi^{L}_{\varepsilon}\Vert_{H}^{2} +\Vert\nabla\mu^{L}_{\varepsilon}\Vert_{H}^{2}).\label{Tuni12}
   	\end{align}
   We infer from \eqref{approweak1}, \eqref{approweak2} and \eqref{quotient12} that
   	\begin{align}
   	&\Vert\nabla\mu^{L}_{\varepsilon}\Vert_{H}^{2} +\Vert\nabla_{\Gamma}\theta^{L}_{\varepsilon}\Vert_{H_{\Gamma}}^{2} +\frac{1}{L}\Vert\theta^{L}_{\varepsilon}-\mu^{L}_{\varepsilon}\Vert_{H_{\Gamma}}^{2}\notag\\
   	&\quad\leq \frac{e^{C\tau}}{\tau}(E(\boldsymbol{\varphi}_{0}) +(\widetilde{C}+C_{1}+C_{2})(|\Omega|+|\Gamma|)),
   \quad\text{for a.a. }t\geq\tau.\label{quotient14}
   	\end{align}
   	Then, by $\mathbf{(A1)}$--$\mathbf{(A3)}$, \eqref{Tuni12} and \eqref{quotient14}, we obtain
   	\begin{align*}
   		\Vert\nabla\varphi^{L}_{\varepsilon}(t)\Vert_{H} \leq\frac{C}{\sqrt{\tau}},\quad\text{for a.a. }t\geq\tau.
   	\end{align*}
   	Similarly, taking surface gradient in both sides of \eqref{appropointtheta}
   	and testing the resultant by $\nabla_{\Gamma}\psi^{L}_{\varepsilon}$, we can deduce that
   	\begin{align*}
   		\Vert\nabla_{\Gamma}\psi^{L}_{\varepsilon}(t)\Vert_{H_\Gamma} \leq\frac{C}{\sqrt{\tau}},\quad\text{for a.a. }t\geq\tau.
   	\end{align*}
   As a consequence, it holds
   	\begin{align*}
   	\Vert\boldsymbol{\varphi}^{L}_{\varepsilon}(t)\Vert_{\mathcal{H}^{1}} \leq\frac{C}{\sqrt{\tau}},\quad\text{for a.a. }t\geq\tau,
   	\end{align*}
   	which yields \eqref{Tuni11} by passing to the limit as $\varepsilon\rightarrow0$.
   \end{proof}
   \begin{lemma}
   	Suppose $L\in(0,+\infty)$. Then for any given $\tau>0$, there exists a constant $C>0$, independent of $\tau$, such that
   	\begin{align}
   	&\Vert\boldsymbol{\mu}^{L}(t)\Vert_{\mathcal{H}^{1}}^{2} +\int_{t}^{t+1}\Vert\boldsymbol{\mu}^{L}(s)\Vert_{\mathcal{H}^{2}}^{2}\,\mathrm{d}s\leq\frac{C}{\tau},\quad\text{for a.a. }t\geq\tau,\label{Tuni13}\\
   	&\Vert \beta(\varphi^{L}(t))\Vert_{V}+\Vert \beta_{\Gamma}(\psi^{L}(t))\Vert_{V_{\Gamma}}\leq\frac{C}{\sqrt{\tau}}, \quad\text{for a.a. }t\geq\tau.\label{Tuni14}
   	\end{align}
   \end{lemma}
   \begin{proof}
   Testing \eqref{weak3} and \eqref{weak4} by $\zeta=\varphi^{L}-\overline{m}_{0}$ and $\zeta_{\Gamma}=\psi^{L}-\overline{m}_{0}$, respectively, we deduce from \eqref{L1control1} and \eqref{L1control2} that
   \begin{align}
   &\widetilde{\delta}_{0}\Big(\int_{\Omega}|\beta(\varphi^{L})|\,\mathrm{d}x +\int_{\Gamma}|\beta_{\Gamma}(\psi^{L})|\,\mathrm{d}S\Big)\notag\\
   	&\quad\leq \widetilde{c}_{1}(|\Omega|+|\Gamma|)+\int_{\Omega}\mu^{L} (\varphi^{L}-\overline{m}_{0})\,\mathrm{d}x+\int_{\Gamma}\theta^{L}(\psi^{L} -\overline{m}_{0})\,\mathrm{d}S\notag\\
   	&\qquad+\int_{\Omega}(J\ast\varphi^{L})(\varphi^{L}-\overline{m}_{0}) \,\mathrm{d}x+\int_{\Gamma}(K\circledast\psi^{L})(\psi^{L}-\overline{m}_{0}) \,\mathrm{d}S\notag\\
   	&\qquad-\int_{\Omega}a_{\Omega}\varphi^{L}(\varphi^{L}-\overline{m}_{0}) \,\mathrm{d}x-\int_{\Gamma}a_{\Gamma}\psi^{L}(\psi^{L}-\overline{m}_{0}) \,\mathrm{d}S\notag\\
   	&\qquad-\int_{\Omega}\pi(\varphi^{L})(\varphi^{L}-\overline{m}_{0}) \,\mathrm{d}x-\int_{\Gamma}\pi_{\Gamma}(\psi^{L})(\psi^{L}-\overline{m}_{0}) \,\mathrm{d}S.\label{Tuni15}
   \end{align}
   For the second line of \eqref{Tuni15}, from \eqref{Tuni4}, we get
   \begin{align}
   &\int_{\Omega}\mu^{L}(\varphi^{L}-\overline{m}_{0})\,\mathrm{d}x +\int_{\Gamma}\theta^{L}(\psi^{L}-\overline{m}_{0})\,\mathrm{d}S\notag\\
   &\quad=\int_{\Omega}(\mu^{L}-\overline{m}(\boldsymbol{\mu}^{L})) \varphi^{L}\,\mathrm{d}x+\int_{\Gamma}(\theta^{L}-\overline{m} (\boldsymbol{\mu}^{L}))\psi^{L}\,\mathrm{d}S\notag\\
   &\quad\leq\Vert\mu^{L}-\overline{m}(\boldsymbol{\mu}^{L})\Vert_{H} \Vert\varphi^{L}\Vert_{H}+\Vert\theta^{L}-\overline{m}(\boldsymbol{\mu}^{L}) \Vert_{H_{\Gamma}}\Vert\psi^{L}\Vert_{H_{\Gamma}}\notag\\
   &\quad\leq C\Vert\boldsymbol{\varphi}^{L}\Vert_{\mathcal{L}^{2}} \Big(\Vert\nabla\mu^{L}(t)\Vert_{H}+\Vert\nabla_{\Gamma}\theta^{L}(t) \Vert_{H_{\Gamma}}+\frac{1}{\sqrt{L}}\Vert\theta^{L}(t)-\mu^{L}(t) \Vert_{H_{\Gamma}}\Big)\notag\\
   &\quad\leq\frac{C}{\sqrt{\tau}},\quad\text{for a.a. }t\geq\tau.\label{Tuni16}
   \end{align}
   For the last three lines of \eqref{Tuni15}, thanks to $\mathbf{(A1)}$, $\mathbf{(A3)}$, $\mathbf{(A4)}$, it holds
   \begin{align*}
   &\int_{\Omega}(J\ast\varphi^{L})(\varphi^{L}-\overline{m}_{0})\,\mathrm{d}x +\int_{\Gamma}(K\circledast\psi^{L})(\psi^{L}-\overline{m}_{0})\,\mathrm{d}S \notag\\
   	&\qquad\leq\Vert J\Vert_{L^{1}(\Omega)}\Vert\varphi^{L}\Vert_{H}\Vert\varphi^{L} -\overline{m}_{0}\Vert_{H}+\Vert K\Vert_{L^{1}(\Gamma)}\Vert\psi^{L}\Vert_{H_{\Gamma}}\Vert\psi^{L} -\overline{m}_{0}\Vert_{H_{\Gamma}},\notag
   \end{align*}
   \begin{align*}
   	&-\int_{\Omega}a_{\Omega}\varphi^{L}(\varphi^{L}-\overline{m}_{0}) \,\mathrm{d}x-\int_{\Gamma}a_{\Gamma}\psi^{L}(\psi^{L} -\overline{m}_{0})\,\mathrm{d}S\notag\\
   	&\qquad\leq a^{\ast}\Vert\varphi^L\Vert_{H}\Vert\varphi^{L}-\overline{m}_{0}\Vert_{H} +a^{\circledast}\Vert\psi^{L}\Vert_{H_{\Gamma}}\Vert\psi^{L} -\overline{m}_{0}\Vert_{H_{\Gamma}},\notag
    \end{align*}
   \begin{align*}
   &-\int_{\Omega}\pi(\varphi^{L})(\varphi^{L}-\overline{m}_{0})\,\mathrm{d}x -\int_{\Gamma}\pi_{\Gamma}(\psi^{L})(\psi^{L}-\overline{m}_{0})\,\mathrm{d}S \notag\\
   	&\qquad\leq\gamma_{1}\Vert\varphi^{L}\Vert_{H}\Vert\varphi^{L} -\overline{m}_{0}\Vert_{H}+\gamma_{2}\Vert\psi^{L}\Vert_{H_{\Gamma}} \Vert\psi^{L}-\overline{m}_{0}\Vert_{H_{\Gamma}}\notag\\[1mm]
   	&\qquad\quad+|\pi(0)||\Omega|^{1/2}\Vert\varphi^{L}-\overline{m}_{0} \Vert_{H}+|\pi_{\Gamma}(0)||\Gamma|^{1/2}\Vert\psi^{L}-\overline{m}_{0} \Vert_{H_{\Gamma}}.
   \end{align*}
   The above estimates together with \eqref{Tuni1}, \eqref{Tuni15}, and \eqref{Tuni16} imply that
   \begin{align}
   	\Vert \beta(\varphi^{L}(t))\Vert_{L^{1}(\Omega)}+\Vert \beta_{\Gamma}(\psi^{L}(t))\Vert_{L^{1}(\Gamma)}\leq\frac{C}{\sqrt{\tau}},\quad\text{for a.a. }t\geq\tau.\label{Tuni17}
   \end{align}
   Then, taking $\zeta=1$ in \eqref{weak3} and $\zeta_{\Gamma}=1$ in \eqref{weak4}, by $\mathbf{(A1)}$, $\mathbf{(A3)}$, $\mathbf{(A4)}$ and \eqref{Tuni17}, we obtain
   \begin{align}
   	\Big|\int_{\Omega}\mu^{L}(t)\,\mathrm{d}x\Big|+\Big|\int_{\Gamma}\theta^{L}(t)\,\mathrm{d}S\Big|\leq\frac{C}{\sqrt{\tau}},\quad\text{for a.a. }t\geq\tau,\notag
   \end{align}
   which, together with \eqref{Tuni4} and generalized Poincar\'e's inequality \eqref{Po3} yields
   \begin{align}
   	\Vert\boldsymbol{\mu}^{L}\Vert_{\mathcal{H}^{1}} \leq\frac{C}{\sqrt{\tau}},\quad\text{for a.a. }t\geq\tau.\notag
   \end{align}
   Hence, we obtain the first estimate in \eqref{Tuni13}.
   The second estimate in \eqref{Tuni13} can be derived by using \eqref{weak1}, \eqref{weak2}, \eqref{Tuni3} and the elliptic regularity theory.
   Finally, by comparison in \eqref{weak3}, \eqref{weak4}, we can conclude \eqref{Tuni14}.
   \end{proof}





   \begin{lemma}
   	\label{Tlemma5}
   	Let $\sigma\in\{0,+\infty\}$. Suppose that  $(\boldsymbol{\varphi}^{\sigma},\boldsymbol{\mu}^{\sigma})$ is the unique global weak solution
   	to problem \eqref{model1}--\eqref{psiini} obtained in Theorem \ref{exist} corresponding to $\sigma\in\{0,+\infty\}$, respectively.
   	Then, for any $\tau>0$, there exists a constant $C>0$ independent of $\tau$, such that
   	\begin{align}
   	&\|\partial_{t}\boldsymbol{\varphi}^{\sigma} \|_{L^{2}_{\mathrm{uloc}}(\tau,+\infty;\mathcal{L}^{2})} +\|\boldsymbol{\varphi}^{\sigma}\|_{L^{\infty}(\tau,+\infty;\mathcal{H}^{1})} +\|\boldsymbol{\mu}^{\sigma}\|_{L^{\infty}(\tau,+\infty;\mathcal{H}^{1})} \notag\\
   	&\qquad +\|\boldsymbol{\beta}(\boldsymbol{\varphi}^{\sigma}) \|_{L^{\infty}(\tau,+\infty;\mathcal{H}^{1})} +\|\boldsymbol{\mu}^{\sigma}\|_{L^{2}_{\mathrm{uloc}} (\tau,+\infty;\mathcal{H}^{2})}
   \leq \frac{C}{\sqrt{\tau}}, \label{0global}
   	\end{align}
   where $\boldsymbol{\beta}(\boldsymbol{\varphi}^{\sigma}):=(\beta(\varphi^\sigma), \beta_\Gamma(\psi^\sigma))$.
    \end{lemma}
    \begin{proof}
    For the case $\sigma=+\infty$, since the two subsystems are completely decoupled,
    we can derive the estimates for the bulk and boundary variables  separately.
    The subsystem in the bulk has been treated in \cite{GGG}, and similar estimates for the subsystem on the boundary can be obtained in a similar way.

    Next, let us consider the case $\sigma=0$. In this aspect, we need to derive uniform estimates with respect to
    $(\varepsilon,L)\in(0,\varepsilon^\ast)\times(0,1)$ for the approximating solutions $(\boldsymbol{\varphi}_{\varepsilon}^{L},\boldsymbol{\mu}_{\varepsilon}^{L})$.
    Recalling the proof for the case with $L\in(0,+\infty)$, we find there are two estimates that may depend on $L\in(0,1)$.
    The first one is \eqref{quotient6}, where we have used the equivalent norm in $\mathcal{H}_{L,0}^{-1}$.
    On one hand, as we did in \eqref{0rate5}, it holds
    \begin{align}
    	&\langle\partial_{t}^{h}\boldsymbol{\varphi}^L_{\varepsilon}, (J\ast\partial_{t}^{h}\varphi^L_{\varepsilon},K\circledast\partial_{t}^{h} \psi^L_{\varepsilon})\rangle_{(\mathcal{H}^{1})',\mathcal{H}^{1}}\notag\\[1mm]
    	&\quad=(\partial_{t}^{h}\boldsymbol{\varphi}^L_{\varepsilon}, (J\ast\partial_{t}^{h}\varphi^L_{\varepsilon},K\circledast\partial_{t}^{h} \psi^L_{\varepsilon}))_{\mathcal{L}^2}\notag\\
    &\quad=\int_\Omega(J\ast\partial_{t}^{h}\varphi^L_{\varepsilon}) \partial_{t}^{h}\varphi^L_\varepsilon\,\mathrm{d}x +\int_\Gamma(K\circledast\partial_{t}^{h}\psi^L_{\varepsilon}) \partial_{t}^{h}\psi^L_{\varepsilon}\,\mathrm{d}S \notag\\
    &\quad\leq(\epsilon\|\mathbb{J}\|_{\mathcal{B}(\mathcal{L}^{2})}^{2} +\|\mathbb{J}-\mathbb{J}_{N}\|_{\mathcal{B}(\mathcal{L}^{2})}) \|\partial_{t}^{h}\boldsymbol{\varphi}^L_{\varepsilon}\|_{\mathcal{L}^{2}}^{2} +C(\epsilon)C(N)\|\partial_{t}^{h}\boldsymbol{\varphi}^L_{\varepsilon}\|_{0,0,\ast}^{2}.\label{0817-1}
    \end{align}
    On the other hand, the following estimate holds
    \begin{align*}
    \|z^{\ast}\|_{0,0,\ast}^{2}&=\|\mathfrak{S}^{0}z^{\ast}\|_{\mathcal{H}_{0,0}^1}^{2}\\
    &=(z^{\ast},\mathfrak{S}^{0}z^{\ast})_{\mathcal{L}^{2}}\\
    &=\langle z^{\ast},\mathfrak{S}^{0}z^{\ast}\rangle_{\mathcal{H}_{L,0}^{-1}, \mathcal{H}_{L,0}^{1}}\\
    &\leq\|z^{\ast}\|_{L,0,\ast}\|\mathfrak{S}^{0}z^{\ast}\|_{\mathcal{H}_{L,0}^1}\\
    &=\|z^{\ast}\|_{L,0,\ast}\|\mathfrak{S}^{0}z^{\ast}\|_{\mathcal{H}_{0,0}^1}\\
    &=\|z^{\ast}\|_{L,0,\ast}\|z^{\ast}\|_{0,0,\ast},
    \end{align*}
    which implies that
    \begin{align}
    \|z^{\ast}\|_{0,0,\ast}\leq \|z^{\ast}\|_{L,0,\ast},\quad\forall\,z^{\ast}\in \mathcal{L}_{(0)}^{2},\quad L\in(0,+\infty).\label{compare2}
    \end{align}
    Using  \eqref{0817-1} and \eqref{compare2}, we can deduce from \eqref{quotient5} that
    \begin{align}
    &\frac{1}{2}\frac{\mathrm{d}}{\mathrm{d}t}\|\partial_{t}^{h} \boldsymbol{\varphi}^L_{\varepsilon}\|_{L,0,\ast}^2 +\Big(a_\ast+\frac{\alpha}{1+\alpha}-\gamma_{1}-\epsilon\|\mathbb{J} \|_{\mathcal{B}(\mathcal{L}^{2})}^{2}-\|\mathbb{J}-\mathbb{J}_{N} \|_{\mathcal{B}(\mathcal{L}^{2})}\Big)\|\partial_t^h \varphi_\varepsilon^L\|_{H}^2\notag\\
    &\qquad+\Big(a_\circledast+\frac{\alpha}{1+\alpha}-\gamma_{2} -\epsilon\|\mathbb{J}\|_{\mathcal{B}(\mathcal{L}^{2})}^{2} -\|\mathbb{J}-\mathbb{J}_{N}\|_{\mathcal{B}(\mathcal{L}^{2})}\Big) \|\partial_t^h\psi_\varepsilon^L\|_{H_\Gamma}^2\notag\\[1mm]
    &\quad\leq C(\epsilon)C(N)\|\partial_{t}^{h}\boldsymbol{\varphi}^L_{\varepsilon} \|_{L,0,\ast}^{2}.\label{0817-2}
    \end{align}
    By $\mathbf{(A3)}$, we can first select $N\in\mathbb{Z}^+$ sufficiently large,
    then $\epsilon>0$ sufficiently small, to get
    \begin{align*}
    	&a_\ast+\frac{\alpha}{1+\alpha}-\gamma_{1} -\epsilon\|\mathbb{J}\|_{\mathcal{B}(\mathcal{L}^{2})}^{2} -\|\mathbb{J}-\mathbb{J}_{N}\|_{\mathcal{B}(\mathcal{L}^{2})}>0,\\
    	&a_\circledast+\frac{\alpha}{1+\alpha}-\gamma_{2} -\epsilon\|\mathbb{J}\|_{\mathcal{B}(\mathcal{L}^{2})}^{2} -\|\mathbb{J}-\mathbb{J}_{N}\|_{\mathcal{B}(\mathcal{L}^{2})}>0.
    \end{align*}
    The above facts together with \eqref{0817-2} yield that there exist two positive constants $\widetilde{C}_1$
    and $\widetilde{C}_2$, independent of $L\in(0,1)$, such that
    \begin{align*}
    \frac{\mathrm{d}}{\mathrm{d}t}\Vert\partial_{t}^{h}\boldsymbol{\varphi}_{\varepsilon}\Vert_{L,0,\ast}^{2}+\widetilde{C}_{1}\Vert\partial_{t}^{h}\boldsymbol{\varphi}_{\varepsilon}\Vert_{\mathcal{L}^{2}}^{2}\leq \widetilde{C}_{2}\Vert\partial_{t}^{h}\boldsymbol{\varphi}_{\varepsilon}\Vert_{L,0,\ast}^{2}.
    \end{align*}
    As a consequence, we have obtained refined estimates of \eqref{quotient6} and \eqref{quotient10},
    which are uniform with respect to the parameters $(\varepsilon,L)\in(0,\varepsilon^\ast)\times(0,1)$.
     Next, we observe that the constant in \eqref{Tuni16} may also depend on $L\in(0,1)$. To overcome this, we can derive uniform estimates with respect to $L\in(0,1)$ by the same calculation as we did in the proof of Lemma \ref{L0uni1}.
    Hence, the constants on the right-hand side of \eqref{Tuni3}, \eqref{Tuni1}, \eqref{Tuni2}, \eqref{Tuni4}, \eqref{Tuni11} and \eqref{Tuni14}
    are all independent of $L\in(0,1)$. Then, passing to the limit as $L\to0$,
    we can conclude the first four estimates in \eqref{0global}.
    Based on the weak formulation \eqref{0weak1} for the case $L=0$
    and the first estimate in \eqref{0global},
    we can conclude the last estimate in \eqref{0global} from the elliptic regularity theory.
    \end{proof}

     \textbf{Proof of Theorem \ref{regularize}.}
     The conclusion of Theorem \ref{regularize} follows from Lemmas \ref{Tlemma1}--\ref{Tlemma5}. Moreover, since $\boldsymbol{\beta}(\boldsymbol{\varphi}):=(\beta(\varphi), \beta_\Gamma(\psi))\in L^{\infty}(\tau,+\infty;\mathcal{H}^{1})$, by the Sobolev embedding theorem, we can conclude \eqref{regularize1} and \eqref{regularize2}. In the two dimensional case, it holds
    \begin{align}
    \|\beta(\varphi(t))\|_{L^{p}(\Omega)}\leq C\sqrt{p}\|\beta(\varphi(t))\|_{V}\leq C\sqrt{p},\quad\text{for a.a.\ }t\geq\tau,\quad\forall\,p\geq2,\notag
    \end{align}
    while in both two and three dimension cases, we have
    \begin{align}
    \|\beta_{\Gamma}(\psi(t))\|_{L^{q}(\Gamma)}\leq C\sqrt{q}\|\beta_{\Gamma}(\psi(t))\|_{V_{\Gamma}}\leq C\sqrt{q},\quad\text{for a.a.\ }t\geq\tau,\quad\forall\,q\geq2,\notag
    \end{align}
    where the constant $C>0$ is independent of $p,q\geq2$.
    \hfill $\square$
    \medskip

Now, we proceed to prove Theorem \ref{separation} on the instantaneous strict separation property and additional regularity for chemical potentials. The proof is based on a suitable De Giorgi's iteration scheme (cf. \cite{Gior,Po,GP}).
  \medskip

\textbf{Proof of Theorem \ref{separation} with $d=3$.}
Let $\tau>0$ be fixed. We consider three positive parameters $T$, $\widetilde{\tau}$ and $\delta$ such that $T-3\widetilde{\tau}\geq\tau/2$ and $\delta\in(0,\min\{\delta_{1}/2,\widetilde{\delta}\})$ (cf. $\mathbf{(A9)}$, $\mathbf{(A10)}$). The precise value of $\widetilde{\tau}$ and $\delta$ will be chosen later. We define two sequences
\begin{align*}
\begin{cases}
t_{-1}=T-3\widetilde{\tau},\\[1mm]
t_{n}=t_{n-1}+\frac{\widetilde{\tau}}{2^{n}},
\end{cases}
\quad\text{and}\quad\kappa_{n}=1-\delta-\frac{\delta}{2^{n}},\quad\forall\,n\in\mathbb{N}.
\end{align*}
Notice that
\begin{align*}
t_{-1}<t_{n}<t_{n+1}<T-\widetilde{\tau},\quad\forall\,n\in\mathbb{N}\quad\text{such that }\quad t_{n}\rightarrow t_{-1}+2\widetilde{\tau}=T-\widetilde{\tau}\ \text{ as }n\rightarrow+\infty,
\end{align*}
and
\begin{align*}
	1-2\delta\leq\kappa_{n}<\kappa_{n+1}<1-\delta,\quad\forall\,n\in\mathbb{N}\quad\text{such that}\quad\kappa_{n}\rightarrow1-\delta\ \text{ as }n\rightarrow+\infty.
	\end{align*}
For $n\in\mathbb{N}$, we introduce $\eta_{n}\in C^{1}(\mathbb{R})$ such that
\begin{align*}
\eta_{n}=
\begin{cases}
1,&t\geq t_{n},\\
0,&t\leq t_{n-1},
\end{cases}
\quad\text{and}\quad|\eta_{n}'(t)|\leq\frac{2^{n+1}}{\widetilde{\tau}}.
\end{align*}
Next, we consider the function $\boldsymbol{\varphi}_{n}=(\varphi_{n},\psi_{n})$ for $n\in\mathbb{N}$, with
\begin{align*}
&\varphi_{n}(x,t)=\max\big\{\varphi(x,t)-\kappa_{n},0\big\}=(\varphi-\kappa_{n})^{+},\\
&\psi_{n}(x,t)=\max\big\{\psi(x,t)-\kappa_{n},0\big\}=(\psi-\kappa_{n})^{+}.
\end{align*}
Consequently, we introduce the sets $I_{n}=[t_{n-1},T]$ and for all $t\in I_n$,
\begin{align*}
&A_{n}(t)=\big\{x\in\Omega:\,\varphi(x,t)-\kappa_{n}\geq0\big\},\\
&\widetilde{A}_n(t)=\big\{x\in\Gamma:\,\varphi(x,t)-\kappa_{n}\geq0\big\},\\
&B_{n}(t)=\big\{x\in\Gamma:\,\psi(x,t)-\kappa_{n}\geq0\big\}.
\end{align*}
If $t\in[0,t_{n-1})$, we set $A_{n}(t)=\widetilde{A}_{n}(t)=B_{n}(t)=\emptyset$. We observe that
\begin{align*}
I_{n+1}\subset I_{n}\quad\forall\,n\in\mathbb{N},\quad I_{n}\rightarrow[T-\widetilde{\tau},T]\quad\text{as }n\rightarrow+\infty
\end{align*}
and
\begin{align*}
A_{n+1}(t)\subset A_{n}(t),\quad \widetilde{A}_{n+1}(t)\subset \widetilde{A}_{n}(t),\quad B_{n+1}(t)\subset B_{n}(t),\quad\forall\,n\in\mathbb{N},\quad t\in I_{n+1}.
\end{align*}
The last ingredient is
\begin{align*}
y_{n}=\int_{I_{n}}\Big(\int_{A_{n}(s)}1\,\mathrm{d}x+\int_{\widetilde{A}_{n}(s)}1\,\mathrm{d}S+\int_{B_{n}(s)}1\,\mathrm{d}S\Big)\mathrm{d}s,\quad\forall\,n\in\mathbb{N}.
\end{align*}
Notice that we introduce a term $\int_{I_n}\int_{\widetilde{A}_n(s)}1\,\mathrm{d}S\,\mathrm{d}s$ in the definition of $y_n$, which enables us to deal with the term $\mathcal{K}_3$ in \eqref{Tuni19}.

For any $n\in\mathbb{N}$, we choose $z=\varphi_{n}\eta_{n}^{2}$ in \eqref{weak1} and $z_\Gamma=\psi_n\eta_n^2$ in \eqref{weak2}, integrating over $[t_{n-1},t]$ where $t_{n}\leq t\leq T$, using \eqref{weak3} and \eqref{weak4}, we obtain
	\begin{align}
&\int_{t_{n-1}}^{t}\langle\partial_{t}\varphi,\varphi_{n}\eta_{n}^{2}\rangle_{V',V}\,\mathrm{d}s+\int_{t_{n-1}}^{t}\langle\partial_{t}\psi,\psi_{n}\eta_{n}^{2}\rangle_{V_\Gamma',V_\Gamma}\,\mathrm{d}s\notag\\
&\qquad+\int_{t_{n-1}}^{t}\int_{A_{n}(s)}(a_\Omega+\beta'(\varphi)+\pi'(\varphi))\nabla\varphi\cdot\nabla\varphi_{n}\eta_{n}^{2}\,\mathrm{d}x\,\mathrm{d}s\notag\\
&\qquad+\int_{t_{n-1}}^{t}\int_{B_{n}(s)}(a_\Gamma+\beta_\Gamma'(\psi)+\pi_\Gamma'(\psi))\nabla_\Gamma\psi\cdot\nabla_\Gamma\psi_{n}\eta_{n}^{2}\,\mathrm{d}S\,\mathrm{d}s\notag\\
&\quad=\int_{t_{n-1}}^{t}\int_{A_{n}(s)}(\nabla J\ast\varphi-\varphi\nabla a_\Omega)\cdot\nabla\varphi_{n}\eta_{n}^{2}\,\mathrm{d}x\,\mathrm{d}s\notag\\
&\qquad+\int_{t_{n-1}}^{t}\int_{B_{n}(s)}(\nabla_{\Gamma} K\circledast\psi-\psi\nabla_{\Gamma}a_\Gamma)\cdot\nabla_{\Gamma}\psi_{n}\eta_{n}^{2}\,\mathrm{d}S\,\mathrm{d}s\notag\\
&\qquad+\frac{1}{L}\int_{t_{n-1}}^{t}\int_\Gamma(\theta-\mu)(\varphi_n-\psi_n)\eta_n^2\,\mathrm{d}S\,\mathrm{d}s\notag\\
&\quad=\mathcal{K}_{1}+\mathcal{K}_{2}+\mathcal{K}_3.\label{Tuni19}
\end{align}
Integrating by parts with respect to time, the first two terms on the left-hand side of \eqref{Tuni19} can be rewritten as
\begin{align}
\int_{t_{n-1}}^{t}\langle\partial_{t}\varphi,\varphi_{n}\eta_{n}^{2}\rangle_{V',V}\,\mathrm{d}s&=\frac{1}{2}\Vert\varphi_{n}(t)\Vert_{H}^{2}-\int_{t_{n-1}}^{t}\Vert\varphi_{n}(s)\Vert_{H}^{2}\eta_{n}\eta_{n}'\,\mathrm{d}s,\label{Tuni20.1}\\
\int_{t_{n-1}}^{t}\langle\partial_{t}\psi,\psi_{n}\eta_{n}^{2}\rangle_{V_\Gamma',V_\Gamma}\,\mathrm{d}s&=\frac{1}{2}\|\psi_n(t)\|_{H_\Gamma}^2-\int_{t_{n-1}}^t\|\psi_n(s)\|_{H_{\Gamma}}^2\eta_n\eta_n'\,\mathrm{d}s.\label{Tuni20.2}
\end{align}
Recall  $|\eta'(t)|\leq\frac{2^{n+1}}{\widetilde{\tau}}$ and the observation (cf. \cite{Po})
\begin{align*}
&0\leq\varphi_{n}\leq2\delta\quad\text{a.e. in }\overline{\Omega},\quad\forall\,t\in[T-2\widetilde{\tau},T],\\
&0\leq\psi_{n}\leq2\delta\quad\text{a.e. on }\Gamma,\quad\forall\,t\in[T-2\widetilde{\tau},T],
\end{align*}
we have
\begin{align}
&\int_{t_{n-1}}^{t}\Vert\varphi_{n}(s)\Vert_{H}^{2}\eta_{n}\eta_{n}'\,\mathrm{d}s+\int_{t_{n-1}}^t\|\psi_n(s)\|_{H_{\Gamma}}^2\eta_n\eta_n'\,\mathrm{d}s\notag\\
&\quad\leq\frac{2^{n+1}}{\widetilde{\tau}}\int_{I_{n}}\int_{A_{n}(s)}(2\delta)^{2}\,\mathrm{d}x\,\mathrm{d}s+\frac{2^{n+1}}{\widetilde{\tau}}\int_{I_{n}}\int_{B_{n}(s)}(2\delta)^{2}\,\mathrm{d}S\,\mathrm{d}s\notag\\
&\quad\leq\frac{2^{n+3}}{\widetilde{\tau}}\delta^{2}y_{n}.\label{Tuni20-1}
\end{align}
By $\mathbf{(A9)}$, $\mathbf{(A10)}$ and $A_{n}(t)\subset A_{0}(t)$, choosing $\delta>0$ sufficiently small, for $t\geq t_{n}$, there holds
\begin{align}
&\int_{t_{n-1}}^{t}\int_{A_{n}(s)}(a_\Omega+\beta'(\varphi)+\pi'(\varphi))\nabla\varphi\cdot\nabla\varphi_{n}\eta_{n}^{2}\,\mathrm{d}x\,\mathrm{d}s\geq \frac{\beta'(1-2\delta)}{2}\int_{t_{n-1}}^{t}\Vert\nabla\varphi_{n}\Vert_{H}^{2}\eta_{n}^{2}\,\mathrm{d}s,\label{Tuni21-1}\\
&\int_{t_{n-1}}^{t}\int_{B_{n}(s)}(a_\Gamma+\beta_\Gamma'(\psi)+\pi_\Gamma'(\psi))\nabla_\Gamma\psi\cdot\nabla_\Gamma\psi_{n}\eta_{n}^{2}\,\mathrm{d}S\,\mathrm{d}s\geq \frac{\beta_\Gamma'(1-2\delta)}{2}\int_{t_{n-1}}^{t}\Vert\nabla_\Gamma\psi_{n}\Vert_{H_\Gamma}^{2}\eta_{n}^{2}\,\mathrm{d}s.\label{Tuni21-2}
\end{align}
For the term $\mathcal{K}_{1}$, using H\"older's inequality and $\mathbf{(A1)}$, we obtain
\begin{align}
|\mathcal{K}_{1}|&=\Big|\int_{t_{n-1}}^{t}\int_{A_{n}(s)}(\nabla J\ast\varphi-\varphi\nabla a_\Omega)\eta_{n}\cdot\nabla\varphi_{n}\eta_{n}\,\mathrm{d}x\,\mathrm{d}s\Big|\notag\\
&\leq\frac{1}{4}\beta'(1-2\delta)\int_{t_{n-1}}^{t}\Vert\nabla\varphi_{n}\Vert_{H}^{2}\eta_{n}^{2}\,\mathrm{d}s+\frac{1}{\beta'(1-2\delta)}\int_{t_{n-1}}^{t}\int_{A_{n}(s)}|\nabla J\ast\varphi-\varphi \nabla a_\Omega|^{2}\eta_{n}^{2}\,\mathrm{d}x\,\mathrm{d}s\notag\\
&\leq\frac{1}{4}\beta'(1-2\delta)\int_{t_{n-1}}^{t}\Vert\nabla\varphi_{n}\Vert_{H}^{2}\eta_{n}^{2}\,\mathrm{d}s+\frac{1}{\beta'(1-2\delta)}\int_{t_{n-1}}^{t}\Vert\nabla J\ast\varphi-\varphi\nabla a_\Omega\Vert^{2}_{L^{\infty}(\Omega)}\int_{A_{n}(s)}1\,\mathrm{d}x\,\mathrm{d}s\notag\\
&\leq\frac{1}{4}\beta'(1-2\delta)\int_{t_{n-1}}^{t}\Vert\nabla\varphi_{n}\Vert_{H}^{2}\eta_{n}^{2}\,\mathrm{d}s+\frac{4(b^{\ast})^{2}}{\beta'(1-2\delta)}\int_{t_{n-1}}^{t}\int_{A_{n}(s)}1\,\mathrm{d}x\,\mathrm{d}s\notag\\
&\leq\frac{1}{4}\beta'(1-2\delta)\int_{t_{n-1}}^{t}\Vert\nabla\varphi_{n}\Vert_{H}^{2}\eta_{n}^{2}\,\mathrm{d}s+\frac{4(b^{\ast})^{2}}{\beta'(1-2\delta)}y_{n}.\label{Tuni22}
\end{align}
Similarly,
\begin{align}
|\mathcal{K}_{2}|\leq\frac{1}{4}\beta_\Gamma'(1-2\delta)\int_{t_{n-1}}^{t}\Vert\nabla_\Gamma\psi_{n}\Vert_{H_\Gamma}^{2}\eta_{n}^{2}\,\mathrm{d}s+\frac{4(b^{\circledast})^{2}}{\beta_\Gamma'(1-2\delta)}y_{n}.\label{Tuni23}
\end{align}
For the term $\mathcal{K}_3$, taking the equations of $\mu$ and $\theta$ (cf. \eqref{weak3}, \eqref{weak4}) into account, since $\beta=\beta_\Gamma$ is monotone non-decreasing (cf. $\mathbf{(A2)}$, $\mathbf{(A7)}$), it holds
\begin{align}
	&\int_{t_{n-1}}^{t}\int_\Gamma(\theta-\mu)(\varphi_n-\psi_n)\eta_n^2\,\mathrm{d}S\,\mathrm{d}s\notag\\
	&\quad=\lim_{\varepsilon\to0}\int_{t_{n-1}}^{t}\int_\Gamma(\theta_\varepsilon-\mu_\varepsilon)((\varphi_\varepsilon-\kappa_{n})^+-(\psi_\varepsilon-\kappa_n)^+)\eta_n^2\,\mathrm{d}S\,\mathrm{d}s\notag\\
	&\quad=\underbrace{\lim_{\varepsilon\to0}\int_{t_{n-1}}^{t}\int_\Gamma (\beta_{\Gamma,\varepsilon}(\psi_\varepsilon)-\beta_\varepsilon(\varphi_\varepsilon))((\varphi_\varepsilon-\kappa_{n})^+-(\psi_\varepsilon-\kappa_n)^+)\eta_n^2\,\mathrm{d}S\,\mathrm{d}s}_{\leq0}\notag\\
	&\qquad+\int_{t_{n-1}}^{t}\int_\Gamma(a_\Gamma\psi-K\circledast\psi+\pi_\Gamma(\psi)-a_\Omega\varphi+J\ast\varphi-\pi(\varphi))(\varphi_n-\psi_n)\eta_n^2\,\mathrm{d}S\,\mathrm{d}s\notag\\
	&\quad\leq \widehat{C}\delta y_n.\label{k3}
\end{align}
The limiting process in \eqref{k3} is essential since $\beta_\varepsilon$ is Lipschitz continuous which infers that $\beta_\varepsilon(\varphi_\varepsilon)|_\Gamma=\beta_\varepsilon(\varphi_\varepsilon|_\Gamma)$ and enables us to use the monotonicity of $\beta$.
Hence, by \eqref{Tuni19}--\eqref{k3}, we obtain
\begin{align}
&\Vert\varphi_{n}(t)\Vert_{H}^{2}+\beta'(1-2\delta)\int_{t_{n-1}}^{t}\Vert\nabla\varphi_{n}\Vert_{H}^{2}\eta_{n}^{2}\,\mathrm{d}s\notag\\
&\qquad+\Vert\psi_{n}(t)\Vert_{H_\Gamma}^{2}+\beta_\Gamma'(1-2\delta)\int_{t_{n-1}}^{t}\Vert\nabla_\Gamma\psi_{n}\Vert_{H_\Gamma}^{2}\eta_{n}^{2}\,\mathrm{d}s\notag\\
&\quad\leq\frac{16(b^{\ast})^{2}}{\beta'(1-2\delta)}y_{n}+\frac{16(b^{\circledast})^{2}}{\beta_\Gamma'(1-2\delta)}y_{n}+\frac{2^{n+5}}{\widetilde{\tau}}\delta^{2}y_{n}+4\widehat{C}\delta y_n,\quad\forall\,t\in[t_{n},T].\label{Tuni24}
\end{align}
In particular, we infer that
\begin{align}
&\max_{t\in I_{n+1}}\Vert\varphi_{n}(t)\Vert_{H}^{2}\leq X_{n},\quad \beta'(1-2\delta)\int_{I_{n+1}}\Vert\nabla\varphi_{n}\Vert_{H}^{2}\,\mathrm{d}s\leq X_{n},\notag\\
&\max_{t\in I_{n+1}}\Vert\psi_{n}(t)\Vert_{H_{\Gamma}}^{2}\leq X_{n},\quad \beta_\Gamma'(1-2\delta)\int_{I_{n+1}}\Vert\nabla_{\Gamma}\psi_{n}\Vert_{H_{\Gamma}}^{2}\,\mathrm{d}s\leq X_{n},\notag
\end{align}
with
\begin{align*}
X_{n}=2^n\max\Big\{\frac{16(b^{\ast})^{2}}{\beta'(1-2\delta)},\frac{16(b^{\circledast})^{2}}{\beta_\Gamma'(1-2\delta)},\frac{2^{5}}{\widetilde{\tau}}\delta^{2},4\widehat{C}\delta\Big\}y_{n}.
\end{align*}
Next, for $t\in I_{n+1}$ and for almost every $x\in A_{n+1}(t)$, following \cite{GGG23}, we find
\begin{align*}
\varphi_{n}(x,t)&=\varphi(x,t)-\Big[1-\delta-\frac{\delta}{2^{n}}\Big]\\
&=\varphi(x,t)-\Big[1-\delta-\frac{\delta}{2^{n+1}}+\frac{\delta}{2^{n+1}}-\frac{\delta}{2^{n}}\Big]\\
&=\varphi(x,t)-\Big[1-\delta-\frac{\delta}{2^{n+1}}\Big]+\delta\Big[\frac{1}{2^{n}}-\frac{1}{2^{n+1}}\Big]\geq\frac{\delta}{2^{n+1}}.
\end{align*}
Analogously, there holds
\begin{align*}
	\varphi_n(x,t)&\geq\frac{\delta}{2^{n+1}},\quad\text{for }t\in I_{n+1}\text{ and a.e. }x\in\widetilde{A}_{n+1}(t),\\
	\psi_n(x,t)&\geq\frac{\delta}{2^{n+1}},\quad\text{for }t\in I_{n+1}\text{ and a.e. }x\in B_{n+1}(t).
\end{align*}
As a consequence, we obtain
\begin{align}
y_{n+1}&=\int_{I_{n+1}}\Big(\int_{A_{n+1}(s)}1\,\mathrm{d}x+\int_{\widetilde{A}_{n+1}(s)}1\,\mathrm{d}S+\int_{B_{n+1}(s)}1\,\mathrm{d}S\Big)\mathrm{d}s\notag\\
&\leq \Big(\frac{2^{n+1}}{\delta}\Big)^{\frac{8}{3}}\int_{I_{n+1}}\int_{A_{n+1}(s)}|\varphi_n|^{\frac{8}{3}}\,\mathrm{d}x\,\mathrm{d}s+\Big(\frac{2^{n+1}}{\delta}\Big)^{\frac{8}{3}}\int_{I_{n+1}}\int_{\widetilde{A}_{n+1}(s)}|\varphi_n|^{\frac{8}{3}}\,\mathrm{d}S\,\mathrm{d}s\notag\\
&\quad+\Big(\frac{2^{n+1}}{\delta}\Big)^{\frac{8}{3}}\int_{I_{n+1}}\int_{B_{n+1}(s)}|\psi_n|^{\frac{8}{3}}\,\mathrm{d}S\,\mathrm{d}s\notag\\
&\leq\Big(\frac{2^{n+1}}{\delta}\Big)^{\frac{8}{3}}\int_{I_{n+1}}\|\varphi_n\|_H^{\frac{5}{3}}\|\varphi\|_{V}\,\mathrm{d}s+C\Big(\frac{2^{n+1}}{\delta}\Big)^{\frac{8}{3}}\int_{I_{n+1}}\|\varphi_n\|_{H^{\frac{3}{4}}(\Omega)}^{\frac{8}{3}}\,\mathrm{d}s\notag\\
&\quad+\Big(\frac{2^{n+1}}{\delta}\Big)^{\frac{8}{3}}\int_{I_{n+1}}\int_{B_{n+1}(s)}|\psi_n|^{\frac{8}{3}}\,\mathrm{d}S\,\mathrm{d}s\notag\\
&\leq\underbrace{\Big(\frac{2^{n+1}}{\delta}\Big)^{\frac{8}{3}}\int_{I_{n+1}}\|\varphi_n\|_H^{\frac{5}{3}}\|\varphi\|_{V}\,\mathrm{d}s+C\Big(\frac{2^{n+1}}{\delta}\Big)^{\frac{8}{3}}\int_{I_{n+1}}\|\varphi_n\|_{H}^{\frac{2}{3}}\|\varphi_n\|_V^2\,\mathrm{d}s}_{\mathcal{K}_4}\notag\\
&\quad+\underbrace{\Big(\frac{2^{n+1}}{\delta}\Big)^{\frac{8}{3}}\int_{I_{n+1}}\int_{B_{n+1}(s)}|\psi_n|^{\frac{8}{3}}\,\mathrm{d}S\,\mathrm{d}s}_{\mathcal{K}_5}.\label{Tuni26}
\end{align}
For the term $\mathcal{K}_4$, recall $0\leq \varphi_n\leq 2\delta$ almost everywhere in $\Omega\times[T-2\widetilde{\tau},T]$, we obtain
\begin{align}
\Big(\frac{\delta}{2^{n+1}}\Big)^{\frac{8}{3}}\mathcal{K}_4&\leq C\int_{I_{n+1}}\|\varphi_n\|_{H}^{\frac{2}{3}}(\|\varphi_n\|_{H}^2+\|\nabla\varphi_n\|_H^2)\,\mathrm{d}s\notag\\
&\leq CX_n^{\frac{1}{3}}\int_{I_{n+1}}\int_{A_n(s)}(2\delta)^2\,\mathrm{d}x\,\mathrm{d}s\notag\\
&\quad+\frac{C}{\beta'(1-2\delta)}\max_{t\in I_{n+1}}\|\varphi_n\|_H^{\frac{2}{3}}\,\beta'(1-2\delta)\int_{I_{n+1}}\|\nabla\varphi_n\|_{H}^2\,\mathrm{d}s\notag\\
&\leq CX_n^{\frac{1}{3}}\delta^2 y_n+\frac{C}{\beta'(1-2\delta)}X_n^{\frac{4}{3}}.\label{k4}
\end{align}
Exploiting interpolation inequality to $\mathcal{K}_5$, it holds
\begin{align}
\Big(\frac{\delta}{2^{n+1}}\Big)^{\frac{8}{3}}\mathcal{K}_5&\leq C \int_{I_{n+1}}\|\psi_n\|_{H_\Gamma}^{\frac{5}{3}}\|\psi_n\|_{V_\Gamma}\,\mathrm{d}s\notag\\
&\leq C \int_{I_{n+1}}\|\psi_n\|_{H_\Gamma}^{\frac{2}{3}}\|\psi_n\|_{V_\Gamma}^2\,\mathrm{d}s\notag\\
& = C \int_{I_{n+1}}\|\psi_n\|_{H_\Gamma}^{\frac{2}{3}}(\|\psi_n\|_{H_\Gamma}^2+\|\nabla_{\Gamma}\psi_n\|_{H_\Gamma}^2)\,\mathrm{d}s\notag\\
&\leq CX_n^{\frac{1}{3}}\int_{I_{n+1}}\int_{B_{n}(s)}(2\delta)^2\,\mathrm{d}S\,\mathrm{d}s\notag\\
&\quad+\frac{C}{\beta_\Gamma'(1-2\delta)}\max_{t\in I_{n+1}}\|\psi_n\|_{H_\Gamma}^{\frac{2}{3}}\,\beta_\Gamma'(1-2\delta)\int_{I_{n+1}}\|\nabla_{\Gamma}\psi_n\|_{H_\Gamma}^2\,\mathrm{d}s\notag\\
&\leq C X_n^{\frac{1}{3}}\delta^2 y_n+\frac{C}{\beta_\Gamma'(1-2\delta)}X_n^{\frac{4}{3}}\notag\\
&= C X_n^{\frac{1}{3}}\delta^2 y_n+\frac{C}{\beta'(1-2\delta)}X_n^{\frac{4}{3}}.\label{Tuni28}
\end{align}
Then, collecting \eqref{Tuni26}, \eqref{k4} and \eqref{Tuni28}, we deduce that
\begin{align}
y_{n+1}&\leq C\Big(\frac{2^{n+1}}{\delta}\Big)^{\frac{8}{3}}\Big(X_n^{\frac{1}{3}}\delta^2 y_n+\frac{1}{\beta'(1-2\delta)}X_n^{\frac{4}{3}}\Big).\label{Tuni29}
 \end{align}
By the definition of $X_n$, we can conclude
\begin{align}
X_n\leq C_\sharp2^n \delta y_n,\quad \text{provided that }\ \widetilde{\tau}\geq\frac{2}{\widetilde{C}_0\min\{(b^\ast)^2+(b^\circledast)^2\}}\delta.\label{restrict1}
\end{align}
Then, by \eqref{Tuni29}, we see that
\begin{align*}
	y_{n+1}\leq \frac{\widetilde{C}_\sharp}{\delta^{\frac{1}{3}}} 16^n y_n^{\frac{4}{3}}.
\end{align*}
Now, we apply Lemma \ref{iteration} for $z_{n}=y_{n}$, $C=\widetilde{C}_{\sharp}\delta^{-\frac{1}{3}}$, $b=16$, $\epsilon=\frac{1}{3}$, to conclude that
\begin{align}
y_{n}\rightarrow0\label{key}
\end{align}
if
\begin{align}
y_{0}\leq \widetilde{C}_\sharp^{-3}16^{-9}\delta.\notag
\end{align}
By the definition of $y_{0}$ and the monotone non-decreasing property of $\beta$, $\beta_\Gamma$, we have
\begin{align*}
y_{0}&=\int_{I_{0}}\Big(\int_{A_{0}(s)}1\,\mathrm{d}x+\int_{\widetilde{A}_{0}(s)}1\,\mathrm{d}S+\int_{B_{0}(s)}1\,\mathrm{d}S\Big)\mathrm{d}s\notag\\
&\leq\int_{I_{0}}\int_{\{x\in\Omega:\,\varphi(x,s)\geq1-2\delta\}}1\,\mathrm{d}x\,\mathrm{d}s+\int_{I_{0}}\int_{\{x\in\Gamma:\,\varphi(x,s)\geq1-2\delta\}}1\,\mathrm{d}S\,\mathrm{d}s\notag\\
&\quad+\int_{I_{0}}\int_{\{x\in\Gamma:\,\psi(x,s)\geq1-2\delta\}}1\,\mathrm{d}S\,\mathrm{d}s\notag\\
&\leq\int_{I_{0}}\int_{A_{0}(s)}\frac{|\beta(\varphi)|}{|\beta(1-2\delta)|}\,\mathrm{d}x\,\mathrm{d}s+\int_{I_{0}}\int_{\widetilde{A}_{0}(s)}\frac{|h_j\circ\beta(\varphi)|}{|\beta(1-2\delta)|}\,\mathrm{d}S\,\mathrm{d}s\notag\\
&\quad+\int_{I_{0}}\int_{B_{0}(s)}\frac{|\beta_\Gamma(\psi)|}{|\beta_\Gamma(1-2\delta)|}\,\mathrm{d}S\,\mathrm{d}s\notag\\
&\leq 3\widetilde{\tau}\|\beta(\varphi)\|_{L^\infty(\frac{\tau}{2},+\infty;V)}\frac{C}{|\text{ln}(\delta)|^\kappa}+3\widetilde{\tau}\|h_j\circ\beta(\varphi)\|_{L^\infty(\frac{\tau}{2},+\infty;V)}\frac{C}{|\text{ln}(\delta)|^\kappa}\notag\\
&\quad+3\widetilde{\tau}\|\beta_\Gamma(\psi)\|_{L^\infty(\frac{\tau}{2},+\infty;V_\Gamma)}\frac{C}{|\text{ln}(\delta)|^\kappa}\notag\\
&\leq 6\widetilde{\tau}\|\beta(\varphi)\|_{L^\infty(\frac{\tau}{2},+\infty;V)}\frac{C}{|\text{ln}(\delta)|^\kappa}+3\widetilde{\tau}\|\beta_\Gamma(\psi)\|_{L^\infty(\frac{\tau}{2},+\infty;V_\Gamma)}\frac{C}{|\text{ln}(\delta)|^\kappa}\notag\\
&\leq C(\tau,E(\boldsymbol{\varphi}_0))\frac{\widetilde{\tau}}{|\text{ln}(\delta)|^\kappa},
\end{align*}
for sufficiently large $j\in\mathbb{N}$ such that $j>|\beta(1-2\delta)|$, where the function $h_j$ is defined by
\begin{align}
	h_j:\ \mathbb{R}\to\mathbb{R},\quad h_j(s)=
	\begin{cases}
		j,&\text{if }s\geq j,\\
		s,&\text{if }s\in(-j,j),\\
		-j,&\text{if }s\leq -j,
	\end{cases}
	\quad\forall\,j\in\mathbb{N}.
\end{align}
Note that the truncation function $h_j$ enables us to have the relation $(h_j\circ\beta(\varphi))|_\Gamma=h_j\circ\beta(\varphi|_\Gamma)$.
Thus, we impose that
\begin{align}
	C(\tau,E(\boldsymbol{\varphi}_0))\frac{\widetilde{\tau}}{|\text{ln}(\delta)|^\kappa}\leq\widetilde{C}_\sharp^{-3}16^{-9}\delta.\label{restrict}
\end{align}
In light of \eqref{restrict1} and \eqref{restrict}, we choose $\delta$ sufficiently small such that $\widetilde{\tau}$ satisfies the relation
\begin{align}
\frac{4}{C_F\min\{(b^\ast)^2+(b^\circledast)^2\}}\delta\leq \widetilde{\tau}\leq 	\frac{\widetilde{C}_\sharp^{-3}16^{-9}}{	C(\tau,E(\boldsymbol{\varphi}_0))}\delta|\text{ln}(\delta)|^\kappa.\label{241202-1}
\end{align}
Now, set $T=\tau+\frac{\widetilde{\tau}}{2}$, reduce $\delta$ to get $\widetilde{\tau}$ even smaller, we clearly have $\tau-\frac{5\widetilde{\tau}}{2}\geq\frac{\tau}{2}$. Hence, we can conclude \eqref{key}, which implies that
\begin{align*}
&\Vert(\varphi-(1-\delta))^{+}\Vert_{L^{\infty}(\Omega\times(T-\widetilde{\tau},T))}=0,\\[1mm]
&\Vert(\psi-(1-\delta))^{+}\Vert_{L^{\infty}(\Gamma\times(T-\widetilde{\tau},T))}=0.
\end{align*}
We can repeat the same argument, with the same $T$ and $\widetilde{\tau}$ fixed, for the case $(\varphi-(-1+\delta))^{-}$ (resp. $(\psi-(-1+\delta))^{-}$), using $\varphi_{n}(t)=(\varphi(t)+k_{n})^{-}$ (resp. $\psi_{n}(t)=(\psi(t)+k_{n})^{-}$). The argument remains exactly the same due to assumptions $\mathbf{(A7)}$--$\mathbf{(A10)}$. We choose the minimum of the $\delta$ obtained in the two cases, ensuring that:
\begin{align*}
&-1+\delta\leq\varphi(x,t)\leq1-\delta,\quad\text{a.e. in }\Omega\times(T-\widetilde{\tau},T),\\[1mm]
&-1+\delta\leq\psi(x,t)\leq1-\delta,\quad\text{a.e. on }\Gamma\times(T-\widetilde{\tau},T).
\end{align*}
In conclusion, since $T-2\widetilde{\tau}=\tau-\frac{3}{2}\widetilde{\tau}\geq\frac{\tau}{2}$, we can repeat the same procedure on the interval $(T,T+\widetilde{\tau})$ (with a new starting time at $t_{-1}=T-2\widetilde{\tau}\geq\frac{\tau}{2}$) and so on, eventually covering the entire interval $[\tau,+\infty)$. Hence, we complete the proof of \eqref{sepa1} and \eqref{sepa2}.

Now, we turn to the proof of the additional regularity results \eqref{mu-high}.
By the definition of $\mu$ in \eqref{weak3} and $\theta$ in \eqref{weak4}, we have
\begin{align*}
	\sup_{t\geq\tau}\|\mu(t)\|_{L^\infty(\Omega)}&\leq \sup_{t\geq\tau}\Big(\|a_\Omega \varphi\|_{L^\infty(\Omega)}+\|J\ast\varphi\|_{L^\infty(\Omega)} +\|\beta(\varphi)\|_{L^\infty(\Omega)}+\|\pi(\varphi)\|_{L^\infty(\Omega)}\Big)\\
	&\leq a^\ast+\|J\|_{L^1(\mathbb{R}^d)}+\max\big\{|\beta(1-\delta)|,|\beta(-1+\delta) |\big\}+\max_{r\in[-1,1]}|\pi(r)|\\
	&\leq \widehat{C}_0,
\end{align*}
and
\begin{align*}
	\sup_{t\geq\tau}\|\theta(t)\|_{L^\infty(\Gamma)}&\leq \sup_{t\geq\tau}\Big(\|a_\Gamma \psi\|_{L^\infty(\Gamma)}+\|K\circledast\psi\|_{L^\infty(\Gamma)} +\|\beta_\Gamma(\psi)\|_{L^\infty(\Gamma)} +\|\pi_\Gamma(\psi)\|_{L^\infty(\Gamma)}\Big)\\
	&\leq a^\circledast+\|K\|_{W^{2,r}(\mathbb{R}^d)} +\max\big\{|\beta_\Gamma(1-\delta)|,|\beta_\Gamma(-1+\delta)|\big\} +\max_{r\in[-1,1]}|\pi_\Gamma(r)|\\
	&\leq \widehat{C}_0.
\end{align*}
Next, using the idea in \cite{GGGP}, we observe that for any $0<t\leq T-h$, it holds
\begin{align}
	\partial_t^h\mu(\cdot)&=a_\Omega\partial_t^h\varphi(\cdot) -J\ast\partial_t^h\varphi(\cdot)+\partial_t^h\varphi(\cdot) \Big(\int_0^1\beta'(s\varphi(\cdot+h)+(1-s)\varphi(\cdot)) \,\mathrm{d}s\Big)\notag\\
	&\quad+\partial_t^h\varphi(\cdot)\Big(\int_0^1\pi'(s\varphi(\cdot+h) +(1-s)\varphi(\cdot))\,\mathrm{d}s\Big)\notag\\
	\partial_t^h\theta(\cdot)&=a_\Gamma\partial_t^h\psi(\cdot) -K\circledast\partial_t^h\psi(\cdot)+\partial_t^h\psi(\cdot)\Big(\int_0^1 \beta_\Gamma'(s\psi(\cdot+h)+(1-s)\psi(\cdot))\,\mathrm{d}s\Big)\notag\\
	&\quad+\partial_t^h\psi(\cdot)\Big(\int_0^1\pi_\Gamma'(s\psi(\cdot+h) +(1-s)\psi(\cdot))\,\mathrm{d}s\Big).\notag
\end{align}
From \eqref{sepa1}, \eqref{sepa2} and strict separation properties of $\varphi$, $\psi$, we find
\begin{align*} &\|s\varphi(\cdot+h)+(1-s)\varphi(\cdot)\|_{L^\infty(\Omega\times(\tau,+\infty))} \leq 1-\delta,\quad\forall\,s\in(0,1),\\[1mm] &\|s\psi(\cdot+h)+(1-s)\psi(\cdot)\|_{L^\infty(\Gamma\times(\tau,+\infty))} \leq 1-\delta,\quad\forall\,s\in(0,1).
\end{align*}
Then, exploiting the fact $\|\partial_t^h\boldsymbol{\varphi}\|_{L^2(0,T-h;\mathcal{L}^2)} \leq\|\partial_t\boldsymbol{\varphi}\|_{L^2(0,T;\mathcal{L}^2)}$, we infer from $\partial_{t}\boldsymbol{\varphi}\in L^{2}_{\mathrm{uloc}}(\tau,+\infty;\mathcal{L}^{2})$ that $\partial_{t}\boldsymbol{\mu}\in L^{2}_{\mathrm{uloc}}(\tau,+\infty;\mathcal{L}^{2})$ and as a result, \eqref{mu-high} holds. The proof is complete.\hfill$\square$
\bigskip

\textbf{Proof of Theorem \ref{separation} with $d=2$.}
When the spatial dimension is two, recall that the assumption $\mathbf{(A8)}$ holds with $\kappa>1/2$, we first refine some estimates we made for the case where the spatial dimension is three:
\begin{itemize}
\item  The estimates \eqref{Tuni21-1} and \eqref{Tuni21-2} can be refined as
 \begin{align}
 	&\int_{t_{n-1}}^{t}\int_{A_{n}(s)}(a_\Omega+\beta'(\varphi)+\pi'(\varphi))\nabla\varphi\cdot\nabla\varphi_{n}\eta_{n}^{2}\,\mathrm{d}x\,\mathrm{d}s\geq \alpha_{\text{min}}\int_{t_{n-1}}^{t}\Vert\nabla\varphi_{n}\Vert_{H}^{2}\eta_{n}^{2}\,\mathrm{d}s,\label{Tuni21-1-2d}\\
 	&\int_{t_{n-1}}^{t}\int_{B_{n}(s)}(a_\Gamma+\beta_\Gamma'(\psi)+\pi_\Gamma'(\psi))\nabla_\Gamma\psi\cdot\nabla_\Gamma\psi_{n}\eta_{n}^{2}\,\mathrm{d}S\,\mathrm{d}s\geq \alpha_{\text{min}}\int_{t_{n-1}}^{t}\Vert\nabla_\Gamma\psi_{n}\Vert_{H_\Gamma}^{2}\eta_{n}^{2}\,\mathrm{d}s,\label{Tuni21-2-2d}
 \end{align}
 with $\alpha_{\text{min}}:=\min\big\{a_\ast+\alpha-\gamma_{1},a_\circledast+\alpha-\gamma_{2}\big\}.$

 \item The estimates \eqref{Tuni22} and \eqref{Tuni23} can be refined as
 \begin{align}
 &|\mathcal{K}_1|\leq\frac{\alpha_{\text{min}}}{4}\int_{t_{n-1}}^{t}\Vert\nabla\varphi_{n}\Vert_{H}^{2}\eta_{n}^{2}\,\mathrm{d}s+\frac{4(b^{\ast})^{2}}{\alpha_{\text{min}}}y_{n},\label{Tuni22-2d}\\
 &|\mathcal{K}_{2}|\leq\frac{\alpha_{\text{min}}}{4}\int_{t_{n-1}}^{t}\Vert\nabla_\Gamma\psi_{n}\Vert_{H_\Gamma}^{2}\eta_{n}^{2}\,\mathrm{d}s+\frac{4(b^{\circledast})^{2}}{\alpha_{\text{min}}}y_{n}.\label{Tuni23-2d}
 \end{align}
\end{itemize}
By \eqref{k3}, \eqref{Tuni24}, \eqref{Tuni21-1-2d}, \eqref{Tuni21-2-2d}, \eqref{Tuni22-2d} and \eqref{Tuni23-2d}, we obtain
\begin{align}
	&\Vert\varphi_{n}(t)\Vert_{H}^{2}+\alpha_{\text{min}}\int_{t_{n-1}}^{t}\Vert\nabla\varphi_{n}\Vert_{H}^{2}\eta_{n}^{2}\,\mathrm{d}s\notag\\
	&\qquad+\Vert\psi_{n}(t)\Vert_{H_\Gamma}^{2}+\alpha_{\text{min}}\int_{t_{n-1}}^{t}\Vert\nabla_\Gamma\psi_{n}\Vert_{H_\Gamma}^{2}\eta_{n}^{2}\,\mathrm{d}s\notag\\
	&\quad\leq\frac{16(b^{\ast})^{2}}{\alpha_{\text{min}}}y_{n}+\frac{16(b^{\circledast})^{2}}{\alpha_{\text{min}}}y_{n}+\frac{2^{n+5}}{\widetilde{\tau}}\delta^{2}y_{n}+4\widehat{C}\delta y_n,\quad\forall\,t\in[t_{n},T].\notag
\end{align}
In particular, we infer that
\begin{align}
	&\max_{t\in I_{n+1}}\Vert\varphi_{n}(t)\Vert_{H}^{2}\leq Y_{n},\quad \alpha_{\text{min}}\int_{I_{n+1}}\Vert\nabla\varphi_{n}\Vert_{H}^{2}\,\mathrm{d}s\leq Y_{n},\notag\\
	&\max_{t\in I_{n+1}}\Vert\psi_{n}(t)\Vert_{H_{\Gamma}}^{2}\leq Y_{n},\quad \alpha_{\text{min}}\int_{I_{n+1}}\Vert\nabla_{\Gamma}\psi_{n}\Vert_{H_{\Gamma}}^{2}\,\mathrm{d}s\leq Y_{n},\notag
\end{align}
with
\begin{align*}
	Y_{n}=2^n\max\Big\{\frac{16(b^{\ast})^{2}}{\alpha_{\text{min}}},\frac{16(b^{\circledast})^{2}}{\alpha_{\text{min}}},\frac{2^{5}}{\widetilde{\tau}}\delta^{2},4\widehat{C}\delta\Big\}y_{n}=2^n \alpha_{\text{max}}y_{n},
\end{align*}
where $\alpha_{\text{max}}=\max\Big\{16(b^{\ast})^{2}/\alpha_{\text{min}},16(b^{\circledast})^{2}/\alpha_{\text{min}}\big\}$, provided that
\begin{align}
	\frac{2^{5}}{\widetilde{\tau}}\delta^{2}\leq \alpha_{\text{max}},\quad4\widehat{C}\delta\leq\alpha_{\text{max}}.\notag
\end{align}
By a similar calculation as \eqref{Tuni26}--\eqref{Tuni29}, it holds
\begin{align}
	y_{n+1}&\leq C\Big(\frac{2^{n+1}}{\delta}\Big)^{\frac{8}{3}}\Big(Y_n^{\frac{1}{3}}\delta^2 y_n+\frac{1}{\alpha_{\text{min}}}Y_n^{\frac{4}{3}}\Big)\notag\\
	&\leq \widetilde{C}_\ast 16^n\delta^{-\frac{8}{3}}y_n^{\frac{4}{3}},\qquad\forall\, n\in\mathbb{N}.\notag
\end{align}
Now, we apply Lemma \ref{iteration} for $z_{n}=y_{n}$, $C=\widetilde{C}_{\sharp}\delta^{-\frac{8}{3}}$, $b=16$, $\epsilon=\frac{1}{3}$, to conclude that
\begin{align}
	y_{n}\rightarrow0\notag
\end{align}
if
\begin{align}
	y_{0}\leq \widetilde{C}_\sharp^{-3}16^{-9}\delta^8.\label{restrict2-2d}
\end{align}
By the definition of $y_{0}$, we have
\begin{align*}
	y_{0}&=\int_{I_{0}}\Big(\int_{A_{0}(s)}1\,\mathrm{d}x+\int_{\widetilde{A}_{0}(s)}1\,\mathrm{d}S+\int_{B_{0}(s)}1\,\mathrm{d}S\Big)\mathrm{d}s\notag\\
	&\leq\int_{I_{0}}\int_{\{x\in\Omega:\,\varphi(x,s)\geq1-2\delta\}}1\,\mathrm{d}x\,\mathrm{d}s+\int_{I_{0}}\int_{\{x\in\Gamma:\,\varphi(x,s)\geq1-2\delta\}}1\,\mathrm{d}S\,\mathrm{d}s\notag\\
	&\quad+\int_{I_{0}}\int_{\{x\in\Gamma:\,\psi(x,s)\geq1-2\delta\}}1\,\mathrm{d}S\,\mathrm{d}s\notag\\
	&\leq\int_{I_{0}}\int_{A_{0}(s)}\frac{|\beta(\varphi)|^q}{|\beta(1-2\delta)|^q}\,\mathrm{d}x\,\mathrm{d}s+\int_{I_{0}}\int_{\widetilde{A}_{0}(s)}\frac{|h_j\circ\beta(\varphi)|^q}{|\beta(1-2\delta)|^q}\,\mathrm{d}S\,\mathrm{d}s\notag\\
	&\quad+\int_{I_{0}}\int_{B_{0}(s)}\frac{|\beta_\Gamma(\psi)|^q}{|\beta_\Gamma(1-2\delta)|^q}\,\mathrm{d}S\,\mathrm{d}s\notag\\
	&\leq 6\widetilde{\tau}(\sqrt{q})^q\|\beta(\varphi)\|_{L^\infty(\frac{\tau}{2},+\infty;V)}^q\frac{C}{|\text{ln}(\delta)|^\kappa}+3\widetilde{\tau}(\sqrt{q})^q\|\beta_\Gamma(\psi)\|_{L^\infty(\frac{\tau}{2},+\infty;V_\Gamma)}^q\frac{C}{|\text{ln}(\delta)|^{\kappa q}}\notag\\
	&\leq [C(\tau,E(\boldsymbol{\varphi}_0))]^q\frac{\widetilde{\tau}(\sqrt{q})^q}{|\text{ln}(\delta)|^{\kappa q}},
\end{align*}
for sufficiently large $j\in\mathbb{N}$ such that $j>|\beta(1-2\delta)|$. Let us set $\delta=e^{-q}$ with $q\geq2$ sufficiently large. The estimate for $y_0$ becomes:
\[y_0\leq [C(\tau,E(\boldsymbol{\varphi}_0))]^q\frac{\widetilde{\tau}(\sqrt{q})^q}{|\text{ln}(\delta)|^{\kappa q}}\leq[C(\tau,E(\boldsymbol{\varphi}_0))]^q \widetilde{\tau}q^{q(\frac{1}{2}-\kappa)}. \]
To ensure \eqref{restrict2-2d} with $\delta=e^{-q}$, we must assume:
\begin{align}
	[C(\tau,E(\boldsymbol{\varphi}_0))]^q \widetilde{\tau}q^{q(\frac{1}{2}-\kappa)}\leq\widetilde{C}_\sharp^{-3}16^{-9}e^{-8q},\notag
\end{align}
which leads to
\begin{align}
\text{exp}\Big[\Big(\text{ln}(C(\tau,E(\boldsymbol{\varphi}_0)))+8+(\frac{1}{2}-\kappa)\text{ln}(q)\Big)q\Big]\leq \frac{\widetilde{C}_\sharp^{-3}16^{-9}}{\widetilde{\tau}}.\label{condition-2d}
\end{align}
It is important to note that:
\[\text{exp}\Big[\Big(\text{ln}(C(\tau,E(\boldsymbol{\varphi}_0)))+8+(\frac{1}{2}-\kappa)\text{ln}(q)\Big)q\Big]\to 0\quad\text{as }q\to+\infty,\]
since $\kappa>1/2$ when $d=2$. Therefore, it is sufficient to select a sufficiently large $q$ such that \eqref{condition-2d} holds, and thus \eqref{restrict2-2d}. With such choice $q$ sufficiently large (and thus $\delta=e^{-q}$ sufficiently small), by passing to the limit as $n\to+\infty$, we can conclude that
\begin{align*}
	&\Vert(\varphi-(1-\delta))^{+}\Vert_{L^{\infty}(\Omega\times(T-\widetilde{\tau},T))}=0,\\[1mm]
	&\Vert(\psi-(1-\delta))^{+}\Vert_{L^{\infty}(\Gamma\times(T-\widetilde{\tau},T))}=0.
\end{align*}
The following argument is same as above for the case $d=3$. Hence, we complete the proof of Theorem \ref{separation}.
\hfill$\square$
	\appendix
	\section{Useful tools}
	\setcounter{equation}{0}
	\noindent We report some technical lemmas that have been used
	in our analysis. First, we recall the compactness lemma of Aubin--Lions--Simon
	type (see, for instance, \cite{Lions} in the case $q>1$ and \cite{Simon}
	when $q=1 $).
	
	\begin{lemma}
		\label{ALS} Let $X_{0} \overset{c}{\hookrightarrow } X_{1}\subset X_{2}$
		where $X_{j}$ are (real) Banach spaces ($j=0,1,2$). Let $1<p\leq +\infty $, $%
		1\leq q\leq +\infty ~$and $I$ be a bounded subinterval of $\mathbb{R}$. Then,
		the sets
		\begin{equation*}
			\left\{ \varphi \in L^{p}\left( I;X_{0}\right) :\partial _{t}\varphi \in
			L^{q}\left( I;X_{2}\right) \right\} \overset{c}{\hookrightarrow }
			L^{p}\left( I;X_{1}\right),\quad \text{ if }1<p<+\infty,
		\end{equation*}
		and
		\begin{equation*}
			\left\{ \varphi \in L^{p}\left( I;X_{0}\right) :\partial _{t}\varphi \in
			L^{q}\left( I;X_{2}\right) \right\} \overset{c}{\hookrightarrow } C\left(
			I;X_{1}\right),\quad \text{ if }p=+\infty ,\text{ }q>1.
		\end{equation*}
	\end{lemma}
	
	Then, we present a fundamental lemma which is essential for the analysis in this paper (cf. \cite{KS}).
	\begin{lemma}
		\label{lem-subsequent}
		Suppose that $\mathbf{(A1)}$ is satisfied. Then the following holds:
		\begin{description}
			\item[(a)] For all $u\in H$, it holds that $J\ast u\in V$ and there exists a constant $C_{J}>0$ depending only on $n$ and $J$, such that
			\begin{align}
				\Vert J\ast u\Vert_{V}\leq C_{J}\Vert u\Vert_{H}.\notag
			\end{align}
			\item[(b)] For all $v\in H_{\Gamma}$, it holds that $K\circledast v\in V_{\Gamma}$ and there exists a constant $C_{K}>0$ depending only on $n$, $r$ and $K$, such that
			\begin{align}
				\Vert K\circledast v\Vert_{V_{\Gamma}}\leq C_{K}\Vert v\Vert_{H_{\Gamma}}.\notag
			\end{align}
		\end{description}
	\end{lemma}
	
	The following well-known result is crucial to the De Giorgi's iteration scheme.
	\begin{lemma}
	\label{iteration}
	Let $\{z_n\}_{n\in\mathbb{N}}\subset\mathbb{R}^+$ satisfy the relation
	\[z_{n+1}\leq C b^n z_n^{1+\epsilon},\]
	for some $C>0$, $b>1$ and $\epsilon>0$. Assume that $z_0\leq C^{-1/\epsilon}b^{-1/\epsilon^2}$. Then, we have
	\[z_n\leq z_0 b^{-\frac{n}{\epsilon}},\quad\forall\,n\geq1.\]
	In particular, $z_n\to0$ as $n\to+\infty$.
	\end{lemma}

\section*{Declarations}	
	\noindent
\textbf{Conflict of interest.} The authors have no competing interests to declare that are relevant to the content of this article.

	\noindent
\textbf{Fundings.} H. Wu was partially supported by NNSFC Grant No. 12071084.

    \noindent
\textbf{Data availability.} Data sharing not applicable to this article as no datasets were generated or analysed during the current study.

	\noindent
\textbf{Acknowledgments.} H. Wu is a member of the Key Laboratory of Mathematics for Nonlinear Sciences (Fudan University), Ministry of Education of China. This research is partially supported by the Shanghai Center for Mathematical Sciences at Fudan University.

\end{document}